\documentclass[hidelinks, 11pt]{article}

\usepackage[UKenglish]{babel}
\usepackage[utf8]{inputenc}
\usepackage[T1]{fontenc}
\usepackage{erewhon}
\usepackage{amsmath}
\usepackage{mathrsfs}
\usepackage{amsthm}
\usepackage{amsfonts}
\usepackage{amssymb}
\usepackage{mathtools}
\usepackage{dsfont} 

\usepackage{quiver}
\usepackage{adjustbox}

\usepackage[dvipsnames]{xcolor}
\usepackage{hyperref}
\usepackage{fullpage}
\usepackage{microtype}
\usepackage{booktabs}
\usepackage{enumitem} 
\usepackage{comment}
\usepackage{authblk} 
\usepackage{fix-cm}

\tikzset{node distance=2cm, auto}
\tikzcdset{row sep/normal=2.7em,column sep/normal=3.5em}

\allowdisplaybreaks

\setlist[description]{font=\normalfont\bfseries\:\!}

\renewcommand\labelenumi{\normalfont[\arabic{enumi}]}
\renewcommand\theenumi\labelenumi

\hypersetup{
    colorlinks=true,
    linkcolor=gray,
    filecolor=gray,      
    urlcolor=gray,
    citecolor=gray
    }
\numberwithin{equation}{section}

\newtheorem{theorem}{Theorem}[section]
\newtheorem{proposition}[theorem]{Proposition}
\newtheorem{lemma}[theorem]{Lemma}
\newtheorem{corollary}[theorem]{Corollary}

\theoremstyle{definition}
\newtheorem{definition}[theorem]{Definition}
\newtheorem{example}[theorem]{Example}

\newtheorem{remark}[theorem]{Remark}
\newtheorem{convention}[theorem]{Convention}

\newenvironment{content}[1][blank]{\vskip .25cm\noindent{\bfseries\scshape\MakeLowercase{#1}}.\thinspace\noindent}{}
\newenvironment{acknowledgements}{\begin{content}[Acknowledgements]}{\end{content}}

\newenvironment{conventions}{\begin{content}[Conventions]}{\end{content}}

\renewenvironment{abstract}{\begin{content}[Abstract]}{\end{content}}

\newcommand{\Tag}[1]{\tag{$#1$}}

\newcommand{\CategoryFont}[1]{\mathsf{#1}}

\newcommand{\ConnFont}[1]{\mathrm{#1}}

\DeclareMathOperator{\Altsum}{\overset{\ominus}{\sum}}

\renewcommand{\.}{\cdot}
\renewcommand{\,}{,{\dots},}

\newcommand{\0}{\mathbf{0}}
\renewcommand{\=}{\mathrel{\mathop:}=}
\renewcommand{\o}{\circ}

\newcommand{\<}{\langle}
\renewcommand{\>}{\rangle}

\renewcommand{\*}{\ast}
\newcommand{\id}{\mathsf{id}}

\newcommand{\Real}{\mathbb{R}}

\newcommand{\X}{\mathbb{X}}

\newcommand{\q}{\mathbf{q}}
\newcommand{\p}{\mathbf{p}}
\newcommand{\T}{\mathrm T}
\newcommand{\TT}{\mathbb T}
\newcommand{\V}{\ConnFont{V}}
\newcommand{\F}{\ConnFont{F}}

\newcommand{\VF}{\mathsf{VF}}
\newcommand{\VVF}{\mathsf{VVF}}
\newcommand{\FVF}{\mathsf{FVF}}
\newcommand{\HVF}{\mathsf{HVF}}
\newcommand{\DB}{\mathsf{DB}}

\newcommand{\VC}{\mathsf{VC}}
\newcommand{\HC}{\mathsf{HC}}
\newcommand{\C}{\mathsf{EC}}

\newcommand{\Dsply}{\mathscr{D}}
\newcommand{\Sbm}{\mathscr{S}}
\newcommand{\FVC}{\mathsf{fVC}}

\newcommand{\AVC}{\mathsf{AVC}}
\newcommand{\sAVC}{\mathsf{sAVC}}

\newcommand{\U}{\mathsf{U}}

\newcommand{\R}{\ConnFont{R}}
\renewcommand{\H}{\ConnFont{H}}
\newcommand{\K}{\mathsf{K}}

\renewcommand{\d}{\mathsf{d}}

\newcommand{\Curv}{\mathsf{C}}
\newcommand{\pb}{\mathsf{pb}}

\title{Ehresmann connections in tangent categories}
\author{\uppercase{Geoffrey Cruttwell}}
\affil{\normalsize\textit{Mount Allison University, Department of Mathematics and Computer Science}}
\author{\uppercase{Marcello Lanfranchi}}
\affil{\normalsize\textit{Macquarie University, School of Mathematical and Physical Sciences}}
\date{}

\begin{document}

\maketitle

\begin{abstract}
The theory of connections is at the very core of differential geometry. Discovered by Levi-Civita and Christoffel and later studied by Cartan, Koszul, and others, connections appear in their most general form under the name of Ehresmann connections. An Ehresmann connection consists of a splitting of the tangent bundle of a submersion into the vertical sub-bundle and a given horizontal distribution. In this paper, we generalize Ehresmann connection to a categorical setting called \emph{tangent categories}.  

Initially introduced by Rosick\'y in 1984 and later generalized by Cockett and the first author in 2014, tangent categories provide a categorical framework to study geometry that extends well beyond smooth manifolds, including algebraic geometry and non-commutative geometry.

In this paper we introduce and study Ehresmann connections in the context of tangent categories. We give various equivalent formulations in term of full and abstract connections and prove that they generalize Koszul connections. We also define parallel transport and curvature for such connections, and prove the structural equation and the Bianchi identity for the curvature.
\end{abstract}

\begin{acknowledgements}
We would like to thank JS Lemay and Geoff Vooys for the precious discussions on submersions and Richard Garner for advice and suggestions. The first author is supported by an NSERC Discovery Grant, while the second author is supported by the AFOSR under award number FA9550-24-1-0008.
\end{acknowledgements}

\par\noindent\rule{\textwidth}{1pt}


\tableofcontents

\section{Introduction}
\label{section:introduction}
Connections are one of the most important concepts of differential geometry and they are at the very core of modern physics and mathematics. One of the most important applications in physics is found in the geometric formulation of general relativity, where the curvature of the Levi-Civita connection determines the gravitational field~\cite{wald:general-relativity}. Gravity is not the only field that has a geometric interpretation. The field strengh of a gauge field, such as the electromagnetic or the bosonic fields of quantum chromodynamics, are also determined by the curvature form of a principal connection~\cite{hamilton:gauge-theory}.

\par Historically, the study of connections began at the end of the 19th century and in the early 20th century with the work of Christoffel and Levi-Civita. However, it was only with Cartan that it became an independent theory, separate from Riemannian geometry. Cartan's goal to generalize Klein's famous Erlangen program culminated with the discovery of Cartan geometries and Cartan connections. A few years later, Kozsul shed new light on the subject by interpreting connections as differential operators, in opposition to Cartan's definition in terms of Lie algebra-valued differential forms.

\par Later, Ehresmann generalized Cartan's approach to describe both Kozsul's covariant derivatives and Cartan's Lie algebra-valued differential forms. We suggest the reader consult~\cite{marle:history-ehresmann-connections} for a historical review of connections and Ehresmann's contributions to the field.

\par Koszul connections were internalized and studied in the context of tangent categories in~\cite{cockett:connections} by Cockett and the first author. In this paper, we do the same for Ehresmann connections.

\begin{content}[Tangent categories]
Tangent categories offer a categorical framework for differential geometry. Introduced by Rosick\'y in~\cite{rosicky:tangent-cats} and later generalized in~\cite{cockett:tangent-cats}, a tangent category consists of a collection of objects, interpreted as models of geometric spaces, a collection of morphisms, intepreted as models of smooth functions, and an assignment, $\T$, which sends each object $M$ to a new object $\T M$ and each morphism $f\colon M\to N$ to a new morphism $\T f\colon\T M\to\T N$.

\par The object $\T M$ should be regarded as the space of tangent vectors of $M$, known as the \textbf{tangent bundle} of $M$, while $\T f$ is interpreted as the smooth function which sends each tangent vector $u\in\T M$ of $M$ at $x$ to the tangent vector $\d_xf(u)$, where $\d_xf$ denotes the differential of $f$ at $x$. The assignment $\T$ satisfies the equations of a functor, that is, $\T$ preserves identities and composition of morphisms.

\par A tangent category also comes with a family of natural transformations, called the \textbf{structural natural transformations} of the tangent structure. The first of these natural transformations is called the \textbf{projection} $p_M\colon\T M\to M$ and is interpreted as the map that sends each tangent vector to its base point. The projection $p_M$ comes with a section, that we call the \textbf{zero morphism} $z_M\colon M\to\T M$, which corresponds to the map that sends each point $x$ of $M$ to the zero tangent vector $0_x$ at $x$. The zero morphism is the unit of the \textbf{sum morphism} $s_M\colon\T_2M\to\T M$, where $\T_2M$ denotes the pullback
\begin{equation*}
\begin{tikzcd}
{\T_2M} & {\T M} \\
{\T M} & M
\arrow["{\pi_2}", from=1-1, to=1-2]
\arrow["{\pi_1}"', from=1-1, to=2-1]
\arrow["\lrcorner"{anchor=center, pos=0.125}, draw=none, from=1-1, to=2-2]
\arrow["{p_M}", from=1-2, to=2-2]
\arrow["{p_M}"', from=2-1, to=2-2]
\end{tikzcd}
\end{equation*}
of the projection along itself, and is interpreted as the space of pairs of tangent vectors $(u,v)$ having the same base point. The sum $s_M$ is regarded as the map that sends each pair $(u,v)$ to $u+v$.

It is sometimes asked that $\T M$ be equipped with a \textbf{negation}, $n_M\colon\T M\to\T M$, which sends each tangent vector $u$ to $-u$. In that case, we say that the tangent category \textbf{has negatives}.
\par Tangent structure has two more natural transformations, the \textbf{vertical lift}, $l_M\colon\T M\to\T^2M$, where $\T^2M$ is the tangent bundle of the tangent bundle of $M$, that is, $\T^2M=\T\T M$, and the \textbf{canonical flip} $c_M\colon\T^2M\to\T^2M$. The vertical lift is used to construct a morphism
\begin{align*}
&\xi_M\=(l_M\times z_{\T M})\T s_M\colon\T_2M\to\T^2M
\end{align*}
which makes the following diagram
\begin{equation*}
\begin{tikzcd}
{\T_2M} & {\T^2M} \\
M & {\T M}
\arrow["{\xi_M}", from=1-1, to=1-2]
\arrow["{\pi_1p_M}"', from=1-1, to=2-1]
\arrow["{\T p_M}", from=1-2, to=2-2]
\arrow["{z_M}"', from=2-1, to=2-2]
\end{tikzcd}
\end{equation*}
into a pullback. This universal condition is interpreted as follows: each fibre $\T_xM\=p^{-1}(x)$ of the tangent bundle $\T M$ of $M$ is \textit{linear}, that is, the tangent bundle $\T\T_xM$ of $\T_xM$ is isomorphic to the Cartesian product $\T_xM\times\T_xM$. This linearity condition is precisely the condition enjoyed by Euclidean spaces, that are $\Real^n$, for which $\T\Real^n\cong\Real^n\times\Real^n$. This expresses the idea that an object in a tangent category is a \emph{locally linear} space and a morphism is a \emph{locally linear} function.

Finally, the canonical flip, $c_M\colon\T^2M\to\T^2M$, is an isomorphism of the double tangent bundle, which encodes the symmetry of the Hessian matrix, explicitly, that the order in which two sequential partial derivatives $\partial_i$, $\partial_j$ are taken does not matter.

\par The archetypal example is the category of finitely-dimensional smooth manifolds with the usual notion of the tangent bundle functor. However, tangent categories go well beyond smooth manifolds. In~\cite{cruttwell:algebraic-geometry}, the category of affine schemes was shown to carry a tangent structure which captures some key aspects of algebraic geometry. This construction was later extended in~\cite{ikonicoff:operadic-algebras-tangent-cats} to algebras over an arbitrary algebraic operad. Convenient manifolds and $C^\infty$-rings also form tangent categories as well as the subcategory of microlinear objects of a model of Synthetic Differential Geometry (SDG) and every Cartesian Differential Category (CDC)~\cite{cockett:tangent-cats}.

\par Numerous concepts of differential geometry have already been successfully extended to tangent categories. Vector fields, Euclidean spaces, vector bundles, Kozsul connections, and ordinary differential equations are only some of the concepts of differential geometry that have been internalized to tangent categories. In this paper, we add to the list Ehresmann connections.
\end{content}

\begin{content}[Submersions]
In their most general version, Ehresmann connections are defined on submersions. A submersion, in differential geometry, consists of a smooth function $q\colon E\to M$, whose differential $\d_yq$ at each point $y\in E$ has maximal rank, that is, for every tangent vector $v\in\T_xM$ of $M$ at $x=q(y)$, there exists a tangent vector $u\in\T_yE$ of $E$ at $y$ such that $\d_yq(u)=v$.
\par A submersion in a tangent category (Definition~\ref{definition:submersion}) consists of a morphism $q\colon E\to M$, for which (1) the pullback $\F q$ of $q$ along the projection $p_M$ exists, and (2) the unique morphism $\pi_q\colon\T E\to\F q$ which makes the following diagram commutative
\begin{equation*}
\begin{tikzcd}
{\T E} \\
& {\F q} & E \\
& {\T M} & M
\arrow["{\pi_q}", dashed, from=1-1, to=2-2]
\arrow["{p_E}", curve={height=-12pt}, from=1-1, to=2-3]
\arrow["{\T q}"', curve={height=12pt}, from=1-1, to=3-2]
\arrow[from=2-2, to=2-3]
\arrow[from=2-2, to=3-2]
\arrow["\lrcorner"{anchor=center, pos=0.125}, draw=none, from=2-2, to=3-3]
\arrow["q", from=2-3, to=3-3]
\arrow["{p_M}"', from=3-2, to=3-3]
\end{tikzcd}
\end{equation*}
is a regular epimorphism in the appropriate category. The morphism $\pi_q$ is called the \textbf{horizontal descent} of $q$ and $\q^\F\colon\F q\to E$ is the \textbf{Finsler bundle} of $q$ (Definition~\ref{definition:finsler-bundle}). The pullback $\F q$ can be regarded as the space of triples $(x,u,y)$ formed by a point $x$ of $M$, a tangent vector $u\in\T_xM$ of $M$ at $x$, and an element $y\in E$ such that $q(y)=x$; the morphism $\pi_q\colon\T E\to\F q$ sends each tangent vector $v\in\T_yE$ of $E$ at $y$ to the triple $(q(y),\d_yv,y)$. Thus, requiring $\pi_q$ be regular epi is to say that for every triple $(x,u,y)$ of $\F q$ there exists a tangent vector $v\in\T_yE$ of $E$ at $y$ such that $\d_yv=u$, which is precisely to say that $q$ is a submersion. We will review this concept in Section~\ref{subsection:fundamental-ses}.
\end{content}

\begin{content}[Tangent display maps]
Although not immediately apparent from their definition, submersions in differential geometry exhibit good behaviour with respect to two important operations: the tangent bundle functor and pullbacks along arbitrary maps. In particular, the pullback of a submersion along a smooth function exists and is still a submersion. From a tangent categorical perspective, stability under pullbacks plays a crucial role in the study of structures and properties of differential geometry. However, it is also a subtle problem to establish whether or not a certain morphism admits all pullbacks and whether or not those pullbacks are also preserved by the tangent bundle functor. In particular, pullbacks in the category of smooth manifolds are notoriously ill-behaved: one cannot assume their existence in full generality, and even when they exist, one cannot expect them to be preserved by the tangent bundle functor.
\par In previous work~\cite{cruttwell:tangent-display-maps}, we tackled this issue by introducing a new class of maps that we called \textbf{tangent display maps} (Definition~\ref{definition:tangent-display-map}). A tangent display map consists of a morphism $q\colon E\to M$ satisfying the following property. For every $m,n\geq 0$, the pullback of $\T^nq\colon\T^nE\to\T^nM$ along any morphism $f\colon N\to\T^nM$ exists and is preserved by all iterates $\T^m\=\T\o\dots\o\T$ of the tangent bundle functor.
In~\cite[Theorem~2.31]{cruttwell:tangent-display-maps}, we shown that in the category of smooth manifolds, tangent display maps are equivalent to submersions.

However, on the one hand it is clear to us that tangent display maps are far too general to represent a good definition of submersion in a tangent category. On the other hand, submersions in the context of tangent categories do not need to admit pullbacks along arbitrary maps. Thus, we only consider submersions whose underlying map is also tangent display, that is, for us a submersion is a tangent display map $q\colon E\to M$ whose horizontal descent $\pi_q\colon\F q\to\T E$ is regular epi in the appropriate category (Definition~\ref{definition:submersion}).
\end{content}

\begin{content}[The vertical and the Finsler bundle]
One of the key properties of a tangent display map is that it has associated \emph{vertical} and \emph{Finsler} bundles. In differential geometry the vertical bundle of a map $q\colon E\to M$ consists of the vector sub-bundle $\q^\V\colon\V q\to E$ of the tangent bundle $\p_E\colon\T E\to E$ of $E$, whose local fibre $\V_yq$ at a point $y$ of $E$ is the kernel $\ker\d_yq$ of the differential of the map $q$ at $y$.

If you imagine to move within the space $E$, you can also imagine to move \emph{vertically} with respect to the map $q$, that is, moving in a such a way to leave the position projected down to the base space $M$ by $q$, unchanged. The tangent vectors that generate these vertical degrees of freedom are precisely those tangent vectors forming the vertical bundle of $q$.

In a tangent category, the \textbf{vertical bundle} of a tangent display map $q\colon E\to M$ (Definition~\ref{definition:vertical-bundle}) is obtained by pulling back $\T q$ along the zero morphism $z_M\colon M\to\T M$:
\begin{equation*}
\begin{tikzcd}
{\V q} & {\T E} \\
M & {\T M}
\arrow["{\iota_q}", from=1-1, to=1-2]
\arrow["{\T^\V q}"', from=1-1, to=2-1]
\arrow["\lrcorner"{anchor=center, pos=0.125}, draw=none, from=1-1, to=2-2]
\arrow["{\T q}", from=1-2, to=2-2]
\arrow["{z_M}"', from=2-1, to=2-2]
\end{tikzcd}
\end{equation*}
The vertical bundle is then defined as the bundle $\q^\V\colon\V q\to E$ obtained by composing $\iota_q\colon\V q\to\T E$ with $p_E\colon\T E\to E$. We will show that the vertical bundle carries the structure of a differential bundle, the tangent categorical-analogue of a vector bundle.

While the vertical vectors capture the vertical movements within the space $E$, there is no canonical \emph{horizontal} direction, that is, a direction \emph{parallel} to the base space that leaves the vertical component unchanged. An Ehresmann connection provides precisely such notion. In order to define an Ehresmann connection, one needs the Finsler bundle. As mentioned above, the Finsler bundle appears in the definition of a submersion.

In a tangent category, the Finsler bundle of a tangent display map $q\colon E\to M$ (Definition~\ref{definition:finsler-bundle}) is constructed by pulling back $q$ along the projection $p_M\colon\T M\to M$:
\begin{equation*}
\begin{tikzcd}
{\F q} & E \\
{\T M} & M
\arrow["{q^\F}", from=1-1, to=1-2]
\arrow["{\T^\F q}"', from=1-1, to=2-1]
\arrow["\lrcorner"{anchor=center, pos=0.125}, draw=none, from=1-1, to=2-2]
\arrow["q", from=1-2, to=2-2]
\arrow["{p_M}"', from=2-1, to=2-2]
\end{tikzcd}
\end{equation*}
The pullback projection $\q^\F\colon\F q\to E$ comes equipped with the structure of a differential bundle, that we call, the Finsler bundle of $q$. As mentioned earlier, the Finsler bundle should be regarded as the bundle containing pairs $(y,u)$ formed by a point $y$ of $E$ together with a tangent vector $u$ of $M$ at $x=q(y)$.
\end{content}

\begin{content}[Ehresmann connections]
In differential geometry, an Ehresmann connection on a submersion $q\colon E\to M$ consists of a choice of a sub-bundle $\H q\to E$ of $p_E\colon\T E\to E$ called the \textbf{horizontal bundle} which is (1) isomorphic to the Finsler bundle $\q^\F\colon\F q\to E$ and (2) is in direct sum with the vertical bundle $\q^\V\colon\V q\to E$, explicitly, $\H_yq\subseteq\T_yE$ and at each point $y$ of $E$, the tangent space $\T_yE$ is equal to the direct sum $\V_yq\oplus\H_yq$.
\par An Ehresmann connection thus provides a system of coordinates for the total space $E$ of $q\colon E\to M$, which allows one to write every tangent vector $u$ of $E$ uniquely as the sum of a vertical vector $u_\V$, that is, a vector belonging to the vertical bundle $\q^\V\colon\V q\to E$, and a horizontal vector $u_\H$, which is a vector of the Finsler bundle $\q^\F\colon\F q\to E$ once embedded in the tangent bundle $\p_E\colon\T E\to E$.

To understand this notion in an arbitrary tangent category, it is useful to introduce the fundamental short exact sequence of a submersion.
\end{content}

\begin{content}[The fundamental short exact sequence]
Every submersion $q\colon E\to M$ in a tangent category defines a short exact sequence
\begin{align*}
\0\to\q^\V\xrightarrow{\iota_q}\p_E\xrightarrow{\pi_q}\q^\F\to\0
\end{align*}
that we call the \textbf{fundamental short exact sequence} of $q$, where $\q^\V\colon\V q\to E$ is the vertical bundle of $q$ and $\q^\F\colon\F q\to E$ denotes the Finsler bundle.
\par The first major contribution of this paper is to define an \textbf{Ehresmann connection} (Definition~\ref{definition:ehresmann-connection}) on a submersion $q\colon E\to M$ in a tangent category as a \emph{splitting} of the fundamental short exact sequence of $q$ (Theorem~\ref{theorem:splitting-of-ses}). In particular, with an Ehresmann connection, $\T E$ becomes (isomorphic to) the fibred biproduct $\V q\oplus_E\F q$ of the vertical bundle and the Finsler bundle.
\end{content}

\begin{content}[Full and abstract connections]
Our definition of an Ehresmann connection decomposes into two parts: a \textbf{vertical} and a \textbf{horizontal connection} (Definitions~\ref{definition:vertical-connection} and~\ref{definition:horizontal-connection}). However, using the splitting lemma, we show that, in a tangent category with negatives, one determines necessarily the other one (Theorem~\ref{theorem:vertical-horizontal-connections-duality}).

A different approach is taken by full connections. A \textbf{full vertical connection} (Definition~\ref{definition:full-vertical-connection}) is a vertical connection $\R$ subject to an extra axiom \textbf{[FVC]}. This extra axiom suffices to uniquely reconstruct the horizontal connection which completes $\R$ into an Ehresmann connection (Theorem~\ref{theorem:full-connections}).

Similarly, a \textbf{full horizontal connection} (Definition~\ref{definition:full-horizontal-connection}) is a horizontal connection $\H$ subject to an extra axiom \textbf{[FHC]}. Every full horizontal connection admits a unique vertical connection which completes $\H$ into an Ehresmann connection (Theorem~\ref{theorem:full-connections}).

Vertical, horizontal, and Ehresmann connections can also be equivalently formulated in terms of connection forms (Definitions~\ref{definition:vertical-connection-form},~\ref{definition:horizontal-connection-form},~\ref{definition:ehresmann-connection-form}). The correspondences between connections and their forms suggests a new generalization to Ehresmann connections, that we call abstract connections.

An \textbf{abstract vertical connection} is a linear idempotent $\phi$ of the tangent bundle $\p_E\colon\T E\to E$, subject to some properties (Definition~\ref{definition:abstract-vertical-connection}). The novelty in this definition is that it does not rely on the vertical bundle and therefore, it can be defined on maps which are not necessarily tangent display.  However, Theorem~\ref{theorem:abstract-vertical-connection} shows that a splitting of an abstract vertical connection is equivalent to a \emph{choice} of the vertical bundle (the vertical bundle only being defined up to a unique isomorphism) together with a full vertical connection on it.

An \textbf{abstract horizontal connection} is a linear idempotent $\psi$ of the tangent bundle $\p_E\colon\T E\to E$, subject to some properties (Definition~\ref{definition:abstract-horizontal-connection}). Again, this definition does not require the Finsler bundle. However, as with abstract vertical connections, Theorem~\ref{theorem:abstract-horizontal-connection} proves that a splitting of an abstract horizontal connection gives a \emph{choice} of the Finsler bundle together with a full horizontal connection on it.

This proves an equivalence between full, split abstract, and Ehresmann connections (Theorem~\ref{theorem:all-equivalent-forms-of-connections}).
\end{content}

\begin{content}[Koszul connections]
Koszul connections on differential bundles have already been introduced in tangent categories in~\cite{cockett:connections} by Cockett and the first author. Differential bundles are the tangent-categorical analogue of a vector bundle, that is, a fibre bundle $q\colon E\to M$ whose local fibres carry the structure of a vector space.  In fact, MacAdam and Burke proved that differential bundles in the category of finitely-dimensional smooth manifolds are equivalent to vector bundles~\cite{macadam:vector-bundles}.

\par As Theorem~\ref{theorem:koszul-linear-connection} proves, Koszul connections on differential bundles turn out to be a special type of Ehresmann connections: \textbf{linear} Ehresmann connections. We also show an interesting alternative point of view: a Koszul connection on a differential bundle $\q\colon E\to M$ can be seen simply as an Ehresmann connection on $q$ when $q$ is regarded as a morphism in the tangent category of differential bundles (Theorem~\ref{theorem:linear-connections-as-connections-on-dbs}).
\end{content}

\begin{content}[Parallel transport and curvature]
After spending Sections~\ref{section:vector-fields} and~\ref{section:distributions} to define \textbf{vertical}, \textbf{Finsler}, and \textbf{horizontal vector fields} (Definitions~\ref{definition:vertical-vector-field},~\ref{definition:finsler-vector-field},and~\ref{definition:horizontal-vector-field}) and the notions of \textbf{distributions} and \textbf{involutivity} (Definitions~\ref{definition:distribution} and~\ref{definition:involutive-distribution}), we show that Ehresmann connections in a tangent category define \textbf{parallel transport} (assuming the tangent category has a suitable \emph{curve object}) and carry a notion of \textbf{curvature} (Section~\ref{section:curvature}). We end by proving the structural equation for the curvature and the Bianchi identity (Theorems~\ref{theorem:structural-equation} and~\ref{theorem:bianchi-identity}).
\end{content}

\begin{content}[Related work]
In~\cite{lemay:submersions}, Lemay and Vooys introduced some important classes of maps in a tangent category, including submersions. As aknowledged in their paper, we independently discovered the definition of submersion while working on the story of this paper, including the fundamental short exact sequence which, in their paper is called the \emph{relative cotangent sequence}. However, Lemay and Vooys spent much more time in working out a full theory of submersions. We suggest the reader consult their paper for full details on this concept.

Lemay and Vooys also considered \emph{linear} sections of the horizontal descent~\cite[Section~8.3]{lemay:submersions}, which we call, horizontal connections, however, they did not focus of the geometric meaning of these linear sections.

They also decided to not consider tangent display maps as we do, and instead require the existence of tangent pullbacks as needed. In particular, they use the notions of \emph{$0$-carrable} and \emph{$p$-carrable}, which respectively means, that a map admits tangent pullbacks along the zero $z_M\colon M\to\T M$ and the projection $p_M\colon\T M\to M$, respectively. This allows them to define the vertical and the Finsler bundle without the need of other pullbacks. The downside of this choice is the lost of stability under pullbacks: if $q$ is $0$-carrable, this might not be the case for the pullback of $\T q$ along $z_M$.

They also call the Finsler bundle, \emph{horizontal} bundle and use different notations which are closer to the ones adopted in Algebraic geometry. Our naming convention avoids confusion with the horizontal distribution, which is the Finsler bundle once is embedded in the tangent bundle via the horizontal connection.
\end{content}

\begin{content}[Significance]
Internalizing Ehresmann connections in tangent categories is not a mere exercise of generalization of an existing concept. Since 2014 (\cite{cockett:tangent-cats}), there has been a lot of effort to rebuild the familiar constructions of differential geometry within the setting of tangent categories (\cite{cockett:differential-bundles,cockett:differential-equations,macadam:phd-thesis}). Conceptually, this program aims to upgrade tangent categories from an abstraction of the tangent bundle functor to a full theory of geometry, which provides a rich setting to support ``doing geometry'' in its own right, without committing in advance to any particular notion of geometric space.

At the same time, there has been a complementary drive to produce new examples of tangent categories to capture geometric worlds beyond smooth manifolds. Algebraic geometry as well as other geometric theories such as non-commutative geometry took great inspiration from differential geometry. In this spirit, the emergence of tangent categories built from (affine) schemes (\cite{cruttwell:algebraic-geometry}), and later from operadic affine schemes (\cite{ikonicoff:operadic-algebras-tangent-cats}), makes a strong statement: what looks like quite different forms of geometry can all be understood as models of the same tangent-categorical theory. Differential, algebraic, and operadic geometry then appear not as unrelated subjects, but as different models of a shared abstract structure.

This justifies our effort to construct a theory of Ehresmann connections for tangent categories. Instead of being a concept of differential geometry, it now becomes a subtheory of tangent category theory that will reflect across the different models.

A further advantage of working categorically is that it makes the structural dependencies of the constructions particularly transparent: one can see exactly which axioms and which bits of structure are doing the work, and therefore what really needs to be assumed to obtain a given result.

This clarity becomes especially visible in three parts of the paper. First, the paper introduces abstract Ehresmann connections as certain linear idempotents on the tangent bundle, subject to compatibility conditions (Section~\ref{subsection:abstract-connection}). Framed this way, connections are not restricted to the classical context of submersions; they can be defined over arbitrary maps. Within this enlarged notion, the “split” abstract connections are shown to coincide precisely with ordinary connections, so the new definition both extends the classical picture and recovers it under the splitting condition (Theorem~\ref{theorem:all-equivalent-forms-of-connections}).

Second, the paper shows that Koszul connections, which in differential geometry correspond to the usual notion of covariant derivatives, fit naturally into this framework (Theorem~\ref{theorem:linear-connections-as-connections-on-dbs}). More precisely, Koszul connections are Ehresmann connections in the category of differential bundles, in the appropriate sense. Rather than being a parallel theory, Koszul connections become an application, and in a sense a consequence, of the general theory of Ehresmann connections once it is set up at the right level of abstraction.

Third, the paper introduces a categorical notion of curvature (Definition~\ref{definition:curvature}). Here curvature is understood as measuring a particular failure: the connection does not behave as a morphism of connections over itself in the relevant categorical setting. This reframes a central geometric idea as an instance of a precise categorical phenomenon, relating the intuitive content of curvature to an obstruction to a structural condition.
\end{content}

\begin{content}[Upcoming work]
One of the objectives of this project is to provide a foundational workground for a theory of \emph{principal connections} in tangent categories. In differential geometry, a principal connection can be equivalently defined as an Ehresmann connection on a principal bundle that is equivariant with respect to the group action. This will be the subject of an upcoming paper, in which many of the definitions and results introduced and proved in this paper will play a crucial role.
\end{content}

\begin{conventions}
In this paper, we adopt diagrammatic composition, that is, we shall denote by $fg$ the composition of $f\colon A\to B$ followed by $g\colon B\to C$.
\end{conventions}


\section{The vertical and the Finsler bundles of a submersion}
\label{section:tangent-display-map}
In this section, we recall the definition of a tangent display map and show that every tangent display map admits a vertical bundle and a Finsler bundle. We then introduce submersions and construct the associated fundamental exact sequence.

\subsection{Tangent display maps}
\label{subsection:tangent-display-map}
Here we recall the main definitions and results of~\cite{cruttwell:tangent-display-maps}, beginning with the definition of a tangent display map.

\begin{definition}
\label{definition:tangent-display-map}
A \textbf{tangent display map} in a tangent category $(\X,\TT)$ consists of a morphism $q\colon E\to M$ subject to the following property. For every $n\geq 0$ and every $f\colon N\to\T^nM$, the pullback
\begin{equation*}
\begin{tikzcd}
P & {\T^nE} \\
N & {\T^nM}
\arrow[from=1-1, to=1-2]
\arrow[from=1-1, to=2-1]
\arrow["\lrcorner"{anchor=center, pos=0.125}, draw=none, from=1-1, to=2-2]
\arrow["{\T^nq}", from=1-2, to=2-2]
\arrow["f"', from=2-1, to=2-2]
\end{tikzcd}
\end{equation*}
of $\T^nq$ along $f$ exists and is a \textbf{tangent pullback}, that is, it is preserved by all iterates $\T^m\=T\o\dots\o\T$ of the tangent bundle functor.
\end{definition}

\begin{example}
\label{example:trivial-display}
Every category $\X$ comes equipped with a trivial tangent structure, whose tangent bundle functor and structural natural transformations are the identity. In such a tangent category, a tangent display map is equivalent to a display map, namely, a map $q\colon E\to M$ for which, for any $f\colon N\to M$, the pullback of $q$ along $f$ exists.
\end{example}

\begin{example}
\label{example:differential-geometry-display}
As proved in~\cite[Theorem~2.31]{cruttwell:tangent-display-maps}, tangent display maps in the tangent category of finitely-dimensional smooth manifolds are equivalent to submersions.
\end{example}

\begin{example}
\label{example:algebraic-geometry-display}
In the tangent category of affine schemes (see~\cite{cruttwell:algebraic-geometry}), every morphism is a tangent display map, since the category of affine schemes admits all pullbacks and the tangent bundle functor preserves those.
\end{example}

The main benefit of working with tangent display maps is that they form the maximal \textbf{tangent display system}, which is a family of morphisms $\Dsply$ stable under tangent pullbacks and the tangent bundle functor.

\begin{theorem}[{\cite[Theorem~2.11]{cruttwell:tangent-display-maps}}]
\label{theorem:maximal-tangent-display-system}
Tangent display maps in a tangent category $(\X,\TT)$ form the maximal tangent display system $\Dsply(\X,\TT)$. Furthermore, $\Dsply(\X,\TT)$ is closed under composition.
\end{theorem}

Since tangent display maps are stable under the tangent bundle functor, they also form a tangent sub-category $\Dsply(\X,\TT)$ of the arrow tangent category. Concretely, the objects of $\Dsply(\X,\TT)$ are tangent display maps and morphisms, commutative squares between the tangent display maps.

Thanks to Theorem~\ref{theorem:maximal-tangent-display-system}, the (tangent) pullback of a tangent display map along a morphism is again a tangent display map. Since we will be working extensively with tangent pullbacks, we recall here a simple yet important property of tangent pullbacks which generalizes a well-known property of pullbacks.

\begin{lemma}[{\cite[Lemma~2.3]{cruttwell:tangent-display-maps}}]
\label{lemma:property-tangent-pullbacks}
Consider the following two commutative diagrams in any tangent category:
\begin{equation*}
\begin{tikzcd}
A & B & C \\
{A'} & {B'} & {C'}
\arrow[from=1-1, to=1-2]
\arrow[from=1-1, to=2-1]
\arrow[from=1-2, to=1-3]
\arrow[from=1-2, to=2-2]
\arrow[from=1-3, to=2-3]
\arrow[from=2-1, to=2-2]
\arrow[from=2-2, to=2-3]
\end{tikzcd}\quad
\begin{tikzcd}
A & B \\
{A'} & {B'} \\
{A''} & {B''}
\arrow[from=1-1, to=1-2]
\arrow[from=1-1, to=2-1]
\arrow[from=1-2, to=2-2]
\arrow[from=2-1, to=2-2]
\arrow[from=2-1, to=3-1]
\arrow[from=2-2, to=3-2]
\arrow[from=3-1, to=3-2]
\end{tikzcd}
\end{equation*}
If the right and the outer squares of the first diagram are tangent pullbacks, so is the left square. Moreover, if the bottom and the outer squares of the second diagram are tangent pullbacks, so is the top square.
\end{lemma}

Next, we recall a key notion in tangent category theory: differential bundles, which generalize (smooth) vector bundles. 

\begin{definition}[{\cite[Definition~2.3]{cockett:differential-bundles}}]
\label{definition:differential-bundle}
A \textbf{differential bundle} in a tangent category $(\X,\TT)$ consists of a pair of objects $E$ and $M$ together with the following structural morphisms
\begin{align*}
&q\colon E\to M         &&z_q\colon M\to E      &&s_q\colon E_2\to E    &&l_q\colon E\to\T E
\end{align*}
respectively called, the \textbf{projection}, the \textbf{zero}, the \textbf{sum} morphism, and the \textbf{vertical lift}, where $E_n$ denotes the $n$-fold pullback of $q$ along itself, which is required to exist and be preserved by all iterates of the tangent bundle functor. The triple $(q,z_q,s_q)$ forms an \textbf{additive bundle}, that is, the following conditions hold:
\begin{equation*}
\adjustbox{width=\linewidth}{
\begin{tikzcd}
{E_2} & E \\
E & M
\arrow["{s_q}", from=1-1, to=1-2]
\arrow["{\pi_1}"', from=1-1, to=2-1]
\arrow["q", from=1-2, to=2-2]
\arrow["q"', from=2-1, to=2-2]
\end{tikzcd}\quad
\begin{tikzcd}
E & {E_2} \\
E & E
\arrow["{\<\id_E,qz_q\>}", from=1-1, to=1-2]
\arrow[equals, from=1-1, to=2-1]
\arrow["{s_q}", from=1-2, to=2-2]
\arrow[equals, from=2-1, to=2-2]
\end{tikzcd}\quad
\begin{tikzcd}
M & E \\
M & M
\arrow["{z_q}", from=1-1, to=1-2]
\arrow[equals, from=1-1, to=2-1]
\arrow["q", from=1-2, to=2-2]
\arrow[equals, from=2-1, to=2-2]
\end{tikzcd}\quad
\begin{tikzcd}
{E_3} & {E_2} \\
{E_2} & E
\arrow["{s_q\times_M\id_E}", from=1-1, to=1-2]
\arrow["{\id_E\times_Ms_q}"', from=1-1, to=2-1]
\arrow["{s_q}", from=1-2, to=2-2]
\arrow["{s_q}"', from=2-1, to=2-2]
\end{tikzcd}\quad
\begin{tikzcd}
{E_2} & {E_2} \\
E & E
\arrow["\tau", from=1-1, to=1-2]
\arrow["{s_q}"', from=1-1, to=2-1]
\arrow["{s_q}", from=1-2, to=2-2]
\arrow[equals, from=2-1, to=2-2]
\end{tikzcd}
}
\end{equation*}
where $\tau\colon A\times B\to B\times A$ denotes the canonical symmetry. Furthermore:
\begin{enumerate}
\item The pair $(z,l_q)\colon q\to\T q$ is an additive bundle morphism:
\begin{equation*}
\begin{tikzcd}
E & {\T E} \\
M & {\T M}
\arrow["{l_q}", from=1-1, to=1-2]
\arrow["q"', from=1-1, to=2-1]
\arrow["{\T q}", from=1-2, to=2-2]
\arrow["z_M"', from=2-1, to=2-2]
\end{tikzcd}\quad
\begin{tikzcd}
E & {\T E} \\
M & {\T M}
\arrow["{l_q}", from=1-1, to=1-2]
\arrow["{z_q}", from=2-1, to=1-1]
\arrow["z_M"', from=2-1, to=2-2]
\arrow["{\T z_q}"', from=2-2, to=1-2]
\end{tikzcd}\quad
\begin{tikzcd}
{E_2} & {\T E_2} \\
E & {\T E}
\arrow["{l_q\times_zl_q}", from=1-1, to=1-2]
\arrow["{s_q}"', from=1-1, to=2-1]
\arrow["{\T s_q}", from=1-2, to=2-2]
\arrow["{l_q}"', from=2-1, to=2-2]
\end{tikzcd}
\end{equation*}

\item The pair $(z_q,l_q)\colon q\to p_E$ is an additive bundle morphism:
\begin{equation*}
\begin{tikzcd}
E & {\T E} \\
M & E
\arrow["{l_q}", from=1-1, to=1-2]
\arrow["q"', from=1-1, to=2-1]
\arrow["{p_E}", from=1-2, to=2-2]
\arrow["{z_q}"', from=2-1, to=2-2]
\end{tikzcd}\quad
\begin{tikzcd}
E & {\T E} \\
M & E
\arrow["{l_q}", from=1-1, to=1-2]
\arrow["{z_q}", from=2-1, to=1-1]
\arrow["{z_q}"', from=2-1, to=2-2]
\arrow["z_E"', from=2-2, to=1-2]
\end{tikzcd}\quad
\begin{tikzcd}
{E_2} & {\T_2E} \\
E & {\T E}
\arrow["{l_q\times_{z_q}l_q}", from=1-1, to=1-2]
\arrow["{s_q}"', from=1-1, to=2-1]
\arrow["s_E", from=1-2, to=2-2]
\arrow["{l_q}"', from=2-1, to=2-2]
\end{tikzcd}
\end{equation*}

\item The vertical lift satisfies the following universal property: the diagram
\begin{equation*}
\begin{tikzcd}
{E_2} & {\T E} \\
M & {\T M}
\arrow["{\xi_q}", from=1-1, to=1-2]
\arrow["{\pi_1q}"', from=1-1, to=2-1]
\arrow["\lrcorner"{anchor=center, pos=0.125}, draw=none, from=1-1, to=2-2]
\arrow["{\T q}", from=1-2, to=2-2]
\arrow["z_M"', from=2-1, to=2-2]
\end{tikzcd}
\end{equation*}
is a tangent pullback diagram, where $\xi_q\colon E_2\xrightarrow{\<l_q,z_E\>}\T E_2\xrightarrow{s_q}\T E$;

\item The vertical lift $l_q$ is compatible with the vertical lift of the tangent bundle $l_E$:
\begin{equation*}
\begin{tikzcd}
{\T E} & {\T^2E} \\
E & {\T E}
\arrow["l_E", from=1-1, to=1-2]
\arrow["{l_q}", from=2-1, to=1-1]
\arrow["{l_q}"', from=2-1, to=2-2]
\arrow["{\T l_q}"', from=2-2, to=1-2]
\end{tikzcd}
\end{equation*}
\end{enumerate}
A \textbf{display differential bundle} is a differential bundle whose projection is a tangent display map.
\end{definition}

We denote a (display) differential bundle by $\q\colon E\to M$ and the corresponding projection, zero morphism, sum morphism, and vertical lift by $q$, $z_q$, $s_q$, and $l_q$, respectively. When the symbol adopted for a differential bundle is decorated with a superscript or a lowerscript, e.g., $\q'_\bullet$, we use the same decorations for the structural morphisms, that is, the structural morphisms of $\q'_\bullet$ are denoted as follows: $q'_\bullet$, ${z'_q}_\bullet$, ${s'_q}_\bullet$, and ${l'_q}_\bullet$.

\begin{example}
\label{example:db-trivial-tangent-bundle}
In any tangent category, every object $M$ canonically has two differential bundles over it:
\begin{itemize}
\item The trivial differential bundle $\0_M\colon M \to M$, whose structural morphisms $q$, $z_q$, and $s_q$ are simply the identity on $M$ and whose vertical lift coincides with the zero morphism $z_M\colon M\to\T M$. Notice that, as the notation suggests, $\0_M$ is the zero object of the category $\DB(\X,\TT;M)$ of differential bundles over $M$;

\item The \textbf{tangent bundle} $\p_M\colon TM \to M$.
\end{itemize}
\end{example}

\begin{example}
\label{example:db-smooth-manifold}
In the tangent category of smooth manifolds, differential bundles are precisely vector bundles~\cite{macadam:vector-bundles}.
\end{example}

\begin{example}
\label{example:db-affine-schemes}
In the tangent category of affine schemes, differential bundles correspond to modules over rings~\cite[Prop. 4.15]{cruttwell:algebraic-geometry}.
\end{example}

We shall now turn our attention on morphisms of differential bundles.

\begin{definition}[{\cite[Definition~2.3]{cockett:differential-bundles}}]
\label{definition:morphism-differential-bundles}
A \textbf{morphism of} (\textbf{display}) \textbf{differential bundles} $(f,g)\colon\q\to\q'$ from a (display) differential bundle $\q\colon E\to M$ to a (display) differential bundle $\q'\colon E'\to M'$ consists of a commutative square
\begin{equation*}
\begin{tikzcd}
E & {E'} \\
M & {M'}
\arrow["g", from=1-1, to=1-2]
\arrow["q"', from=1-1, to=2-1]
\arrow["{q'}", from=1-2, to=2-2]
\arrow["f"', from=2-1, to=2-2]
\end{tikzcd}
\end{equation*}
between the projections. A morphism is \textbf{linear} provided that it preserves the lift; that is, the diagram
\begin{equation*}
\begin{tikzcd}
{\T E} & {\T  E'} \\
E & {E'}
\arrow["{\T g}", from=1-1, to=1-2]
\arrow["{l_q}", from=2-1, to=1-1]
\arrow["g"', from=2-1, to=2-2]
\arrow["{l_q'}"', from=2-2, to=1-2]
\end{tikzcd}
\end{equation*}
commutes.
\end{definition}

(Display) differential bundles together with linear morphisms form a tangent category denoted by $\DB(\X,\TT)$ whose tangent bundle functor sends a differential bundle $\q\colon E\to M$ to the differential bundle $\T\q\colon\T E\to\T M$ whose projection, zero, and sum morphisms are $\T q$, $\T z_q$, and $\T s_q$, respectively, and whose vertical lift is the morphism:
\begin{align*}
&l_{\T q}\colon\T E\xrightarrow{\T l_q}\T^2E\xrightarrow{c_E}\T^2E
\end{align*}
Furthermore, we denote by $\DB(\X,\TT;M)$ the tangent category of differential bundles with fixed target $M$ and linear morphisms, whose tangent bundle functor is defined as in the slice tangent category.

\begin{convention}
\label{convention:display-differential-bundles}
In the following, unless otherwise indicated, we assume the differential bundles of a tangent category to be display; we also refer to a display differential bundle simply as a differential bundle, and we stress \textit{ordinary differential bundles} when we only consider differential bundles which are not necessarily display.
\end{convention}

As proved in~\cite[Corollaries~3.8 and~3.9]{cruttwell:tangent-display-maps}, under mild conditions, every differential bundle is display, provided that tangent bundles are display.

\subsection{The vertical bundle}
\label{subsection:vertical-bundle}
In differential geometry, the vertical bundle of a submersion $q\colon E\to M$ is the sub-bundle $\q^\V\colon\V q\to E$ of the tangent bundle $\p_E\colon\T E\to E$ on $E$ spanned by the \emph{vertical tangent vectors}, which are those tangent vectors $v\in\T E$ of $E$ sent to zero by the differential of the projection $q$. In other words, the fibre $\V_yq$ of $\q^\V$ at $y\in E$ is the kernel of the differential $\d_yq\colon\T_yE\to\T_{q(y)}M$.

\par An important property exhibited by tangent display maps is that they always admit a vertical bundle. To begin, let us consider a tangent display map $q\colon E\to M$ in a tangent category $(\X,\TT)$. By the properties of the tangent display maps, the (tangent) pullback
\begin{equation}
\label{equation:vertical-bundle-pullback}
\begin{tikzcd}
{\V q} & {\T E} \\
M & {\T M}
\arrow["{\iota_q}", from=1-1, to=1-2]
\arrow["{\T^\V q}"', from=1-1, to=2-1]
\arrow["\lrcorner"{anchor=center, pos=0.125}, draw=none, from=1-1, to=2-2]
\arrow["{\T q}", from=1-2, to=2-2]
\arrow["{z_M}"', from=2-1, to=2-2]
\end{tikzcd}
\end{equation}
of $\T q$ along the zero morphism exists. The vertical bundle of $q$ is the map $q^\V\colon\V q\to E$ obtained by composing $\iota_q\colon\V q\to\T E$ with the projection $p_E\colon\T E\to E$. We want to show that such a map $q^\V$ carries the structure of a differential bundle.

One could directly construct the structural morphisms of this differential bundle. Instead, we follow a more structural approach that clarifies the origin of this differential bundle.

\par As already proved by Rosick\'y in~\cite{rosicky:tangent-cats}, the slice category $(\X,\TT)/M$, that is, the category whose objects are tangent display maps with target $M$ and whose morphisms are commutative triangles, comes equipped with a canonical tangent structure. In this tangent category, known as the \textbf{slice tangent category}, the tangent bundle functor sends each $q\colon E\to M$ to $\T^\V q\colon\V q\to M$ defined by the pullback of Equation~\eqref{equation:vertical-bundle-pullback} (see also~\cite[Section~3.3]{cruttwell:tangent-display-maps} for details on this construction).

\par To begin, recall that a \textbf{lax tangent morphism}~\cite[Definition~2.7]{cockett:tangent-cats} from a tangent category $(\X,\TT)$ to a tangent category $(\X',\TT')$ consists of a functor $F\colon\X\to\X'$ together with a natural transformation $\alpha_M\colon F\T M\to\T'FM$, natural in $M$, which is compatible with the tangent structures. Recall that, a \textbf{Cartesian tangent morphism}~\cite[Definition~4.16]{cockett:differential-bundles} is a lax tangent morphism $(F,\alpha)\colon(\X,\TT)\to(\X',\TT')$ whose underlying functor $F$ preserves tangent pullbacks and whose distributive law is Cartesian, that is, for each morphism $f\colon M\to N$ of $(\X,\TT)$, the naturality square of $f$
\begin{equation*}
\begin{tikzcd}
{F\T M} & {\T'FM} \\
{F\T N} & {\T'FN}
\arrow["{\alpha_M}", from=1-1, to=1-2]
\arrow["{F\T f}"', from=1-1, to=2-1]
\arrow["{\T'Ff}", from=1-2, to=2-2]
\arrow["{\alpha_N}"', from=2-1, to=2-2]
\end{tikzcd}
\end{equation*}
is a tangent pullback.

\begin{lemma}
\label{lemma:tangent-morphism-slice}
The domain functor $\Pi\colon(\X,\TT)/M\to(\X,\TT)$ which sends a tangent display map $q\colon E\to M$ to $E$ and a morphism $f\colon q^E_M\to q'^{E'}_M$ to $f\colon E\to E'$ extends to a Cartesian tangent morphism with distributive law defined as follows:
\begin{align*}
&\Pi(\T^\V q)=\V q\xrightarrow{\iota_q}\T E=\T\Pi(q)
\end{align*}
\end{lemma}
\begin{proof}
Straightforward by using Lemma~\ref{lemma:property-tangent-pullbacks}.    
\end{proof}

\cite[Proposition~4.22]{cockett:differential-bundles} establishes that every Cartesian morphism $(F,\alpha)\colon(\X,\TT)\to(\X',\TT')$ sends ordinary differential bundles to ordinary differential bundles. In particular, given an ordinary differential bundle $\q\colon E\to M$ of $(\X,\TT)$, $\DB(F,\alpha)[\q]$ is the ordinary differential bundle whose structural morphisms are defined as follows:
\begin{align*}
FE\xrightarrow{Fq}FM      &&FM\xrightarrow{Fz_q}FE      &&(FE)_2\xrightarrow{\<F\pi_1,F\pi_2\>}F(E_2)\xrightarrow{Fs_q}FE &&FE\xrightarrow{Fl_q}F\T E\xrightarrow{\alpha_E}\T'FE
\end{align*}
Therefore, by Lemma~\ref{lemma:tangent-morphism-slice}, the domain functor $\Pi\colon(\X,\TT)/M\to(\X,\TT)$ induces a lax tangent morphism
\begin{align*}
\DB(\Pi,\iota)&\colon\DB((\X,\TT)/M)\to\DB(\X,\TT)
\end{align*}
from the category of differential bundles of the slice tangent category to the category of differential bundles of the base tangent category. Recall that in any tangent category, each tangent bundle $\p_X\colon\T X\to X$ is a differential bundle (Example~\ref{example:db-trivial-tangent-bundle}). Therefore, given a tangent display map $q\colon E\to M$ of $(\X,\TT)$, the functor $\DB(\Pi,\iota)$ sends the tangent bundle $\p^\V_q\colon\T^\V q\to q$ of $q$ in $(\X,\TT)/M$ to a differential bundle $\q^\V\colon\V q\to E$.

\begin{definition}
\label{definition:vertical-bundle}
The \textbf{vertical bundle} of a tangent display map $q\colon E\to M$ in a tangent category $(\X,\TT)$ is the differential bundle $\q^\V\colon\V q\to E$ image of the tangent bundle $\p^\V_q\colon\T^\V q\to q$ of $q$ in the slice tangent category over $M$ via the Cartesian tangent morphism $\DB(\Pi,\iota)$.
\end{definition}

\begin{example}
\label{example:vertical-bundle-differential-bundles}
The universal property enjoyed by the vertical lift $l_q$ of a differential bundle $\q\colon E\to M$ can be equivalently expressed by saying that the vertical bundle $\q^\V$ of $\q$ is necessarily trivial, that is, $\V\q=E_2$ and $\q^\V=\pi_1\colon E_2\to E$. This property of differential bundles can be understood as being ``locally linear''.
\end{example}

Concretely, the vertical bundle $\q^\V$ of a tangent display map is the differential bundle so defined:
\begin{description}
\item[Projection] The projection $q^\V\colon\V q\to E$ is the morphism defined as follows:
\begin{align*}
&q^\V\colon\V q\xrightarrow{\iota_q}\T E\xrightarrow{p_E}E
\end{align*}

\item[Zero morphism] The zero morphism $z_q^\V\colon E\to\V q$ is the morphism defined as follows:
\begin{equation*}
\begin{tikzcd}
E \\
& {\V q} & {\T E} \\
& M & {\T M}
\arrow["{z_q^\V}", dashed, from=1-1, to=2-2]
\arrow["{z_E}", curve={height=-12pt}, from=1-1, to=2-3]
\arrow["q"', curve={height=12pt}, from=1-1, to=3-2]
\arrow["{\iota_q}", from=2-2, to=2-3]
\arrow["{\T^\V q}"', from=2-2, to=3-2]
\arrow["\lrcorner"{anchor=center, pos=0.125}, draw=none, from=2-2, to=3-3]
\arrow["{\T q}", from=2-3, to=3-3]
\arrow["{z_M}"', from=3-2, to=3-3]
\end{tikzcd}
\end{equation*}

\item[Sum morphism] The sum morphism $s_q^\V\colon\V_2q\to\V q$ is the morphism defined as follows
\begin{equation*}
\begin{tikzcd}
	{\V_2q} && {\T_2E} \\
	& {\V q} & {\T E} \\
	{\V q} & M & {\T M}
	\arrow["{\iota_q\times\iota_q}", from=1-1, to=1-3]
	\arrow["{s_q^\V}", dashed, from=1-1, to=2-2]
	\arrow["{\pi_1^\V}"', from=1-1, to=3-1]
	\arrow["{s_E}", from=1-3, to=2-3]
	\arrow["{\iota_q}", from=2-2, to=2-3]
	\arrow["{\T^Mq}"', from=2-2, to=3-2]
	\arrow["\lrcorner"{anchor=center, pos=0.125}, draw=none, from=2-2, to=3-3]
	\arrow["{\T q}", from=2-3, to=3-3]
	\arrow["{\T^Mq}"', from=3-1, to=3-2]
	\arrow["{z_M}"', from=3-2, to=3-3]
\end{tikzcd}
\end{equation*}
where $\V_2q$ denotes the pullback of $q^\F$ along itself and $\pi_1^\V\colon\V_2q\to\V q$, the projection;

\item[Vertical lift] The vertical lift $l_q^\V\colon\V q\to\T\V q$ is the morphism defined as follows:
\begin{equation*}
\begin{tikzcd}
{\V q} && {\T E} \\
& {\T\V q} & {\T^2E} \\
M & {\T M} & {\T^2M}
\arrow["{\iota_q}", from=1-1, to=1-3]
\arrow["{l_q^\V}", dashed, from=1-1, to=2-2]
\arrow["{\T^\V q}"', from=1-1, to=3-1]
\arrow["{l_E}", from=1-3, to=2-3]
\arrow["{\T\iota_q}", from=2-2, to=2-3]
\arrow["{\T\T^\V q}"', from=2-2, to=3-2]
\arrow["\lrcorner"{anchor=center, pos=0.125}, draw=none, from=2-2, to=3-3]
\arrow["{\T^2q}", from=2-3, to=3-3]
\arrow["{z_M}"', from=3-1, to=3-2]
\arrow["{\T z_M}"', from=3-2, to=3-3]
\end{tikzcd}
\end{equation*}
\end{description}

\begin{remark}
\label{remark:vertical-bundle-0-carrability}
Definition~\ref{definition:vertical-bundle} introduces the vertical bundle of a tangent display map. However, this definition can be easily extended to any map for which the tangent pullback of Equation~\eqref{equation:vertical-bundle-pullback} exists. In~\cite{lemay:submersions}, those maps are called \textbf{0-carrable}. In our paper, we need this level of generality only twice. In those cases, we will simply say that a map \textbf{admits the vertical bundle}.
\end{remark}

\begin{remark}
\label{remark:vertical-bundle-choice}
It is important to realize that \emph{the} vertical bundle of a tangent display map is only defined up to a unique isomorphism, since it is defined through tangent pullbacks. Using our notation $\q^\V$ for the vertical bundle, we are implicitly making a \emph{choice} of the vertical bundle. In Section~\ref{subsection:abstract-connection}, we will develop a theory of connections free from this choice.
\end{remark}

The vertical bundle $\q^\V$ of a tangent display map can be regarded as a sub-bundle of the tangent bundle $\p_E\colon\T E\to E$, as proved by the next lemma.

\begin{lemma}
\label{lemma:iota-tangent-monic}
The morphism $\iota_q\colon\V q\to\T E$ is \textbf{tangent monic}, that is, for every $n\geq0$, $\T^n\iota_q$ is monic.
\end{lemma}
\begin{proof}
For starters, $\iota_q$ is the pullback of $z_M$ along $\T q$. Thus, since monics are stable under pullbacks and $z_M$ is a section of the projection $p_M$, $\iota_q$ is necessarily monic. Furthermore, since $q$ is tangent display, for each $n\geq0$, $\T^n\iota_q$ is the pullback of $\T^nz_M$ along $\T^{n+1}q$. However, each $\T^nz_M$ is a section of $\T^np_M$. Thus, again $\T^n\iota_q$ is the pullback of a monic and therefore a monic itself.
\end{proof}

We conclude this section with a technical result that will be useful later.

\begin{lemma}
\label{lemma:functoriality-V}
There is a strong tangent morphism $(-)^\V\colon\Dsply(\X,\TT)\to\DB(\X,\TT)$ which sends each tangent display map $q\colon E\to M$ to the vertical bundle $\q^\V\colon\V q\to E$ of $q$.
\end{lemma}
\begin{proof}
For starters, let us show that $(-)^\V$ is functorial. Consider a morphism $(f,g)\colon q^E_M\to q'^{E'}_{M'}$ of bundles, that is, a pair of morphisms $f\colon M\to M'$ and $g\colon E\to E'$ such that $qf=gq'$. Define $\V_fg$ as the unique morphism satisfying the following equations. $\V_fg\T^{M'}q'=\T^\V qf$ and $\V_fg\iota_{q'}=\iota_q\T g$. We leave it to the reader to show that $(g,\V_fg)\colon\q^\V\to\q'^\V$ is a linear morphism of differential bundles. Finally, to define a tangent morphism, we need to define a distributive law $(\T q)^\V\to\T^\DB q^\V$ between $(-)^\V$ and the tangent bundle functors. Define $(\id_{\T E},\gamma^\V_q)$ where $\gamma^\V_q\colon\V\T q\to\T\V q$ is the unique morphism such that $\gamma^\V_q\T\T^\V q=\T^{\T M}(\T q)$ and $\gamma^\V_q\T\iota_q=\iota_{\T q}c_E$. Since the canonical flip $c_E$ is invertible, it is not hard to show that $(\id_{\T M},\gamma^\V_q)$ is in fact an isomorphism. We leave it to the reader to show the compatibility conditions between the distributive law and the tangent structures.
\end{proof}

\subsection{The Finsler bundle}
\label{subsection:finsler-bundle}
We now introduce another differential bundle associated to a tangent display map. We refer to this as the \textbf{Finsler bundle} of a tangent display map. This naming choice was inspired by the name in differential geometry given to the bundle $\T_2M\to M$, which is called the Finsler bundle \cite[Remark 4.1.1]{kertesz:connections}. Our definition generalizes this bundle to an arbitrary tangent display map.

In differential geometry, given a submersion $q\colon E\to M$ we may define a vector bundle $\q^\F\colon\F q\to E$ on $E$ whose fibre $\F_yq$ at each $y\in E$ is the Cartesian product of the tangent space $\T_xM$ of $M$ at $x\=q(y)$ and the fibre $E_x$ of $q$ over $x$.

To generalize this to tangent categories, consider a tangent display map $q\colon E\to M$ in a tangent category $(\X,\TT)$. The tangent pullback of $q$ along the projection $p_M\colon\T M\to M$
\begin{equation}
\label{equation:finsler-bundle-pullback}
\begin{tikzcd}
{\F q} & E \\
{\T M} & M
\arrow["{q^\F}", from=1-1, to=1-2]
\arrow["{\T^\F q}"', from=1-1, to=2-1]
\arrow["\lrcorner"{anchor=center, pos=0.125}, draw=none, from=1-1, to=2-2]
\arrow["q", from=1-2, to=2-2]
\arrow["{p_M}"', from=2-1, to=2-2]
\end{tikzcd}
\end{equation}
exists. By~\cite[Lemma~2.7]{cockett:differential-bundles}, the tangent pullback of a differential bundle is still a differential bundle. Thus, since the tangent bundle $\p_M\colon\T M\to M$ is a differential bundle, the projection $q^\F\colon\F q\to E$ also carries the structure of a differential bundle $\q^\F$.

\begin{definition}
\label{definition:finsler-bundle}
The \textbf{Finsler bundle} of a tangent display map $q\colon E\to M$ is the differential bundle $\q^\F\colon\F q\to E$, which is the tangent pullback of $\p_M\colon\T M\to M$ along $q$.
\end{definition}

\begin{remark}
\label{remark:finsler-bundle-p-carrable}
As with the vertical bundle (see Remark~\ref{remark:vertical-bundle-0-carrability}), also the definition of the Finsler bundle can be easily extended to every map $q\colon E\to M$ for which the tangent pullback of Equation~\eqref{equation:finsler-bundle-pullback} exists. In~\cite{lemay:submersions}, those maps are called \textbf{p-carrable}. In our paper, we only make use of this level of generality once. In that case, we will simply say that a map \textbf{admits the Finsler bundle}. The theory of connections we are developing extends to the more general case, as long as the required tangent pullbacks exist. Our choice of using tangent display maps allows us not to worry about these pullbacks. We also point out that Lemay and Vooys in~\cite{lemay:submersions} call the Finsler bundle the horizontal bundle. However, this clashes with the usual notion of horizontal distribution, which is the embedding of the Finsler bundle in the tangent bundle given by the connection.
\end{remark}

\begin{remark}
\label{remark:finler-bundle-choice}
As mentioned in Remark~\ref{remark:vertical-bundle-choice} for the vertical bundle, \emph{the} Finsler bundle of a tangent display map is only defined up to a unique isomorphism, since it is defined through tangent pullbacks. In using our notation $\q^\F$ we are implicitly making a choice. In Section~\ref{subsection:abstract-connection}, we will develop a theory of connections free from this choice.
\end{remark}

\begin{remark}
\label{remark:finsler-bundle}
Notice that, since we have assumed all the differential bundles are display, each tangent bundle $\p_M$ is automatically a tangent display map. Therefore, every morphism $q\colon E\to M$ admits the tangent pullback of $p_M$ along $q$. In particular, assuming the tangent bundles of a tangent category being display is equivalent to requiring the existence of the Finsler bundle for every morphism of the category.
\end{remark}

Concretely, the Finsler bundle $\q^\F\colon\F q\to E$ of a tangent display map $q\colon E\to M$ consists of the differential bundle so defined:
\begin{description}
\item[Projection] The projection of $\q^\F$ is $q^\F\colon\F q\to E$;

\item[Zero morphism] The zero morphism $z_q^\F\colon E\to\F q$ is defined as follows:
\begin{equation*}
\begin{tikzcd}
E \\
& {\F q} & E \\
M & {\T M} & M
\arrow["{z_q^\F}", dashed, from=1-1, to=2-2]
\arrow[curve={height=-18pt}, equals, from=1-1, to=2-3]
\arrow["q"', from=1-1, to=3-1]
\arrow["{q^\F}", from=2-2, to=2-3]
\arrow["{\T^\F q}"', from=2-2, to=3-2]
\arrow["\lrcorner"{anchor=center, pos=0.125}, draw=none, from=2-2, to=3-3]
\arrow["q", from=2-3, to=3-3]
\arrow["{z_M}"', from=3-1, to=3-2]
\arrow["{p_M}"', from=3-2, to=3-3]
\end{tikzcd}
\end{equation*}

\item[Sum morphism] The sum morphism $s_q^\F\colon\F_2q\to\F q$ is defined as follows
\begin{equation*}
\begin{tikzcd}
{\F_2q} && {\F q} \\
& {\F q} & E \\
{\T_2M} & {\T M} & M
\arrow["{\pi_1^\F}", from=1-1, to=1-3]
\arrow["{s_q^\F}", dashed, from=1-1, to=2-2]
\arrow["{\T^\F q\times\T^\F q}"', from=1-1, to=3-1]
\arrow["{q^\F}", from=1-3, to=2-3]
\arrow["{q^\F}", from=2-2, to=2-3]
\arrow["{\T^\F q}"', from=2-2, to=3-2]
\arrow["\lrcorner"{anchor=center, pos=0.125}, draw=none, from=2-2, to=3-3]
\arrow["q", from=2-3, to=3-3]
\arrow["{s_M}"', from=3-1, to=3-2]
\arrow["{p_M}"', from=3-2, to=3-3]
\end{tikzcd}
\end{equation*}
where $\pi_1^\F$ denotes the projection of the pullback of $q^\F$ along itself;

\item[Vertical lift] The vertical lift $l_q^\F\colon\F q\to\T\F q$ is defined as follows:
\begin{equation*}
\begin{tikzcd}
{\F q} && E \\
& {\T\F q} & {\T E} \\
{\T M} & {\T^2M} & {\T M}
\arrow["{q^\F}", from=1-1, to=1-3]
\arrow["{l_q^\F}", dashed, from=1-1, to=2-2]
\arrow["{\T^\F q}"', from=1-1, to=3-1]
\arrow["{z_E}", from=1-3, to=2-3]
\arrow["{\T q^\F}", from=2-2, to=2-3]
\arrow["{\T\T^\F q}"', from=2-2, to=3-2]
\arrow["\lrcorner"{anchor=center, pos=0.125}, draw=none, from=2-2, to=3-3]
\arrow["{\T q}", from=2-3, to=3-3]
\arrow["{l_M}"', from=3-1, to=3-2]
\arrow["{\T p_M}"', from=3-2, to=3-3]
\end{tikzcd}
\end{equation*}
\end{description}
The Finsler bundle $\q^\F\colon\F q\to E$ of a tangent display map $q\colon E\to M$ comes with an important morphism that will play a crucial role in the definition of a connection.

\begin{definition}
\label{definition:horizontal-descent}
The \textbf{horizontal descent} of a tangent display map $q\colon E\to M$ is the unique morphism $\pi_q\colon\T E\to\F q$ which makes the following diagram
\begin{equation*}
\begin{tikzcd}
{\T E} \\
& {\F q} & E \\
& {\T M} & M
\arrow["{\pi_q}", dashed, from=1-1, to=2-2]
\arrow["{p_E}", curve={height=-18pt}, from=1-1, to=2-3]
\arrow["{\T q}"', curve={height=18pt}, from=1-1, to=3-2]
\arrow["{q^\F}", from=2-2, to=2-3]
\arrow["{\T^\F q}"', from=2-2, to=3-2]
\arrow["\lrcorner"{anchor=center, pos=0.125}, draw=none, from=2-2, to=3-3]
\arrow["q", from=2-3, to=3-3]
\arrow["{p_M}"', from=3-2, to=3-3]
\end{tikzcd}
\end{equation*}
commutative.
\end{definition}

\begin{lemma}
\label{lemma:horizontal-descent-linear}
The horizontal descent of a tangent display map defines a linear morphism of differential bundles $\pi_q=(\id_E,\pi_q)\colon\p_E\to\q^\F$.
\end{lemma}
\begin{proof}
Follows by straightforward calculation.  
\end{proof}

We conclude this section with a technical result that will be useful later.

\begin{lemma}
\label{lemma:functoriality-F}
There is a strong tangent morphism $(-)^\F\colon\Dsply(\X,\TT)\to\DB(\X,\TT)$ which sends each tangent display map $q\colon E\to M$ to the Finsler bundle $\q^\F\colon\F q\to E$ of $q$.
\end{lemma}
\begin{proof}
For starters, let us show that $(-)^\F$ is functorial. Consider a morphism $(f,g)\colon q^E_M\to q'^{E'}_{M'}$ of bundles, that is, a pair of morphisms $f\colon M\to M'$ and $g\colon E\to E'$ such that $qf=gq'$. Define $\F_fg$ as the unique morphism satisfying the following equations. $\F_fg\T^\F q=\T^\F q\T f$ and $\F_fgq'^\F=q^\F g$. We leave it to the reader to show that $(g,\F_fg)\colon\q^\F\to\q'^\F$ is a linear morphism of differential bundles. Finally, to define a tangent morphism, we need to define a distributive law $(\T q)^\F\to\T^\DB q^\F$ between $(-)^\F$ and the tangent bundle functors. Define $(\id_{\T E},\gamma^\F_q)$ where $\gamma^\F_q\colon\F\T q\to\T\F q$ is the unique morphism such that $\gamma^\F_q\T^\F q=\T^\F q c_M$ and $\gamma^\F_q\T(q^\F)=(\T q)^\F$. Since the canonical flip $c_M$ is invertible, it is not hard to show that $(\id_{\T M},\gamma^\F_q)$ is in fact an isomorphism. We leave it to the reader to show the compatibility conditions between the distributive law and the tangent structures.
\end{proof}

\subsection{Submersions and the fundamental short exact sequence}
\label{subsection:fundamental-ses}
Lemma~\ref{lemma:iota-tangent-monic} proves that the vertical bundle $\q^\V\colon\V q\to E$ of a tangent display map can be allegedly regarded as a sub-bundle of the tangent bundle $\p_M\colon\T E\to E$, since $\iota_q$ is tangent monic. One would expect the horizontal descent $\pi_q\colon\T E\to\F q$ to be an epimorphism and that the Finsler bundle $\q^\F$ could be regarded as the cokernel of $\iota_q$. However, in general, this is not guaranteed.
\par In the context of differential geometry, the horizontal descent $\pi_q$ of a smooth function $q$ sends a tangent vector $v\in\T_yE$ of $E$ at $y$ to the pair $(\d_yq(v),y)\in\T_xM\times E_x$, where $x=q(y)$. Therefore, asking that $\pi_q$ is surjective is equivalent to saying that for every tangent vector $u\in\T_xM$ of $M$ at $x$ and every $y\in E$ such that $q(y)=x$, there exists a tangent vector $v\in\T_yE$ such that $\d_yq(v)=u$. However, this is precisely what it means for $q$ to be a submersion.
\par We start by recalling a result due to Lemay and Vooys, which shows that the vertical bundle is always the kernel of $\pi_q$.

\begin{lemma}[{\cite[Theorem~4.2.1]{lemay:submersions}}]
\label{lemma:vertical-bundle-kernel}
The vertical bundle $\q^\V$ of a tangent display map $q\colon E\to M$ is the kernel of the horizontal descent, explicitly, the diagram
\begin{equation*}
\begin{tikzcd}
{\q^\V} & {\p_E} \\
{\0_E} & {\q^\F}
\arrow["{\iota_q}", from=1-1, to=1-2]
\arrow["{q^\V}"', from=1-1, to=2-1]
\arrow["{\pi_q}", from=1-2, to=2-2]
\arrow["{z_q^\F}"', from=2-1, to=2-2]
\end{tikzcd}
\end{equation*}
is a pullback in $\DB(\X,\TT;E)$.
\end{lemma}

\begin{remark}
\label{remark:vertical-bundle-kernel}
The original statement of~\cite[Theorem~4.2.1]{lemay:submersions} establishes that the vertical bundle is the equalizer in $\DB(\X,\TT)$ of the horizontal descent $\pi_q\colon\T E\to\F q$ against the morphism:
\begin{align*}
&\T E\xrightarrow{p_E}E\xrightarrow{z^\F_q}\F q
\end{align*}
We decided to adopt the equivalent form in terms of a pullback.
\end{remark}

We may now define the concept of submersions.

\begin{definition}
\label{definition:submersion}
A \textbf{submersion} in a tangent category is a tangent display map $q\colon E\to M$ for which the diagram
\begin{equation}
\label{equation:submersion-diagram}
\begin{tikzcd}
{\q^\V} & {\p_E} \\
{\0_E} & {\q^\F}
\arrow["{\iota_q}", from=1-1, to=1-2]
\arrow["{q^\V}"', from=1-1, to=2-1]
\arrow["{\pi_q}", from=1-2, to=2-2]
\arrow["{z_q^\F}"', from=2-1, to=2-2]
\end{tikzcd}
\end{equation}
is a pushout in $\DB(\X,\TT;E)$.
\end{definition}

\begin{remark}
\label{remark:pushout-in-DB}
It is important to realize that in Definition~\ref{definition:submersion}, the diagram~\ref{equation:submersion-diagram} is required to be a pushout in the category of differential bundles and linear morphisms over $E$, not in the base tangent category.
\end{remark}

For a submersion, the square Diagram~\ref{equation:submersion-diagram} is both a pullback and a pushout in the category of differential bundles. Recall that in a \textbf{semi-additive category}\footnote{In the literature, it is sometimes called, \emph{additive} (see~\cite{lucyshyn:connections}). We will reserve this name to categories with biproducts, enriched over Abelian groups.}, which is a category with finite biproducts, such as $\DB(\X,\TT;E)$ (\cite{lucyshyn:connections}), a \textbf{short exact sequence} consists of a sequence $0\to A\xrightarrow{\iota} B\xrightarrow{\pi} C\to 0$, where $0$ is the zero object, such that the following diagram
\begin{equation*}
\begin{tikzcd}
A & B \\
0 & C
\arrow["\iota", from=1-1, to=1-2]
\arrow[from=1-1, to=2-1]
\arrow["\pi", from=1-2, to=2-2]
\arrow[from=2-1, to=2-2]
\end{tikzcd}
\end{equation*}
is both a pullback and a pushout square.

\begin{theorem}
\label{theorem:submersion}
For a given tangent display map $q\colon E\to M$ of a tangent category, the following are equivalent:
\begin{enumerate}
\item $q$ is a submersion;

\item The horizontal descent $\pi_q$ of $q$ is the cokernel of $\iota_q$;

\item The sequence of differential bundles
\begin{align*}
&\0_E\to\q^\V\xrightarrow{\iota_q}\p_E\xrightarrow{\pi_q}\q^\F\to\0_E
\end{align*}
is short exact in $\DB(\X,\TT;E)$.
\end{enumerate}
\end{theorem}

\begin{definition}
\label{definition:fundamental-ses}
The \textbf{fundamental short exact sequence} of a submersion is the short exact sequence of differential bundles of Theorem~\ref{theorem:submersion}.
\end{definition}

\begin{remark}
\label{remark:fundamental-ses}
In~\cite[Definition~4.2.2]{lemay:submersions}, the fundamental short exact sequence is called the \textit{relative cotangent sequence} in agreement with the algebraic geometry literature.
\end{remark}

One of the desirable properties of submersions in differential geometry is their stability under pullbacks and the tangent bundle functor. Unfortunately, in a general tangent category, there is no reason for a submersion, as per Definition~\ref{definition:submersion}, to behave well with respect to these two operations. For this reason, we introduce a stronger concept.

\begin{definition}
\label{definition:display-submersion}
A \textbf{tangent display submersion} consists of a tangent display map $q\colon E\to M$ for which, for any integer $n\geq0$, $\T^nq$ is a submersion. Furthermore, for every morphism $f\colon M'\to\T^nM$, the tangent pullback $q'\colon E'\to M'$ of $\T^nq$ along $f$ is again a submersion.
\end{definition}

By design, tangent display submersions form a tangent display system, and therefore, they form a \textbf{tangent fibration} (\cite[Proposition~3.11]{cruttwell:tangent-display-maps}). We shall not explain the details of this important structure, since it is not of primary importance for the story of this paper. Instead, we refer to~\cite[Section~5]{cockett:differential-bundles}.

\begin{proposition}
\label{proposition:submersions-tangent-display-system}
Tangent display submersions in a tangent category form a tangent display system $\Sbm(\X,\TT)$. Furthermore, the codomain functor $\Pi\colon\Sbm(\X,\TT)\to(\X,\TT)$ is a tangent fibration.
\end{proposition}
\begin{proof}
Consider a tangent display submersion $q\colon E\to M$. By definition, $\T q$ is again a submersion. Furthermore, for any $n\geq0$, $\T^n\T q=\T^{n+1}q$ is a submersion and for every $f\colon M'\to\T^{n+1}M$, the tangent pullback of $\T q$ along $f$ is again a submersion, since $q$ is a tangent display submersion. This proves that $\T q$ is in fact a tangent display submersion. Now, consider a morphism $f\colon M'\to M$ and the tangent pullback $q'\colon E'\to M'$ of $q$ along $f$. By assumption, $q'$ is a submersion. Furthermore, since $\T^nq'$ is the tangent pullback of $\T^nq'$ along $\T^nf$, and since $\T^nq'$ is a tangent display submersion, $\T^nq'$ is also a submersion. Finally, consider a morphism $g\colon M''\to M'$ and let $q''\colon E''\to M''$ be the tangent pullback of $q'$ along $g$. However, $q''$ is also the tangent pullback of $q$ along $gf$. Thus, $q''$ must be a submersion. This proves that $q'$ is a tangent display submersion and therefore, that $\Sbm(\X,\TT)$ is, in fact, a tangent display system. Finally, by invoking~\cite[Proposition~3.11]{cruttwell:tangent-display-maps}, which establishes that every tangent display system defines in fact a tangent fibration, we conclude that the codomain functor $\Pi\colon\Sbm(\X,\TT)\to(\X,\TT)$ is in fact a tangent fibration.
\end{proof}


\section{Ehresmann connections}
\label{section:ehresmann-connections}
In this section, we introduce the main character of this paper: Ehresmann connections in the context of tangent categories. An Ehresmann connection consists of two distinct components: a vertical connection and a horizontal connection, which together form a splitting of the tangent bundle $\T E$ of the total space $E$ of a submersion $q\colon E\to M$. A vertical connection is a map that projects a tangent vector $v\in\T_yE$ of $E$ down to a vertical component $v_\V$. A horizontal connection is a map that sends a pair formed by a point $y$ of $E$ and a tangent vector $u$ of $\T_xM$ where $x=q(y)$ to a tangent vector $v_\H\in\T_yE$ of $E$ at $y$.
\par We start by introducing these two notions separately; we then bring them together to define Ehresmann connections.

\subsection{Vertical connections}
\label{subsection:vertical-connection}
A vertical connection on a tangent display map $q\colon E\to M$ consists of a linear morphism $\R\colon\T E\to\V q$ from the tangent bundle of the total space $E$ of $q$ to the vertical bundle of $q$ which restricts to the identity on vertical vectors. The kernel of a vertical connection defines a sub-bundle of $\T E$ in direct sum with the vertical sub-bundle which encodes the horizontal bundle of the connection. We may begin by introducing this definition in tangent categories.

\begin{definition}
\label{definition:vertical-connection}
A \textbf{vertical connection} on a tangent display map $q\colon E\to M$ consists of a morphism $\R\colon\T E\to\V q$ subject to the following conditions:
\begin{description}
\item[VC.1] The map $\R$ is a retraction of $\iota_q$:
\begin{equation*}
\begin{tikzcd}
{\V q} & {\T E} \\
& {\V q}
\arrow["{\iota_q}", from=1-1, to=1-2]
\arrow[equals, from=1-1, to=2-2]
\arrow["\R", from=1-2, to=2-2]
\end{tikzcd}
\end{equation*}

\item[VC.2] The morphism $\R\colon\p_E\to\q^\V$ is a linear morphism of differential bundles over $E$:
\begin{equation*}
\begin{tikzcd}
{\T E} & {\V q} \\
E & E
\arrow["\R", from=1-1, to=1-2]
\arrow["{p_E}"', from=1-1, to=2-1]
\arrow["{q^\V}", from=1-2, to=2-2]
\arrow[equals, from=2-1, to=2-2]
\end{tikzcd}\quad
\begin{tikzcd}
{\T^2E} & {\T\V q} \\
{\T E} & {\V q}
\arrow["{\T\R}", from=1-1, to=1-2]
\arrow["{l_E}", from=2-1, to=1-1]
\arrow["\R"', from=2-1, to=2-2]
\arrow["{l_q^\V}"', from=2-2, to=1-2]
\end{tikzcd}
\end{equation*}
\end{description}
\end{definition}

\textbf{[VC.1]} requires that $\R$ be a projection to the vertical sub-bundle of $\T E$ which sends every vertical tangent vector to itself and \textbf{[VC.2]} requires that such a projection must be linear.
\par We denote by $\VC(\X,\TT)$ the tangent category of pairs $(q,\R)$ formed by a tangent display map $q$ of $(\X,\TT)$ and a vertical connection $\R$ on $q$, and commutative squares $(f,g)\colon q\to q'$ for morphisms, that is, $qf=gq'$, compatible with the vertical connections as follows:
\begin{equation*}
\begin{tikzcd}
{\T E} & {\T E'} \\
{\V q} & {\V q'}
\arrow["{\T g}", from=1-1, to=1-2]
\arrow["\R"', from=1-1, to=2-1]
\arrow["{\R'}", from=1-2, to=2-2]
\arrow["{\V g}"', from=2-1, to=2-2]
\end{tikzcd}
\end{equation*}
The tangent bundle functor of $\VC(\X,\TT)$ sends each $(q,\R)$ to $(\T q,\R_\T)$, where
\begin{align*}
&\R_\T\colon\T^2E\xrightarrow{c_E}\T^2E\xrightarrow{\T\R}\T\V q\xrightarrow{\gamma_q}\V\T q
\end{align*}
and $\gamma_q\colon\T\V q\to\V\T q$ is the strong distributive law of the tangent morphism $(-)^\V$ of Lemma~\ref{lemma:functoriality-V}.

\begin{lemma}
\label{lemma:vertical-connections}
If $\R\colon\T E\to\V q$ is a vertical connection on a tangent display map $q\colon E\to M$, then:
\begin{enumerate}
\item The morphism $(p_M,\R)\colon\T q\to\T^\V q$ is a bundle morphism:
\begin{equation*}
\begin{tikzcd}
{\T E} & {\V q} \\
{\T M} & M
\arrow["\R", from=1-1, to=1-2]
\arrow["{\T q}"', from=1-1, to=2-1]
\arrow["{\T^\V q}", from=1-2, to=2-2]
\arrow["{p_M}"', from=2-1, to=2-2]
\end{tikzcd}
\end{equation*}

\item The morphism $(q,\R)\colon\p_E\to\T^\V q$ is a bundle morphism:
\begin{equation*}
\begin{tikzcd}
{\T E} & {\V q} \\
E & M
\arrow["\R", from=1-1, to=1-2]
\arrow["{p_E}"', from=1-1, to=2-1]
\arrow["{\T^\V q}", from=1-2, to=2-2]
\arrow["q"', from=2-1, to=2-2]
\end{tikzcd}
\end{equation*}
\end{enumerate}
\end{lemma}
\begin{proof}
To prove (i), we compute:
\begin{align*}
\R \T^\V q&=~\R\iota_qp_Eq                  \Tag{\T^\V q=\iota_qp_Eq}\\
&=~\R q^\V q                            \Tag{q^\V=\iota_qp_E}\\
&=~p_Eq                                   \Tag{\textbf{[VC.2]}}\\
&=~\T qp_M                                \Tag{p_Eq=\T qp_M}
\end{align*}
However, $\R p^\V=p$, by \textbf{[VC.2]}, thus:
\begin{align*}
&\R\T^\V q=\R \T^\V q=p_Eq
\end{align*}
To prove (ii), notice that, by the naturality of $p$, $\R\T^\V q=\T qp_M=p_Eq$.
\end{proof}

Thanks to~\textbf{[VC.1]}, a vertical connection $\R\colon\T E\to\V q$ defines an idempotent $\phi_\R\=\R\iota_q\colon\T E\to\T E$ on the tangent bundle, called the vertical connection form of $\R$. The next definition formalizes this construction.

\begin{definition}
\label{definition:vertical-connection-form}
A \textbf{vertical connection form} of a tangent display map $q\colon E\to M$ is an endomorphism $\phi\colon\T E\to\T E$ of $\T E$ subject to the following conditions:
\begin{description}
\item[VCF.1] The following diagram commutes:
\begin{equation*}
\begin{tikzcd}
{\V q} & {\T E} \\
& {\T E}
\arrow["{\iota_q}", from=1-1, to=1-2]
\arrow["{\iota_q}"', from=1-1, to=2-2]
\arrow["\phi", from=1-2, to=2-2]
\end{tikzcd}
\end{equation*}

\item[VCF.2] The morphism $\phi\colon\p_E\to\p_E$ is a linear endomorphism of differential bundles;

\item[VCF.3] The following diagram commutes:
\begin{equation*}
\begin{tikzcd}
{\T E} && {\T E} \\
E & M & {\T M}
\arrow["\phi", from=1-1, to=1-3]
\arrow["{p_E}"', from=1-1, to=2-1]
\arrow["{\T q}", from=1-3, to=2-3]
\arrow["q"', from=2-1, to=2-2]
\arrow["{z_M}"', from=2-2, to=2-3]
\end{tikzcd}
\end{equation*}
\end{description}
\end{definition}

The next theorem shows that the information of a vertical connection is entirely retained in its vertical connection form.

\begin{theorem}
\label{theorem:vertical-connection-form}
There is a bijective correspondence between vertical connections and vertical connection forms on a given tangent display map $q\colon E\to M$. Concretely, if $\R\colon\T E\to\V q$ is a vertical connection on $q$, the morphism
\begin{align*}
&\phi_\R\colon\T E\xrightarrow{\R}\V q\xrightarrow{\iota_q}\T E
\end{align*}
is a vertical connection form on $q$. Conversely, if $\phi\colon\T E\to\T E$ is vertical connection form on a tangent display map $q$, the unique morphism $\R_\phi\colon\T E\to\V q$ of $\DB(\X,\TT;E)$ rendering the following diagram
\begin{equation*}
\begin{tikzcd}
{\T E} && \\
& {\V q} & {\T E} \\
& E & {\F q}
\arrow["{\R_\phi}", dashed, from=1-1, to=2-2]
\arrow["\phi", curve={height=-24pt}, from=1-1, to=2-3]
\arrow["{p_E}"', curve={height=24pt}, from=1-1, to=3-2]
\arrow["{\iota_q}", from=2-2, to=2-3]
\arrow["{q^\V}"', from=2-2, to=3-2]
\arrow["\lrcorner"{anchor=center, pos=0.125}, draw=none, from=2-2, to=3-3]
\arrow["{\pi_q}", from=2-3, to=3-3]
\arrow["{z_q^\F}"', from=3-2, to=3-3]
\end{tikzcd}
\end{equation*}
commutative, is a vertical connection on $q$.
\end{theorem}
\begin{proof}
For starters, suppose that $\R$ is a vertical connection on a tangent display map $q\colon E\to M$ and consider the morphism $\phi_\R\=\R\iota_q$. Since $\R$ is a retract of $\iota_q$, that is, $\iota_q\R=\id_{\V q}$, we immediately prove that $\iota_q\phi_\R=\iota_q$, that is, \textbf{[VCF.1]} holds. \textbf{[VCF.2]} is also immediate, since, by \textbf{[VC.2]}, $\R$ is a linear morphism of differential bundles and so is $\iota_q$, since, by construction, $\iota_ql_E=l_q^\V\T\iota_q$. However, linear morphisms of differential bundles are stable under composition \cite[Prop. 2.7(i)]{cockett:connections}; thus, $\phi_\R$ is a linear morphism of differential bundles. Let us prove \textbf{[VCF.3]}:
\begin{align*}
\phi_\R\T q&=~\R\iota_q\T q\\
&=~\R\T^\V qz_M                   \Tag{\iota_q\T q=\T^\V qz_M}\\
&=~p_Eqz_M                      \Tag{\text{Lemma}~\ref{lemma:vertical-connections}.(b)}
\end{align*}
Conversely, consider a vertical connection form $\phi\colon\T E\to\T E$ of $q\colon E\to M$ and let us prove that $\R_\phi$ is a vertical connection. First, let us prove that $\R_\phi$ is well-defined. For this, we will begin by proving that $\phi\pi_q=p_Ez_q^\F$, by using the universal property that defines the Finsler bundle. Consider the following:
\begin{align*}
\phi\pi_qq^\F&=~\phi p_E                        \Tag{\pi_qq^\F=p_E}\\
&=~p_E                                          \Tag{\textbf{[VCF.2]}}\\
&=~p_Ez_Ep_E                                    \Tag{z_Ep_E=\id_E}\\
&=~p_Ez_E\pi_qq^\F                              \Tag{\pi_qq^\F=p_E}\\
&=~p_Ez_E\pi_qq^\F                              \Tag{z_E\pi_q=z_q^\F}
\end{align*}
Moreover, by using \textbf{[VCF.3]}, we can also compute the following:
\begin{align*}
\phi\pi_q\T^\F q&=~\phi\T q                    \Tag{\pi_q\T^\F q=\T q}\\
&=~p_Eqz_M                                      \Tag{\textbf{[VCF.3]}}\\
&=~p_Ez_E\T q                                   \Tag{\text{Naturality of }z}\\
&=~p_Ez_E\pi_q\T^\F q                          \Tag{\T q=\pi_q\T^\F q}\\
&=~p_Ez_q^\F\T^\F q                            \Tag{z_q^\F=z_E\pi_q}
\end{align*}
Thus, $\phi\pi_q=p_Ez_q^\F$. By Lemma~\ref{lemma:vertical-bundle-kernel}, the square diagram is a tangent pullback in $\DB(\X,\TT;E)$, and by \textbf{[VCF.2]}, $\phi$ is linear. Therefore, $\R_\phi$ is well-defined. Moreover, $\R_\phi$ is automatically linear, since the morphisms of $\DB(\X,\TT;E)$ are linear, thus, \textbf{[VC.2]} holds. Moreover, by using that $\R_\phi\iota_q=\phi$ and \textbf{[VCF.1]}, we compute $\iota_q\R_\phi\iota_q=\iota_q\phi=\iota_q$. However, since $\iota_q$ is tangent monic by Lemma~\ref{lemma:iota-tangent-monic}, this implies that $\iota_q\R_\phi=\id_{\V q}$, that is, \textbf{[VC.1]} holds. Finally, we want to show that this correspondence is bijective. Consider a vertical connection $\R$ on $q$. Thus, the vertical connection $\R_{\phi_\R}$ satisfies the equation $\R_{\phi_\R}\iota_q=\phi_\R=\R\iota_q$. However, since $\iota_q$ is monic, this implies that $\R_{\phi_\R}=\R$. Conversely, consider a vertical connection form $\phi$ on $q$. Thus, the morphism $\phi_{\R_\phi}$ is equal to $\R_\phi\iota_q$. However, $\R_\phi\iota_q$, by construction, is equal to $\phi$. Thus, $\phi_{\R_\phi}=\phi$.
\end{proof}

One of the crucial properties of tangent display maps is stability under pullbacks. The next result shows that this also holds for vertical connections.

\begin{proposition}
\label{proposition:pullback-vertical-connections}
If $q\colon E\to M$ is a tangent display map, $\R\colon\T E\to\V q$ a vertical connection on $q$, and $f\colon M'\to M$ is a morphism, the tangent display map $q'\colon E'\to M'$, pullback of $q$ along $f$
\begin{equation*}
\begin{tikzcd}
{E'} & E \\
{M'} & M
\arrow["g", from=1-1, to=1-2]
\arrow["{q'}"', from=1-1, to=2-1]
\arrow["\lrcorner"{anchor=center, pos=0.125}, draw=none, from=1-1, to=2-2]
\arrow["q", from=1-2, to=2-2]
\arrow["f"', from=2-1, to=2-2]
\end{tikzcd}
\end{equation*}
comes with a vertical connection $\R'\colon\T E'\to\V q'$, defined as the unique morphism rendering the following diagram commutative:
\begin{equation*}
\begin{tikzcd}
{\T E'} && {\T E} & {\V q} \\
& {\V q'} & {\T E'} & {\T E} \\
{E'} & {M'} & {\T M'} & {\T M}
\arrow["{\T g}", from=1-1, to=1-3]
\arrow["{\R'}", dashed, from=1-1, to=2-2]
\arrow["{p_{E'}}"', from=1-1, to=3-1]
\arrow["\R", from=1-3, to=1-4]
\arrow["{\iota_q}", from=1-4, to=2-4]
\arrow["{\iota_{q'}}", from=2-2, to=2-3]
\arrow["{\T^\V q'}"', from=2-2, to=3-2]
\arrow["\lrcorner"{anchor=center, pos=0.125}, draw=none, from=2-2, to=3-3]
\arrow["{\T g}", from=2-3, to=2-4]
\arrow["{\T q'}", from=2-3, to=3-3]
\arrow["\lrcorner"{anchor=center, pos=0.125}, draw=none, from=2-3, to=3-4]
\arrow["{\T q}", from=2-4, to=3-4]
\arrow["{q'}"', from=3-1, to=3-2]
\arrow["{z_M}"', from=3-2, to=3-3]
\arrow["{\T f}"', from=3-3, to=3-4]
\end{tikzcd}
\end{equation*}
\end{proposition}
\begin{proof}
For starters, we prove that $\R'$ satisfies \textbf{[VC.1]}. Let us compute the following equalities:
\begin{align*}
\iota_{q'}\R'\iota_{q'}\T g&=~\iota_{q'}\T g\R\iota_q               \Tag{\text{Definition of }\R'}\\
&=~\V_fg\iota_q\R\iota_q                                            \Tag{\V_fg\iota_q=\iota_{q'}\T g}\\
&=~\V_fg\iota_q                                                     \Tag{\textbf{[VC.1]}}\\
&=~\iota_{q'}\T g                                                   \Tag{\V_fg\iota_q=\iota_{q'}\T g}
\end{align*}
Furthermore:
\begin{align*}
\iota_{q'}\R'\T^\V q'&=~\iota_{q'}\T q'p_M                     \Tag{\text{Lemma}~\ref{lemma:vertical-connections}.(a)}\\
&=~\T^\V q'z_Mp_M                                              \Tag{\iota_{q'}\T q'=\T^\V q'z_M}\\
&=~\T^\V q'                                                    \Tag{z_Mp_M=\id_M}
\end{align*}
Thus, by the universal property of the two pullback diagrams, we conclude that $\iota_q'\R'=\id_{\V q'}$. Next, let us prove \textbf{[VC.2]}. First, let us prove that $\R'q'^\V=p_{E'}$:
\begin{align*}
\R'q'^\V g&=~\R'\iota_{q'}p_{E'}g                               \Tag{q'^\V=\iota_{q'}p_{E'}}\\
&=~\R'\iota_{q'}\T gp_E                                         \Tag{\text{Naturality of }p}\\
&=~\T g\R\iota_qp_E                                             \Tag{\text{Definition of }\R'}\\
&=~\T gp_E                                                      \Tag{\textbf{[VC.1]}}\\
&=~p_{E'}g                                                      \Tag{\text{Naturality of }p}
\end{align*}
Moreover, we can also write
\begin{align*}
\R'q'^\V q'=\R'\T^\V q'=p_{E'}q'
\end{align*}
where we used that $q'^\V q'=\T^\V q'$ and again \textbf{[VC.1]}. Thus, by the universal property of the pullback, $\R'q'^\V=p_{E'}$. Finally, let us prove linearity:
\begin{align*}
\R'l_{q'}^\V\T\iota_{q'}\T^2g&=~\R'\iota_{q'}\T gl_E        \Tag{\text{Naturality of $l$ and linearity of }\iota_{q'}}\\
&=~\T g\R\iota_ql_E                                         \Tag{\text{Definition of }\R'}\\
&=~\T gl_E\T\R\T\iota_q                                     \Tag{\textbf{[VC.2]}}\\
&=~l_E\T^2g\T\R\T\iota_q                                    \Tag{\text{Naturality of $l$ and linearity of }\iota_{q'}}\\
&=~l_E\T\R'\T\iota_{q'}\T^2g                                \Tag{\text{Definition of }\R'}
\end{align*}
Moreover, we also compute
\begin{align*}
&\R'l_{q'}^\V\T\T^\V q'=\R'\T^\V q'z_{M'}=p_Eq'z_M=\T q'p_Mz_M=\T q'l_Mp_{\T M}=l_{E'}p_{E'}\T q'
\end{align*}
where we used the naturality of $l$, of $p$, that $lp_\T=pz$, and \textbf{[VC.2]}. Thus, by the universal property of the two tangent pullbacks, we conclude that $\R'l_{q'}^\V=l_{E'}\T\R'$, thus, that \textbf{[VC.2]} holds for $\R'$.
\end{proof}

Thanks to stability under pullbacks, vertical connections can be organized into a tangent fibration.

\begin{proposition}
\label{proposition:fibration-vertical-connections}
The codomain functor $\Pi\colon\VC(\X,\TT)\to(\X,\TT)$ which sends a vertical connection $(q,\R)$ to the codomain $M$ of $q\colon E\to M$ is a tangent fibration.
\end{proposition}
\begin{proof}
This is a direct consequence of the stability under the tangent bundle functor and along tangent pullbacks of tangent display maps and of vertical connections (see~\cite[Proposition~3.11]{cruttwell:tangent-display-maps}).
\end{proof}

In general, tangent display maps are not stable under retration, however, they become stable when the idempotents of the tangent category split~\cite[Corollary~2.25]{cruttwell:tangent-display-maps}. Concretely, this means that, if the following diagram
\begin{equation}
\label{equation:section-retrations-tangent-display}
\begin{tikzcd}
E & {E'} & E \\
M & {M'} & M
\arrow["{r_E}", from=1-1, to=1-2]
\arrow["q"', from=1-1, to=2-1]
\arrow["{s_E}", from=1-2, to=1-3]
\arrow["{q'}", from=1-2, to=2-2]
\arrow["q", from=1-3, to=2-3]
\arrow["{s_M}"', from=2-1, to=2-2]
\arrow["{r_M}"', from=2-2, to=2-3]
\end{tikzcd}
\end{equation}
commutes and $(s_E,r_E)$ and $(s_M,r_M)$ are section-retraction pairs, that is, $s_Er_E=\id_{E'}$ and $s_Mr_M=\id_{M'}$, then, if $q$ is a tangent display map, so is $q'$. The next result shows that, under the condition of tangent display maps to be stable under retration, also vertical connections are too.

\begin{proposition}
\label{proposition:vertical-connections-retraction}
In a tangent category whose tangent display maps are closed under retraction, vertical connections are also closed under retraction. Explicitly, if $\R$ is a vertical connection on a tangent display map $q\colon E\to M$ and $(s_E,r_E)$ and $(s_M,r_M)$ are section-retraction pairs as in Equation~\eqref{equation:section-retrations-tangent-display}, then the tangent display map $q'\=s_Eqr_M\colon E'\to M'$ comes with a vertical connection $\R'\colon\T E'\to\V q'$ defined as follows:
\begin{align*}
&\R'\colon\T E'\xrightarrow{\T s_E}\T E\xrightarrow{\R}\V q\xrightarrow{\V_{r_M}r_E}\V q'
\end{align*}
\end{proposition}
\begin{proof}
To show that \textbf{[VC.1]} holds for $\R'$, let us compute the following:
\begin{align*}
&\iota_{q'}\R'=\iota_{q'}\T s_E\R\V_{r_M}r_E=\V_{s_M}s_E\iota_q\R\V_{r_M}r_E=\V_{s_M}s_E\V_{r_M}r_E=\id_{\V q'}
\end{align*}
To prove \textbf{[VC.2]}, let us start by computing the following:
\begin{align*}
&\R'q'^\V=\T s_E\R\V_{r_M}r_E\iota_{q'}p_{E'}=\T s_E\R\T r_Ep_{E'}=\T s_E\R\iota_qp_Er_E=\T s_Ep_Er_E=p_{E'}s_Er_E=p_{E'}
\end{align*}
where we used that $\R\iota_qp_E=\R q^\V=p_E$. This shows that $\R'$ is a bundle morphism. To prove that $\R'$ is linear, notice that, $\R'$ is composition of linear morphisms, since $\T s_E$, $\R$, and $\V_{r_M}r_E$ are all linear. Thus, $\R'$ is also linear and \textbf{[VC.2]} holds.
\end{proof}

\subsection{Horizontal connections}
\label{subsection:horizontal-connection}
Generally, a vertical connection is insufficient for specifying the horizontal bundle of a connection. In fact, in the absence of negatives as in the general case of a tangent category, the kernel of $\R$ may fail to be in direct sum with the vertical bundle. For this reason, an Ehresmann connection comprises a second piece of information: a horizontal connection. In this section, we introduce this concept in tangent categories and study some of its properties.

\begin{definition}
\label{definition:horizontal-connection}
A \textbf{horizontal connection} on a tangent display map $q\colon E\to M$ consists of a morphism $\H\colon\F q\to\T E$ satisfying the following conditions:
\begin{description}
\item[HC.1] The map $\H$ is a section of the horizontal descent $\pi_q$:
\begin{equation*}
\begin{tikzcd}
{\F q} & {\T E} \\
& {\F q}
\arrow["{\H}", from=1-1, to=1-2]
\arrow[equals, from=1-1, to=2-2]
\arrow["{\pi_q}", from=1-2, to=2-2]
\end{tikzcd}
\end{equation*}

\item[HC.2] The morphism $\H\colon q^\F\to\p_E$ is a linear morphism of differential bundles over $E$:
\begin{equation*}
\begin{tikzcd}
{\F q} & {\T E} \\
E & E
\arrow["\H", from=1-1, to=1-2]
\arrow["{q^\F}"', from=1-1, to=2-1]
\arrow["{p_E}", from=1-2, to=2-2]
\arrow[equals, from=2-1, to=2-2]
\end{tikzcd}\quad
\begin{tikzcd}
{\T\F q} & {\T^2E} \\
{\F q} & {\T E}
\arrow["{\T\H}", from=1-1, to=1-2]
\arrow["{l_q^\F}", from=2-1, to=1-1]
\arrow["\H"', from=2-1, to=2-2]
\arrow["{l_E}"', from=2-2, to=1-2]
\end{tikzcd}
\end{equation*}
\end{description}
\end{definition}

\textbf{[HC.1]} requires that $\H$ is an inclusion of the Finsler bundle $\q^\F\colon\F q\to E$ into the tangent bundle of $E$; \textbf{[HC.2]} requires that this inclusion is linear.
\par We denote by $\HC(\X,\TT)$ the tangent category of pairs $(q,\H)$ formed by a tangent display map $q$ of $(\X,\TT)$ and a horizontal connection $\H$ on $q$, and commutative squares $(f,g)\colon q\to q'$ for morphisms, that is, $qf=gq'$, compatible with the horizontal connections as follows:
\begin{equation*}
\begin{tikzcd}
{\F q} & {\F q'} \\
{\T E} & {\T E'}
\arrow["{\F g}", from=1-1, to=1-2]
\arrow["\H", from=2-1, to=1-1]
\arrow["{\T g}"', from=2-1, to=2-2]
\arrow["{\H'}"', from=2-2, to=1-2]
\end{tikzcd}
\end{equation*}
The tangent bundle functor of $\HC(\X,\TT)$ sends each $(q,\H)$ to $(\T q,\H_\T)$, where
\begin{align*}
&\H_\T\colon\F\T q\xrightarrow{\gamma^\F_q}\T\F q\xrightarrow{\T\H}\T^2E\xrightarrow{c}\T^2E
\end{align*}
and $\gamma^\F_q\colon\T\F q\to\F\T q$ is the strong distributive law of the tangent morphism $(-)^\F$ of Lemma~\ref{lemma:functoriality-F}.

\begin{lemma}
\label{lemma:horizontal-connection}
If \ $\H\colon\F q\to\T E$ is a horizontal connection on a tangent display map $q\colon E\to M$, then the morphism $\H\colon\T^\F q\to\T q$ is a bundle morphism:
\begin{equation*}
\begin{tikzcd}
{\F q} & {\T E} \\
{\T M} & {\T M}
\arrow["\H", from=1-1, to=1-2]
\arrow["{\T^\F q}"', from=1-1, to=2-1]
\arrow["{\T q}", from=1-2, to=2-2]
\arrow[equals, from=2-1, to=2-2]
\end{tikzcd}
\end{equation*}
\end{lemma}
\begin{proof}
Recall that $\pi_q\T^\F q=\T q$, thus, we compute:
\begin{align*}
\H\T q&=~\H\pi_q\T^\F q                                    \Tag{\T q=\pi_q\T^\F q}\\
&=~\T^\F q                                                 \Tag{\textbf{[HC.1]}}
\end{align*}
This concludes the proof.
\end{proof}

Thanks to~\textbf{[HC.1]}, a horizontal connection $\H\colon\F q\to\T E$ is naturally associated with a linear idempotent $\psi_\H\=\pi_q\H\colon\T E\to\T E$ on the tangent bundle, called the horizontal connection form of $\H$.  Just as with the vertical connection form, we formally introduce this concept.

\begin{definition}
\label{definition:horizontal-connection-form}
A \textbf{horizontal connection form} of a tangent display map $q\colon E\to M$ consists of an endomorphism $\psi\colon\T E\to\T E$, subject to the following conditions:
\begin{description}
\item[HCF.1] The following diagram commutes:
\begin{equation*}
\begin{tikzcd}
{\T E} & {\T E} \\
& {\F q}
\arrow["\psi", from=1-1, to=1-2]
\arrow["{\pi_q}"', from=1-1, to=2-2]
\arrow["{\pi_q}", from=1-2, to=2-2]
\end{tikzcd}
\end{equation*}

\item[HCF.2] The morphism $\psi\colon\p_E\to\p_E$ is a linear endomorphism of differential bundles;

\item[HCF.3] The following diagram commutes:
\begin{equation*}
\begin{tikzcd}
{\V q} & {\T E} \\
E & {\T E}
\arrow["{\iota_q}", from=1-1, to=1-2]
\arrow["{q^\V}"', from=1-1, to=2-1]
\arrow["\psi", from=1-2, to=2-2]
\arrow["{z_E}"', from=2-1, to=2-2]
\end{tikzcd}
\end{equation*}
\end{description}
\end{definition}

The next theorem proves that  if the bundle map is a submersion, the information of a horizontal connection is entirely retained in its horizontal connection form.  In particular, to reconstruct a horizontal connection from a horizontal connection form, we make use of the universal property of the pushout diagram of Definition~\ref{definition:submersion}.

\begin{theorem}
\label{theorem:horizontal-connection-form}
There is a bijective correspondence between horizontal connections and horizontal connection forms on a submersion $q\colon E\to M$. Concretely, if $\H\colon\F q\to\T E$ is a horizontal connection on a submersion $q\colon E\to M$, the morphism
\begin{align*}
&\psi_\H\colon\T E\xrightarrow{\pi_q}\F q\xrightarrow{\T E}\T E
\end{align*}
is a horizontal connection form of $q$. Conversely, if $\psi\colon\T E\to\T E$ is a horizontal connection form on a submersion $q\colon E\to M$, the unique morphism $\H_\psi\colon\F q\to\T E$ of $\DB(\X,\TT;E)$ rendering the following diagram
\begin{equation*}
\begin{tikzcd}
{\V q} & {\T E} \\
E & {\F q} \\
&& {\T E}
\arrow["{\iota_q}", from=1-1, to=1-2]
\arrow["{q^\V}"', from=1-1, to=2-1]
\arrow["{\pi_q}", from=1-2, to=2-2]
\arrow["\psi", curve={height=-18pt}, from=1-2, to=3-3]
\arrow["{z_q^\F}"', from=2-1, to=2-2]
\arrow["{z_E}"', curve={height=18pt}, from=2-1, to=3-3]
\arrow["\lrcorner"{anchor=center, pos=0.125, rotate=180}, draw=none, from=2-2, to=1-1]
\arrow["{\H_\psi}", dashed, from=2-2, to=3-3]
\end{tikzcd}
\end{equation*}
commutative is a horizontal connection on $q$.
\end{theorem}
\begin{proof}
For starters, assume that $\H$ is a horizontal connection and let us prove that $\psi_\H$ is in fact a horizontal connection form. For this, by \textbf{[HC.1]}, $\psi_\H\pi_q=\pi_q\H\pi_q=\pi_q$, thus, \textbf{[HCF.1]} holds. To prove \textbf{[HCF.2]}, notice that $\psi_\H$ is composition of two linear morphisms, $\H$ and $\pi_q$, thus, it is also linear. Finally, we compute the following:
\begin{align*}
\iota_q\psi_\H&=~\iota_q\pi_q\H         \Tag{\text{Definition of }\psi_\H}\\
&=~q^\V z_q^\F\H                        \Tag{\iota_q\pi_q=q^\V z_q^\F,\text{Lemma}~\ref{lemma:vertical-bundle-kernel}}\\
&=~q^\V z_E                              \Tag{\H\text{ is linear, thus it preserves the zero}}
\end{align*}
Therefore, \textbf{[HCF.3]} holds. Conversely, assume that $\psi$ is a horizontal connection form and let us prove that $\H_\psi$ is a horizontal connection. For starters, notice that by \textbf{[HCF.3]} and thanks to the linearity of $\psi$ guaranteed by \textbf{[HCF.2]}, $\H_\psi$ is well-defined. Moreover, $\H_\psi$ satisfies \textbf{[HC.2]} since the morphisms of $\DB(\X,\TT;E)$ are linear. To prove \textbf{[HC.1]}, notice that $\pi_q\H_\psi\pi_q=\psi\pi_q=\pi_q$, where we used \textbf{[HCF.1]}. Moreover, $z_q^\F\H_\psi\pi_q=z_E\pi_q=z_q^\F$, thus, by the universal property of the pushout diagram, $\H_\psi\pi_q=\id_{\F q}$, that is, \textbf{[HC.1]} holds. Finally, let us prove that these two operations are inverse to each other. Start with a horizontal connection $\H$. Thus, $\H_{\psi_\H}$ is the unique morphism such that $\pi_q\H_{\psi_\H}=\psi_\H=\pi_q\H$ and since $\pi_q$ is epi, $\H$ must coincides with $\H_{\psi_\H}$. Now, let us consider a horizontal connection form $\psi$. Thus, $\psi_{\H_\psi}=\pi_q\H_\psi$, which, by definition, is just $\psi$.
\end{proof}

Similar to vertical connections, horizontal connections are also stable under pullbacks. To prove this result we first need a technical lemma.

\begin{lemma}
\label{lemma:pullback-horizontal-connections}
Consider a tangent display map $q\colon E\to M$ and the tangent pullback diagram of $q$ along a morphism $f\colon M'\to M$:
\begin{equation*}
\begin{tikzcd}
{E'} & E \\
{M'} & M
\arrow["g", from=1-1, to=1-2]
\arrow["{q'}"', from=1-1, to=2-1]
\arrow["\lrcorner"{anchor=center, pos=0.125}, draw=none, from=1-1, to=2-2]
\arrow["q", from=1-2, to=2-2]
\arrow["f"', from=2-1, to=2-2]
\end{tikzcd}
\end{equation*}
Thus, the following diagram commutes, and it is also a tangent pullback diagram:
\begin{equation*}
\begin{tikzcd}
{\F q'} & {\F q} \\
{\T M'} & {\T M}
\arrow["{\F_fg}", from=1-1, to=1-2]
\arrow["{\T^\F q}"', from=1-1, to=2-1]
\arrow["\lrcorner"{anchor=center, pos=0.125}, draw=none, from=1-1, to=2-2]
\arrow["{\T^\F q}", from=1-2, to=2-2]
\arrow["{\T f}"', from=2-1, to=2-2]
\end{tikzcd}
\end{equation*}
\end{lemma}
\begin{proof}
The diagram commutes since, by definition of $\F_fg$, $\F_fg\T^\F q=\T^\F q\T f$. To prove that the diagram is also a tangent pullback, consider the following diagram:
\begin{equation*}
\begin{tikzcd}
{\F q'} & {\F q} & E \\
{\T M'} & {\T M} & M
\arrow["{\F_fg}", from=1-1, to=1-2]
\arrow["{\T^\F q}"', from=1-1, to=2-1]
\arrow["{q^\F}", from=1-2, to=1-3]
\arrow["{\T^\F q}", from=1-2, to=2-2]
\arrow["\lrcorner"{anchor=center, pos=0.125}, draw=none, from=1-2, to=2-3]
\arrow["q", from=1-3, to=2-3]
\arrow["{\T f}"', from=2-1, to=2-2]
\arrow["{p_M}"', from=2-2, to=2-3]
\end{tikzcd}
\end{equation*}
The right square is a tangent pullback by definition of the Finsler bundle. Notice that the outer diagram can be rewritten as follows:
\begin{equation*}
\begin{tikzcd}
{\F q'} & {E'} & E \\
{\T M'} & {M'} & M
\arrow["{q'^\F}", from=1-1, to=1-2]
\arrow["{\T^\F q}"', from=1-1, to=2-1]
\arrow["\lrcorner"{anchor=center, pos=0.125}, draw=none, from=1-1, to=2-2]
\arrow["g", from=1-2, to=1-3]
\arrow["{q'}", from=1-2, to=2-2]
\arrow["\lrcorner"{anchor=center, pos=0.125}, draw=none, from=1-2, to=2-3]
\arrow["q", from=1-3, to=2-3]
\arrow["{p_{M'}}"', from=2-1, to=2-2]
\arrow["f"', from=2-2, to=2-3]
\end{tikzcd}
\end{equation*}
However, this diagram is two tangent pullback diagrams stacked together. Thus, the outer rectangle of the previous diagram is also a tangent pullback, and by Lemma~\ref{lemma:property-tangent-pullbacks}, so is the left one.
\end{proof}

\begin{proposition}
\label{proposition:pullback-horizontal-connections}
If $q\colon E\to M$ is a tangent display map, $\H\colon\F q\to\T E$ a horizontal connection on $q$, and $f\colon M'\to M$ is a morphism, the tangent display map $q'\colon E'\to M'$, pullback of $q$ along $f$
\begin{equation*}
\begin{tikzcd}
{E'} & E \\
{M'} & M
\arrow["g", from=1-1, to=1-2]
\arrow["{q'}"', from=1-1, to=2-1]
\arrow["\lrcorner"{anchor=center, pos=0.125}, draw=none, from=1-1, to=2-2]
\arrow["q", from=1-2, to=2-2]
\arrow["f"', from=2-1, to=2-2]
\end{tikzcd}
\end{equation*}
comes with a horizontal connection $\H'\colon\F q'\to\T E'$, defined as the unique morphism rendering the following diagram
\begin{equation*}
\begin{tikzcd}
{\F q'} && {\F q} \\
& {\T E'} & {\T E} \\
& {\T M'} & {\T M}
\arrow["{\F_fg}", from=1-1, to=1-3]
\arrow["{\H'}", dashed, from=1-1, to=2-2]
\arrow["{\T^\F q}"', curve={height=18pt}, from=1-1, to=3-2]
\arrow["\H", from=1-3, to=2-3]
\arrow["{\T g}", from=2-2, to=2-3]
\arrow["{\T q'}"', from=2-2, to=3-2]
\arrow["\lrcorner"{anchor=center, pos=0.125}, draw=none, from=2-2, to=3-3]
\arrow["{\T q}", from=2-3, to=3-3]
\arrow["{\T f}"', from=3-2, to=3-3]
\end{tikzcd}
\end{equation*}
commutative.
\end{proposition}
\begin{proof}
To begin, let us prove that $\H'$ satisfies \textbf{[HC.1]}. Let us compute the following:
\begin{align*}
\H'\pi_{q'}\F_fg&=~\H'\T g\pi_q         \Tag{\pi_{q'}\F_fg=\T g\pi_q}\\
&=~\F_fg\H\pi_q                         \Tag{\text{Definition of }\H'}\\
&=~\F_fg                                \Tag{\textbf{[HC.1]}}
\end{align*}
Moreover, we can also compute:
\begin{align*}
\H'\pi_{q'}\T^\F q=\H'\T q'=\T^\F q
\end{align*}
Thus, by the universal property of the tangent pullback diagram of Lemma~\ref{lemma:pullback-horizontal-connections}, $\H'\pi_{q'}$ must be $\id_{\F q'}$, that is, $\H'$ satisfies \textbf{[HC.1]}. To prove \textbf{[HC.2]}, let us start by showing that $\H'p_{E'}=\H'\pi_{q'}q'^\F=q'^\F$, where we used that $p_{E'}=\pi_{q'}q'^\F$. Finally, let us prove linearity by using again the universal property of the tangent pullback diagram of Lemma~\ref{lemma:pullback-horizontal-connections}. We start by computing the following:
\begin{align*}
\H'l_{E'}\T^2g&=~\H'\T gl_E                         \Tag{\text{Naturality of }l}\\
&=~\F_fg\H l_E                                      \Tag{\text{Definition of }\H}\\
&=~\F_fgl_q^\F\T\H                                  \Tag{\textbf{[HC.2]}}\\
&=~l_{q'}^\F\T\F_fg\T\H                             \Tag{\text{Linearity of }\F_fg}\\
&=~l_{q'}\T\H'\T^2g                                 \Tag{\text{Definition of }\H}
\end{align*}
Moreover, we compute:
\begin{align*}
\H'l_{E'}\T^2q'&=~\H'\T q'l_{M'}                    \Tag{\text{Naturality of }l}\\
&=~\T^\F q l_{M'}                                  \Tag{\text{Definition of }\H'}\\
&=~l_{q'}^\F\T\T^\F q                              \Tag{\text{Linearity of }\T^\F q}\\
&=~l_{q'}^\F\T\H'\T^2q'                             \Tag{\text{Definition of }\H'}
\end{align*}
Thus, we conclude that $\H'l_{E'}=l_{q'}^\F\T\H'$, that is, \textbf{[HC.2]} holds.
\end{proof}

Thanks to stability under pullbacks, horizontal connections form a tangent fibration.

\begin{proposition}
\label{proposition:fibration-horizontal-connections}
The codomain functor $\Pi\colon\HC(\X,\TT)\to(\X,\TT)$ which sends a horizontal connection $(q,\H)$ to the codomain $M$ of $q\colon E\to M$ is a tangent fibration.
\end{proposition}
\begin{proof}
This is a direct consequence of the stability under the tangent bundle functor and along tangent pullbacks of tangent display maps and of horizontal connections (see~\cite[Proposition~3.11]{cruttwell:tangent-display-maps}).
\end{proof}

As with vertical connections,  horizontal connections are stable under retraction.

\begin{proposition}
\label{proposition:horizontal-connections-retraction}
Whenenever tangent display maps are closed under retraction, horizontal connections are also closed under retraction. Explicitly, if $\H$ is a horizontal connection on a tangent display map $q\colon E\to M$ and $(s_E,r_E)$ and $(s_M,r_M)$ are section-retraction pairs as in Equation~\eqref{equation:section-retrations-tangent-display}, then the tangent display map $q'\=s_Eqr_M\colon E'\to M$ comes with a horizontal connection $\H'\colon\F q'\to\T E'$ defined as follows:
\begin{align*}
&\H'\colon\F q'\xrightarrow{\F_{s_M}s_E}\F q\xrightarrow{\H}\T E\xrightarrow{\T r_E}\T E'
\end{align*}
\end{proposition}
\begin{proof}
The proof is fairly similar to that of Proposition~\ref{proposition:vertical-connections-retraction}; therefore, we leave it to the reader to complete the details.
\end{proof}

\subsection{Ehresmann connections}
\label{subsection:ehresmann-connection}
A vertical connection together with a horizontal connection, subject to two compatibility conditions specify an Ehresmann connection. In this section, we introduce Ehresmann connections in tangent categories.

\begin{definition}
\label{definition:ehresmann-connection}
An \textbf{Ehresmann connection} of a tangent display map $q\colon E\to M$ is a pair $(\R,\H)$ consisting of a vertical connection $\R\colon\T E\to\V q$ of $q$ and a horizontal connection $\H\colon\T E\to\F q$ of $q$ subject to the following conditions:
\begin{description}
\item[EC.1] The following diagram commutes:
\begin{equation*}
\begin{tikzcd}
{\F q} & {\T E} \\
E & {\V q}
\arrow["\H", from=1-1, to=1-2]
\arrow["{q^\F}"', from=1-1, to=2-1]
\arrow["\R", from=1-2, to=2-2]
\arrow["{z_q^\V}"', from=2-1, to=2-2]
\end{tikzcd}
\end{equation*}

\item[EC.2] The following diagram commutes:
\begin{equation*}
\begin{tikzcd}
{\T E} & {\T_2E} \\
& {\T E}
\arrow["{\<\R\iota_q,\pi_q\H\>}", from=1-1, to=1-2]
\arrow[equals, from=1-1, to=2-2]
\arrow["{s_E}", from=1-2, to=2-2]
\end{tikzcd}
\end{equation*}
\end{description}
\end{definition}

Axiom \textbf{[EC.1]} requires that the image of the horizontal connection $\H$ is in the kernel of the vertical connection $\R$; \textbf{[EC.2]} requires that every tangent vector of $E$ is a sum of a vertical vector and a horizontal vector. Together, \textbf{[EC.1]} and \textbf{[EC.2]} establish that the horizontal and the vertical connections are in direct sum with each other.
\par We denote by $\C(\X,\TT)$ the tangent category of triples $(q;\R,\H)$ formed by a tangent display map $q\colon E\to M$ and by an Ehresmann connection $(\R,\H)$ of $q$. A morphism of $\C(\X,\TT)$ consists of a bundle morphism $(f,g)\colon\q\to\q'$ which is at the same time a morphism of vertical connections and a morphism of horizontal connections. The tangent bundle functor sends a triple $(q;\R,\H)$ to $(\T q;\R_\T,\H_\T)$, where $\R_\T$ and $\H_\T$ are defined as in $\VC(\X,\TT)$ and $\HC(\X,\TT)$, respectively.
\par Vertical and horizontal connections are in bijective correspondence with vertical and horizontal connection forms. Similarly, an Ehresmann connection is equivalent to an Ehresmann connection form. In the next definition, we make this concept precise.

\begin{definition}
\label{definition:ehresmann-connection-form}
An \textbf{Ehresmann connection form} on a tangent display map $q\colon E\to M$ is a pair $(\phi,\psi)$ consisting of a vertical connection form $\phi\colon\T E\to\T E$ and a horizontal connection form $\psi\colon\T E\to\T E$ on $q$, subject to the following conditions:
\begin{description}
\item[ECF.1] The following diagram commutes:
\begin{equation*}
\begin{tikzcd}
{\T E} & {\T E} \\
E & {\T E}
\arrow["\psi", from=1-1, to=1-2]
\arrow["{p_E}"', from=1-1, to=2-1]
\arrow["\phi", from=1-2, to=2-2]
\arrow["{z_E}"', from=2-1, to=2-2]
\end{tikzcd}
\end{equation*}

\item[ECF.2] The following diagram commutes:
\begin{equation*}
\begin{tikzcd}
{\T E} & {\T_2E} \\
& {\T E}
\arrow["{\<\phi,\psi\>}", from=1-1, to=1-2]
\arrow[equals, from=1-1, to=2-2]
\arrow["{s_E}", from=1-2, to=2-2]
\end{tikzcd}
\end{equation*}
\end{description}
\end{definition}

The correspondence between vertical connections and vertical connection forms and and one between horizontal connections and horizontal connection forms extend to a correspondence between Ehresmann connections and Ehresmann connection forms.

\begin{theorem}
\label{theorem:ehresmann-connection-form}
There is a bijective correspondence between Ehresmann connections and Ehresmann connection forms of a submersion $q\colon E\to M$. Explicitly, if $(\R,\H)\colon\F q\to\T E$ is a connection on $q\colon E\to M$, the morphisms $\phi_\R$ and $\psi_\H$ of Theorems~\ref{theorem:vertical-connection-form} and~\ref{theorem:horizontal-connection-form} form a connection form $(\phi_\R,\psi_\H)$ of $q$. Conversely, if $(\phi,\psi)$ is an Ehresmann connection form on $q\colon E\to M$, the morphisms $\R_\phi$ and $\H_\psi$ of Theorems~\ref{theorem:vertical-connection-form} and~\ref{theorem:horizontal-connection-form} define an Ehresmann connection $(\R_\phi,\H_\psi)$ of $q$.
\end{theorem}
\begin{proof}
It is not hard to see that, in the correspondences of Theorems~\ref{theorem:vertical-connection-form} and~\ref{theorem:horizontal-connection-form}, axioms\textbf{[EC.1]} and \textbf{[EC.2]} correspond precisely to \textbf{[ECF.1]} and \textbf{[ECF.2]}, respectively.
\end{proof}

From the stability under pullbacks of vertical and horizontal connections, it follows that Ehresmann connections are also stable under pullbacks.

\begin{proposition}
\label{proposition:pullback-ehresmann-connections}
Ehresmann connections are closed under pullbacks. Concretely, given a tangent display map $q\colon E\to M$ and a morphism $f\colon M'\to M$, the tangent display map $q'\colon E'\to M'$ pullback of $q$ along $f$ carries an Ehresmann connection $(\R',\H')$ provided that $q$ does, where $\R'$ and $\H'$ are the vertical and horizontal connections of Propositions~\ref{proposition:pullback-vertical-connections} and~\ref{lemma:pullback-horizontal-connections}, respectively.
\end{proposition}
\begin{proof}
We leave it to the reader to complete this proof.
\end{proof}

Thanks to stability under pullbacks, Ehresmann connections form a tangent fibration.

\begin{proposition}
\label{proposition:fibration-ehresmann-connections}
The codomain functor $\Pi\colon\C(\X,\TT)\to(\X,\TT)$ which sends a Ehresmann connection $(q,\R,\H)$ to the codomain $M$ of $q\colon E\to M$ is a tangent fibration.
\end{proposition}
\begin{proof}
This is a direct consequence of the stability under the tangent bundle functor and along tangent pullbacks of tangent display maps and of Ehresmann connections.
\end{proof}

\subsection{Ehresmann connections and the fundamental short exact sequence}
\label{subsection:ehresmann-ses}
We may now connect our definition of an Ehresmann connection with one of the standard definitions in differential geometry. We will prove that an Ehresmann connection is precisely a splitting of the fundamental short exact sequence of a submersion, which implies that an Ehresmann connection is a splitting of the tangent bundle into a direct sum of the vertical bundle and a horizontal bundle. In a semi-additive category, a \textbf{splitting} of a short (not necessarily exact) sequence
\begin{align*}
&0\to A\xrightarrow{\iota}B\xrightarrow{\pi}C\to 0
\end{align*}
consists of two morphisms $\sigma\colon C\to B$ and $\rho\colon B\to A$ such that, $\sigma$ is a section of $\pi\colon B\to C$, that is, $\sigma\pi=\id_C$ and $\rho$ is a retraction of $\iota\colon A\to B$, that is, $\iota\rho=\id_A$. Furthermore, $\sigma$ and $\rho$ must be in direct sum with each other, that is, the following diagrams must commute
\begin{equation*}
\begin{tikzcd}
C & B \\
0 & A
\arrow["\rho", from=1-1, to=1-2]
\arrow[from=1-1, to=2-1]
\arrow["\sigma", from=1-2, to=2-2]
\arrow[from=2-1, to=2-2]
\end{tikzcd}\quad
\begin{tikzcd}
B & {B\oplus B} \\
0 & B
\arrow["{\<\rho\iota,\pi\sigma\>}", from=1-1, to=1-2]
\arrow[from=1-1, to=2-1]
\arrow["{+}", from=1-2, to=2-2]
\arrow[from=2-1, to=2-2]
\end{tikzcd}
\end{equation*}
where $+$ denotes the unique morphism $[\id_B,\id_B]$ induced by the universal property of the biproduct $B\oplus B$.

\begin{theorem}
\label{theorem:splitting-of-ses}
Given a tangent display map $q\colon E\to M$, the following statements are equivalent:
\begin{enumerate}
\item $(\R,\H)$ is an Ehresmann connection on $q$;

\item $(\R,\H)$ is a splitting of the fundamental short exact sequence of $q$.
\end{enumerate}
\end{theorem}
\begin{proof}
Consider an Ehresmann connection $(\R,\H)$ on a submersion $q$. By \textbf{[VC.1]} and \textbf{[HC.1]}, $\R$ is a retraction of the inclusion $\iota_q\colon\V q\to\T E$ and $\H$ is a section of the horizontal descent $\pi_q\colon\T E\to\F q$. Moreover, by \textbf{[VC.2]} and \textbf{[HC.2]}, $\R$ and $\H$ are morphisms of $\DB(\X,\TT;E)$. Finally, by \textbf{[EC.1]} and \textbf{[EC.2]}, $\R$ and $\H$ are in direct sum with each other. Conversely, consider a splitting $(\R,\H)$ of the fundamental short sequence of $q$. Since $\R$ is a retraction of $\iota_q$, $\R$ satisfies \textbf{[VC.1]}. Moreover, $\R$ is linear, thus, $\R$ satisfies also \textbf{[VC.2]}. Similarly, since $\H$ is a linear section of $\pi_q$, $\H$ satisfies both \textbf{[HC.1]} and \textbf{[HC.2]}. Finally, by the orthogonality of $\R$ and $\H$, the pair $(\R,\H)$ satisfies \textbf{[EC.1]} and \textbf{[EC.2]}. Thus, $(\R,\H)$ is an Ehresmann connection on $q$.
\end{proof}

Thanks to Theorem~\ref{theorem:splitting-of-ses}, we can now make precise an important point. Altough the fundamental short exact sequence of Definition~\ref{definition:fundamental-ses} is a well-defined sequence for every tangent display map, only for submersions is it exact. However, if a tangent display map $q$ has an Ehresmann connection, $q$ becomes a submersion, automatically.

\begin{corollary}
\label{corollary:connecton-implies-submersion}
If a tangent display map $q$ admits an Ehresmann connection, then $q$ is a submersion.
\end{corollary}
\begin{proof}
This result is a consequence of a general phenomenon of short sequences. Suppose in fact that $0\to A\to B\to C\to 0$ is a sequence in a semi-additive category. If this sequence admits a splitting, the sequence is necessarily exact. To prove this, let us call the left morphism of the sequence $\iota\colon A\to B$ and the right morphism, $\pi\colon B\to C$. If the sequence admits a splitting $\rho$ and $\sigma$, then, $\iota$ and $\pi$ become respectively a slit mono and a split epi morphism. Therefore, the sequence becomes exact in $A$ and $C$. To prove that it is also exact in $B$, one needs to show that the diagram
\begin{equation*}
\begin{tikzcd}
A & B \\
0 & C
\arrow["\iota", from=1-1, to=1-2]
\arrow[from=1-1, to=2-1]
\arrow["\pi", from=1-2, to=2-2]
\arrow[from=2-1, to=2-2]
\end{tikzcd}
\end{equation*}
is both a pullback and a pushout. We prove that the diagram is a pushout since by Lemma~\ref{lemma:vertical-bundle-kernel}, we already know that the one of the fundamental short sequence of a tangent display map is always a pullback. However, for the general case, the proof can be done dually. Consider two morphisms $\alpha\colon 0\to X$ and $\beta\colon B\to X$ subject to the condition that $\iota\beta=0\alpha$, where $0\colon A\to 0$. However, since $0$ is a zero object, there is a unique map $0_{A,X}\colon A\to X$ that factors through $0$, the zero map. By the orthogonality condition of the splitting, $\rho\iota+\pi\sigma=\id_B$. Therefore, we compute:
\begin{align*}
&\beta=(\rho\iota+\pi\sigma)\beta=\rho\iota\beta+\pi\sigma\beta=0_{B,X}+\pi\sigma\beta=\pi\sigma\beta
\end{align*}
where we used the additivity of $\beta$. Therefore, $\sigma\beta\colon C\to X$ satisfies the following equations. $\pi\sigma\beta=\beta$, $0_{0,C}\sigma\beta=0_{0,X}$, the latter being necessary from the universal property of the object $0$. Now, take any other map $\gamma\colon C\to X$, satisfying the same equations. Thus, $\pi\gamma=\beta=\pi\sigma\beta$. However, since $\pi$ is split epi, this implies that $\gamma=\sigma\beta$. This proves that the diagram is a pushout. By applying this general argument to the fundamental short sequence of a tangent display map and using Theorems~\ref{theorem:submersion} and~\ref{theorem:splitting-of-ses}, we conclude that a tangent display map equipped with an Ehresmann connection is necessarily a submersion.
\end{proof}

A splitting of the fundamental short exact sequence provides a decomposition of the tangent bundle $\p_E\colon\T E\to E$ of $E$ into a vertical component and a horizontal component, given by the Finsler bundle.

\begin{corollary}
\label{corollary:decomposition-connection}
An Ehresmann connection $(\R,\H)$ on a tangent display map $q\colon E\to M$ makes the diagram
\begin{equation*}
\begin{tikzcd}
{\q^\V} && {\q^\F} \\
& {\p_E} \\
{\q^\V} && {\q^\F}
\arrow["\H", from=1-3, to=2-2]
\arrow["\R", from=2-2, to=1-1]
\arrow["{\pi_q}", from=2-2, to=3-3]
\arrow["{\iota_q}", from=3-1, to=2-2]
\end{tikzcd}
\end{equation*}
into a biproduct diagram in the category $\DB(\X,\TT;E)$. In particular, it induces a linear isomorphism
\begin{align*}
&\p_E\cong\q^\V\oplus_E\q^\F
\end{align*}
which preserves the projections (or inclusions) of the biproduct, where $\oplus_E$ denotes a biproduct in $\DB(\X,\TT;E)$. Conversely, a linear isomorphism between $\p_E\cong\q^\V\oplus_E\q^\F$ which preserves the projections (or inclusions) of the biproduct induces an Ehresmann connection on $q$.
\end{corollary}
\begin{proof}
This is a standard result of short exact sequences that extends readily to the context of the fundamental short exact sequence of submersions. We leave it to the reader to check the details of the proof. 
\end{proof}

In a semi-additive category, giving a splitting of a short exact sequence requires specifying two pieces of information: a section $\sigma$ and a retraction $\rho$. Thanks to the splitting lemma, this becomes redundant in an \textbf{additive category}, which is a semi-additive category enriched over the category of Abelian groups. The semi-additive category $\DB(\X,\TT;E)$ of differential bundles over $E$ (\cite[Proposition~4.1]{lucyshyn:connections}\footnote{Notice that the author uses \emph{additive} to mean semi-additive.}) becomes additive when the base tangent category $(\X,\TT)$ admits negatives. This is a folklore result, but for completeness, we report it here.

\begin{lemma}
\label{lemma:additivity-differential-bundles}
In a tangent category with negatives, the category $\DB(\X,\TT;E)$ of differential bundles over an object $E$ and linear morphisms becomes an additive category.
\end{lemma}
\begin{proof}
When the tangent structure admits negatives, using the universal property of the vertical lift of differential bundles, one can induce a negation morphism from the tangent bundle of the total space. Using this construction, one can see that the semi-additive category of differential bundles and linear morphisms over a fixed base object becomes additive.
\end{proof}

Thus, thanks to the exactness of the fundamental short exact sequence of a submersion, in the presence of negatives, giving an Ehresmann connection is equivalent to providing either a vertical or a horizontal connection.

\begin{theorem}
\label{theorem:vertical-horizontal-connections-duality}
In a tangent category with negatives, for each vertical connection $\R$ of a submersion $q$, there exists exactly one horizontal connection $\H_\R$ of $q$ making $(\R,\H_\R)$ into an Ehresmann connection of $q$. Moreover, for each horizontal connection $\H$ of $q$, there exists exactly one vertical connection $\R_\H$ rendering the pair $(\R_\H,\H)$ an Ehresmann connection on $q$.
\end{theorem}
\begin{proof}
By the previous lemma, in the presence of negatives, the category $\DB(\X,\TT;E)$ is additive. We can then employ the splitting lemma. Explicitly, suppose that $\R$ is a vertical connection. Since $q$ is a submersion, we can use the universal property of the pushout diagram of Definition~\ref{definition:submersion} and construct a unique morphism $\H_\R$ rendering the following diagram
\begin{equation*}
\begin{tikzcd}
{\V q} & {\T E} & {\T_2E} \\
E & {\F q} \\
&& {\T E}
\arrow["{\iota_q}", from=1-1, to=1-2]
\arrow["{q^\F}"', from=1-1, to=2-1]
\arrow["{\<\R\iota_qn_E,\id_{\T E}\>}", from=1-2, to=1-3]
\arrow["{\pi_q}", from=1-2, to=2-2]
\arrow["{s_E}", from=1-3, to=3-3]
\arrow["{z_q^\F}"', from=2-1, to=2-2]
\arrow["{z_E}"', curve={height=18pt}, from=2-1, to=3-3]
\arrow["\lrcorner"{anchor=center, pos=0.125, rotate=180}, draw=none, from=2-2, to=1-1]
\arrow["{\H_\R}", dashed, from=2-2, to=3-3]
\end{tikzcd}
\end{equation*}
commutative in $\DB(\X,\TT;E)$, where we used negatives $n_E\colon\T E\to\T E$. From $\pi_q\H_\R=\<\R\iota_qn_E,\id_{\T E}\>s_E$. We immediately deduce that $\<\R\iota_q,\H\pi_q\>s_E=p_Ez_E$, that is, \textbf{[EC.2]}. Using the universal property of the pushout diagram for a submersion, and by the following computation
\begin{align*}
\<\R\iota_qn_E,\id_{\T E}\>s_E\R&=~\<\R\iota_q\R n_{\V q},\R\>s_q^\V       \Tag{\text{Linearity of }\R}\\
&=~\<\R n_{\V q},\R\>s_q^\V                                                \Tag{\textbf{[VC.1]}}\\
&=~p_Ez_q^\F                                                          \Tag{\text{Unitality}}
\end{align*}
we deduce that $\H_\R\R=q^\F z_q^\F$, that is, \textbf{[EC.1]}. \textbf{[HC.2]} is also automatic, since $\H_\R$ is already a morphism of $\DB(\X,\TT;E)$, so, in particular, linear. Finally, to prove \textbf{[HC.1]}, we use the universal property of the pushout diagram again, together with the following computation:
\begin{align*}
\<\R\iota_qn_E,\id_{\T E}\>s_E\pi_q&=~\<\R\iota_q\pi_qn_{\F q},\pi_q\>s_{\F q}      \Tag{\text{Linearity of }\pi_q}\\
&=~\<p_Ez_q^\F n_{\F q},\pi_q\>s_{\F q}                                             \Tag{\iota_q\pi_q=q^\V z_q^\F}\\
&=~p_Ez_q^\F                                                                        \Tag{\text{Unitality}}
\end{align*}
Thus, $(\R,\H_\R)$ is in fact an Ehresmann connection on $q$. Dually, if $\H$ is a horizontal connection, we construct a vertical connection $\R_\H$ as the unique morphism rendering the following diagram:
\begin{equation*}
\begin{tikzcd}
{\T E} && {\T_2E} \\
& {\V q} & {\T E} \\
& E & {\F q}
\arrow["{\<\id_{\T E},\pi_q\H n_E\>}", from=1-1, to=1-3]
\arrow["{\R_\H}", dashed, from=1-1, to=2-2]
\arrow["{p_E}"', curve={height=18pt}, from=1-1, to=3-2]
\arrow["{s_E}", from=1-3, to=2-3]
\arrow["{\iota_q}", from=2-2, to=2-3]
\arrow["{q^\F}"', from=2-2, to=3-2]
\arrow["\lrcorner"{anchor=center, pos=0.125}, draw=none, from=2-2, to=3-3]
\arrow["{\pi_q}", from=2-3, to=3-3]
\arrow["{z_q^\F}"', from=3-2, to=3-3]
\end{tikzcd}
\end{equation*}
Using dual arguments, one can easily show that $\R_\H$ is in fact a vertical connection and that $(\R_\H,\H)$ is in fact an Ehresmann connection. Finally, the uniqueness of $\H_\R$ and $\R_\H$ comes from the universal property of the pullback and pushout diagram.
\end{proof}

From the stability under retraction of vertical and horizontal connections, it also follows that Ehresmann connections are also stable under retraction.

\begin{theorem}
\label{theorem:ehresmann-connections-retraction}
In a tangent category whose tangent display maps are closed under retraction, Ehresmann connections are also closed under retraction.
\end{theorem}
\begin{proof}
We leave it to the reader to complete this proof.
\end{proof}

\subsection{Full connections}
\label{subsection:full-connection}
In~\cite{lucyshyn:connections}, Lucyshyn-Right presented an alternative approach to define an Koszul connection in tangent categories. In this section, we extend this point of view to Ehresmann connections. We introduce the notions of \textit{full vertical} and \textit{full horizontal connections} and prove an equivalence between full connections and Ehresmann connections. We begin by introducing the notion of a full vertical connection.

\begin{definition}
\label{definition:full-vertical-connection}
A \textbf{full vertical connection} on a tangent display map $q\colon E\to M$ consists of a vertical connection $\R\colon\T E\to\V q$ on $q$ subject to the following condition:
\begin{description}
\item[FVC] The following diagram
\begin{equation*}
\begin{tikzcd}
{\T E} & {\V q} \\
{\T M} & M
\arrow["\R", from=1-1, to=1-2]
\arrow["{\T q}"', from=1-1, to=2-1]
\arrow["\lrcorner"{anchor=center, pos=0.125}, draw=none, from=1-1, to=2-2]
\arrow["{\T^\V q}", from=1-2, to=2-2]
\arrow["{p_M}"', from=2-1, to=2-2]
\end{tikzcd}
\end{equation*}
is a tangent pullback in the base tangent category $(\X,\TT)$.
\end{description}
\end{definition}

A full vertical connection suffices to define an Ehresmann connection.

\begin{proposition}
\label{proposition:full-vertical-connection}
If $\R\colon\T E\to\V q$ is a full vertical connection on a tangent display map $q\colon E\to M$, there exists a unique horizontal connection $\H_\R\colon\F q\to\T E$ on $q$ that completes $\R$ into an Ehresmann connection $(\R,\H_\R)$ on $q$.
\end{proposition}
\begin{proof}
Consider a full vertical connection $\R\colon\T E\to\V q$ on $q$. Using \textbf{[FVC]}, we construct a (necessarily unique) morphism $\H_\R\colon\F q\to\T E$ which renders the following diagram
\begin{equation*}
\begin{tikzcd}
{\F q} && E \\
& {\T E} & {\V q} \\
{\T M} & {\T M} & M
\arrow["{q^\F }", from=1-1, to=1-3]
\arrow["{\H_\R}", dashed, from=1-1, to=2-2]
\arrow["{\T^\F q}"', from=1-1, to=3-1]
\arrow["{z_q^\V}", from=1-3, to=2-3]
\arrow["\R", from=2-2, to=2-3]
\arrow["{\T q}"', from=2-2, to=3-2]
\arrow["\lrcorner"{anchor=center, pos=0.125}, draw=none, from=2-2, to=3-3]
\arrow["{\T^\V q}", from=2-3, to=3-3]
\arrow[equals, from=3-1, to=3-2]
\arrow["{p_M}"', from=3-2, to=3-3]
\end{tikzcd}
\end{equation*}
commutative. To prove the existence of such a morphism $\H_\R$, we first compute:
\begin{align*}
&\T^\V q=\T^\V qz_Mp_M=\iota_q\T qp_M=\iota_qp_Eq
\end{align*}
It follows that $q^\F z_q^\V\T^\V q=\T^\F qp_M$:
\begin{align*}
q^\F z_q^\V\T^\V q&=~q^\F z_q^\V\iota_qp_Eq                      \Tag{\T^\V q=\iota_qp_Eq}\\
&=~q^\F z_Ep_Eq                                                  \Tag{z_q^\V\iota_q=z_E}\\
&=~q^\F q                                                    \Tag{z_Ep_E=\id_E}\\
&=~\T^\F qp_M                                                    \Tag{q^\F q=\T^\F qp_M}
\end{align*}
Now, we shall prove that $\H_\R\colon\F q\to\T E$ is, in fact, a horizontal connection on $q$. To begin, we show that $\H_\R p_E=q^\F $ by computing the following:
\begin{align*}
\H_\R p_E&=~\H_\R\R\iota_qp_E                                   \Tag{p_E=\R\iota_qp_E}\\
&=~q^\F z_q^\V\iota_qp_E                                      \Tag{\H_\R\R=q^\F z_q^\V}\\
&=~q^\F z_Ep_E                                                  \Tag{z_q^\V\iota_q=z_E}\\
&=~q^\F                                                     \Tag{z_Ep_E=\id_E}
\end{align*}
Next, we prove that $\H_\R$ is linear, that is, $\H_\R l_E=l_q^\F\T\H_\R$. First, we compute the following:
\begin{align*}
\H_\R l_E\T\R&=~\H_\R\R l_q^\V                          \Tag{l_E\T\R=\R l_q^\V}\\
&=~q^\F z_q^\V l_q^\V                                 \Tag{\H_\R\R=q^\F z_q^\V}\\
&=~q^\F z_E\T z_q^\V                                      \Tag{z_q^\V l_q^\V=z\T z_q^\V}
\end{align*}
However, we can also compute:
\begin{align*}
l_q^\F\T\H_\R\T\R&=~l_q^\F\T q^\F \T z_q^\V               \Tag{\H_\R\R=q^\F z_q^\V}\\
&=~q^\F z_E\T z_q^\V                                        \Tag{l_q^\F\T q^\F =q^\F z_E}
\end{align*}
Furthermore:
\begin{align*}
\H_\R l_E\T^2q&=~\H_\R\T ql_M                                         \Tag{l_E\T^2q=\T ql_M}\\
&=~\T^\F ql_M                                                   \Tag{\H_\R\T q=\T^\F q}
\end{align*}
and moreover:
\begin{align*}
l_q^\F\T\H_\R\T^2q&=~l_q^\F\T\T^\F q                         \Tag{\H_\R\T q=\T^\F q}\\
&=~\T^\F ql_M                                                   \Tag{l_q^\F\T\T^\F q=\T^\F ql_M}
\end{align*}
Therefore, by invoking \textbf{[FVC]}, we conclude that $\H_\R l_E=l_q^\F\T\H_\R$, that is, that $\H_\R\colon\q^\F\to\p_E$ is a linear morphism. Finally, we prove that $\H_\R\pi_q=\id_{\F q}$. We compute:
\begin{align*}
\H_\R\pi_q\T^\F q&=~\H_\R\T q                                       \Tag{\pi_q\T^\F q=\T q}\\
&=~\T^\F q                                                    \Tag{\H_\R\T q=\T^\F q}
\end{align*}
and also:
\begin{align*}
\H_\R\pi_qq^\F &=~\H_\R p_E                                         \Tag{\pi_qq^\F =p_E}\\
&=~q^\F                                                     \Tag{\H_\R p_E=q^\F }
\end{align*}
Therefore, by the universal property of the Finsler bundle $\F q$, $\H_\R\pi_q$ must coincide with the identity on $\F q$. So far, we have shown that $\H_\R\colon\F q\to\T E$ is a horizontal connection on $q$. Furthermore, $\H_\R$ satisfies \textbf{[EC.1]} by construction. It is left to show that $(\R,\H_\R)$ verifies also \textbf{[EC.2]}. Let us compute the following:
\begin{align*}
\<\R\iota_q,\pi_q\H_\R\>s_E\R&=~\<\R\iota_q\R,\pi_q\H_\R\R\>s_q^\V           \Tag{\R\text{ linear thus }s_E\R=(\R\times_E\R)s_q^\V}\\
&=~\<\R,\pi_qq^\F z_q^\V\>s_q^\V                                   \Tag{\iota_q\R=\id_{\V q},\H_\R\R=q^\F z_q^\V}\\
&=~\R\<\id,\T^\V qz_q^\V\>s_q^\V                                     \Tag{\pi_qq^\F =p_E=\R \T^\V q}\\
&=~\R                                                               \Tag{\text{unitality}}
\end{align*}
Moreover:
\begin{align*}
\<\R\iota_q,\pi_q\H_\R\>s_E\T q&=~\<\R\iota_q\T q,\pi_q\H_\R\T q\>s_M         \Tag{s_E\T q=\T_2qs_M}\\
&=~\<\R\T^\V qz_M,\pi_q\T^\F q\>s_M                                         \Tag{\iota_q\T q=\T^\V qz_M,\H_\R\T q=\T^\F q}\\
&=~\T q\<p_Mz_M,\id\>s_M                                                  \Tag{\R\T^\V q=\T qp,\pi_q\T^\F q=\T q}\\
&=~\T q                                                             \Tag{\text{unitality}}
\end{align*}
Therefore, by \textbf{[FVC]}, $\<\R\iota_q,\pi_q\H_\R\>s_E=\id_{\T E}$, thus $(\R,\H_\R)$ is an Ehresmann connection. Uniqueness of $\H_\R$ follows from the universal property of the pullback diagram \textbf{[FVC]}.
\end{proof}

A full vertical connection is a vertical connection which satisfies \textbf{[FVC]}. Proposition~\ref{proposition:full-vertical-connection} shows that every vertical connection subject to this extra condition can be completed uniquely into an Ehresmann connection. Similarly, starting from a horizontal connection $\H$, one would like to construct a vertical connection $\R_\H$ that completes $\H$. The next definition introduces the extra assumption on a horizontal connection required to construct $\R_\H$.

\begin{definition}
\label{definition:full-horizontal-connection}
A \textbf{full horizontal connection} on a tangent display map $q\colon E\to M$ is a horizontal connection $\H\colon\F q\to\T E$ on $q$ subject to the following condition:
\begin{description}
\item[FHC] The following diagram
\begin{equation*}
\begin{tikzcd}
{\0_E} & {\q^\V} \\
{q^\F} & {\p_E}
\arrow["{z_q^\V}", from=1-1, to=1-2]
\arrow["{z_q^\F}"', from=1-1, to=2-1]
\arrow["{\iota_q}", from=1-2, to=2-2]
\arrow["\H"', from=2-1, to=2-2]
\arrow["\lrcorner"{anchor=center, pos=0.125, rotate=180}, draw=none, from=2-2, to=1-1]
\end{tikzcd}
\end{equation*}
is a pushout diagram in the category $\DB(\X,\TT;E)$.
\end{description}
\end{definition}

As we proved for full vertical connections, a full horizontal connection is sufficient to define an Ehresmann connection.

\begin{proposition}
\label{proposition:full-horizontal-connection}
If $\H\colon\T E\to\V q$ is a full horizontal connection on a submersion $q\colon E\to M$, there exists a unique vertical connection $\R_\H\colon\T E\to\V q$ on $q$ that completes $\H$ into an Ehresmann connection $(\R_\H,\H)$ on $q$.
\end{proposition}
\begin{proof}
Consider a full horizontal connection $\H\colon\F q\to\T E$ on $q$. We construct a (necessarily unique) morphism $\R_\H\colon\T E\to\V q$ which renders the following diagram
\begin{equation*}
\begin{tikzcd}
{\0_E} & {\q^\V} \\
{\q^\F} & {\p_E} \\
{\0_E} && {q^\V}
\arrow["{z^\V_q}", from=1-1, to=1-2]
\arrow["{z^\F_q}"', from=1-1, to=2-1]
\arrow["{\iota_q}", from=1-2, to=2-2]
\arrow[curve={height=-12pt}, equals, from=1-2, to=3-3]
\arrow["\H"', from=2-1, to=2-2]
\arrow["{q^\F }"', from=2-1, to=3-1]
\arrow["\lrcorner"{anchor=center, pos=0.125, rotate=180}, draw=none, from=2-2, to=1-1]
\arrow["{\R_\H}", dashed, from=2-2, to=3-3]
\arrow["{z^\V_q}"', from=3-1, to=3-3]
\end{tikzcd}
\end{equation*}
commutative. To prove the existence of such a morphism, notice that, since $\0_E$ is a zero object in $\DB(\X,\TT;E)$, that is, it is both initial and terminal, there must exist a unique morphism $\0_E\to\0_E$. Thus, $z_q^\F q^\F$ coincides with the identity on $\0_E$. Now, we shall prove that $\R_\H$ is in fact a vertical connection on $q$. Since the pushout diagram is in the category $\DB(\X,\TT;E)$ of differential bundles over $E$ and linear morphisms, $\R_\H\colon\p_E\to\q^\V$ is automatically a linear morphism, thus, $\R_\H$ satisfies \textbf{[VC.2]}. Furthermore, by construction, $\iota_q\R_\H=\id_{\V q}$, thus, $\R_\H$ verifies \textbf{[VC.1]} as well. Therefore, $\R_\H$ is a vertical connection on $q$. Finally, we prove that $(\R_\H,\H)$ is an Ehresmann connection. To this end, it is only left to prove \textbf{[EC.2]}, since \textbf{[EC.1]} holds by construction. Let us compute the following:
\begin{align*}
\H\<\R_\H\iota_q,\pi_q\H\>s&=~\<q^\F z_q^\V\iota_q,\H\>s_E                 \Tag{\H\R_\H=q^\F z_q^\V,\H\pi_q=\id_{\F q}}\\
&=~\H\<p_Ez_E,\id\>s_E                                                        \Tag{q^\F =\H p_E,z_q^\V\iota_q=z_E}\\
&=~\H                                                                   \Tag{\text{unitality}}
\end{align*}
Moreover:
\begin{align*}
\iota_q\<\R_\H\iota_q,\pi_q\H\>s_E&=~\<\iota_q,\iota_qp_Ez_q^\F\H\>s_E     \Tag{\iota_q\R_\H=\id,\iota_q\pi_q=\iota_qp_Ez_q^\F}\\
&=~\iota_q\<\id,p_Ez_E\>s_E                                               \Tag{\H\text{ linear thus }z_q^\F\H=z}\\
&=~\iota_q                                                          \Tag{\text{unitality}}
\end{align*}
By invoking \textbf{[FHC]}, we conclude that $\<\R_\H\iota_q,\pi_q\H\>s_E=\id_{\T E}$, that is, $(\R_\H,\H)$ satisfies \textbf{[EC.2]} and it forms a Ehresmann connection. Uniqueness of $\R_\H$ follows from the universal property of the pushout diagram \textbf{[FHC]}.
\end{proof}

\begin{remark}
\label{remark:submersions-instead-of-tangent-display}
In Proposition~\ref{proposition:full-vertical-connection}, the map $q$ is assumed to be only tangent display. This is not sufficient for Proposition~\ref{proposition:full-horizontal-connection}, in which $q$ is in fact assumed to be a submersion, since we made use of the pushout of Equation~\eqref{equation:submersion-diagram} to construct $\R_\H$.
\end{remark}

Propositions~\ref{proposition:full-vertical-connection} and~\ref{proposition:full-horizontal-connection} show that a full vertical connection or a full horizontal connection suffices to construct an Ehresmann connection. The next result shows that the converse is also true, that is, the vertical and the horizontal components of an Ehresmann connection are always full.

\begin{proposition}
\label{proposition:full-connection}
If $(\R,\H)$ is an Ehresmann connection on a tangent display map $q\colon E\to M$, then $\R$ is a full vertical connection and $\H$ is a full horizontal connection.
\end{proposition}
\begin{proof}
Consider an Ehresmann connection $(\R,\H)$. We will prove that $\R$ is full by induction. First, we show that $\R$ verifies \textbf{[FVC]}. Consider two morphisms $f\colon X\to\T M$ and $g\colon X\to\V q$ making the following diagram
\begin{equation*}
\begin{tikzcd}
X \\
& {\T E} & {\V q} \\
& {\T M} & M
\arrow["g", curve={height=-12pt}, from=1-1, to=2-3]
\arrow["f"', curve={height=12pt}, from=1-1, to=3-2]
\arrow["\R", from=2-2, to=2-3]
\arrow["{\T q}"', from=2-2, to=3-2]
\arrow["{\T^\V q}", from=2-3, to=3-3]
\arrow["{p_M}"', from=3-2, to=3-3]
\end{tikzcd}
\end{equation*}
commutative. We may denote by $\tilde f\colon X\to\F q$ the unique morphism rendering the following diagram
\begin{equation*}
\begin{tikzcd}
X && {\V q} \\
& {\F q} & E \\
& {\T M} & M
\arrow["g", from=1-1, to=1-3]
\arrow["{\tilde f}", dashed, from=1-1, to=2-2]
\arrow["f"', curve={height=12pt}, from=1-1, to=3-2]
\arrow["{q^\V}", from=1-3, to=2-3]
\arrow["{\T^\V q}", shift left, curve={height=-24pt}, from=1-3, to=3-3]
\arrow["{q^\F}", from=2-2, to=2-3]
\arrow["{\T^\F q}"', from=2-2, to=3-2]
\arrow["\lrcorner"{anchor=center, pos=0.125}, draw=none, from=2-2, to=3-3]
\arrow["q", from=2-3, to=3-3]
\arrow["{p_M}"', from=3-2, to=3-3]
\end{tikzcd}
\end{equation*}
commutative. Now, we define another morphism as follows:
\begin{align*}
&h\colon X\xrightarrow{\<g\iota_q,\tilde f\H\>}\T_2E\xrightarrow{s_E}\T E
\end{align*}
We want to prove that $h$ is the unique morphism satisfying the following equations:
\begin{equation*}
h\R=g\qquad
h\T q=f
\end{equation*}
Let us start by showing that $h$ is unique. Consider a morphism $h'$ satisfying the same equations, namely, $h'\R=g$ and $h'\T q=f$. Thus:
\begin{align*}
h'\pi_q\T^\F q&=~h'\T q                                   \Tag{\pi_q\T^\F q=\T q}\\
&=~f                                                    \Tag{h'\T q=f}
\end{align*}
Moreover:
\begin{align*}
h'\pi_qq^\F &=~h'\R\iota_qp_E                             \Tag{\pi_qq^\F =p_E=\R\iota_qp_E}\\
&=~g\iota_qp_E                                            \Tag{h'\R=g}
\end{align*}
Thus, by the universal property of the Finsler bundle, $h'\pi_q=\tilde f$. Thus:
\begin{align*}
h&=~\<g\iota_q,\tilde f\H\>s_E                           \Tag{h=\<g\iota_q,\tilde f\H\>s_E}\\
&=~h'\<\R\iota_q,\pi_q\H\>s_E                             \Tag{g=h'\R,\tilde f=h'\pi_q}\\
&=~h'                                                   \Tag{\textbf{[EC.2]}}
\end{align*}
This proves the uniqueness of $h$. Now, we prove that $h$ satisfies the desired equations:
\begin{align*}
h\T q&=~\<g\iota_q,\tilde f\H\>s_E\T q                        \Tag{h=\<g\iota_q,\tilde f\H\>s_E}\\
&=~\<g\iota_q\T q,\tilde f\H\T q\>s_M                         \Tag{s_E\T q=\T_2qs_M}\\
&=~\<g\T^\V qz_M,\tilde f\T^\F q\>s_M                               \Tag{\iota_q\T q=\T^\V qz_M,\H\T q=\T^\F q}\\
&=~f\<p_Mz_M,\id\>s_M                                              \Tag{g\T^\V q=fp_M,\tilde f\T^\F q=f}\\
&=~f                                                        \Tag{\text{unitality}}
\end{align*}
Furthermore:
\begin{align*}
h\R&=~\<g\iota_q,\tilde f\H\>s_E\R                            \Tag{h=\<g\iota_q,\tilde f\H\>s_E}\\
&=~\<g\iota_q\R,\tilde f\H\R\>s_q^\V                      \Tag{\R\text{ linear thus }s_E\R=(\R\times_E\R)s_q^\V}\\
&=~\<g,\tilde fq^\F z_q^\V\>s_q^\V                      \Tag{\iota_q\R=\id_{\V q},\H\R=q^\F z_q^\V}\\
&=~g\<\id,q^\V z_q^\V\>s_q^\V                         \Tag{\tilde fq^\F =gq^\V}\\
&=~g                                                        \Tag{\text{unitality}}
\end{align*}
This proves that the diagram of \textbf{[FVC]} is a pullback. To prove that this is a tangent pullback, consider an integer $n>0$.  If $(\R,\H)$ is an Ehresmann connection on $q$, $\T^\C(\R,\H)$ is an Ehresmann connection on $\T q$, where $\T^\C$ denotes the tangent bundle functor of the tangent category $\C(\X,\TT)$ of Ehresmann connections of $(\X,\TT)$. Thus, by induction, ${\T^\C}^n(\R,\H)$ is an Ehresmann connection on $\T^nq$. The vertical connection part of this Ehresmann connection can be written as $\gamma_n\T^n\R$, where $\gamma_n$ is the isomorphism:
\begin{align*}
&\gamma_n=c_{\T^n}\T c_{\T^{n-1}}\.{\dots}\.\T^kc_{\T^{n-k}}\.{\dots}\.\T^nc
\end{align*}
We have already proven that the vertical connection $\R$ of an Ehresmann connection $(\R,\H)$ renders the diagram of \textbf{[FVC]} into a pullback. Applying the same argument for ${\T^\C}^n(\R,\H)$, we conclude that the following diagram
\begin{equation*}
\begin{tikzcd}
{\T^{n+1}E} & {\V\T^nq} \\
{\T^{n+1}M} & {\T^nM}
\arrow["{\gamma_n\T^n\R}", from=1-1, to=1-2]
\arrow["{\T^{n+1}q}"', from=1-1, to=2-1]
\arrow["\lrcorner"{anchor=center, pos=0.125}, draw=none, from=1-1, to=2-2]
\arrow["{\iota_{\T^nq}p_{\T^n E}\T^nq}", from=1-2, to=2-2]
\arrow["{p_{\T^n M}}"', from=2-1, to=2-2]
\end{tikzcd}
\end{equation*}
is also a pullback. However, by definition of $\gamma_n$ and by the properties of the canonical flip, $p_{\T^n}=\gamma_n\T^np$. Furthermore, using the invertibility of $\gamma_n$ and $\gamma_{n-1}$, the following diagram
\begin{equation*}
\begin{tikzcd}
{\T^{n+1}E} & {\T^{n+1}E} & {\V\T^nq} & {\T^n\V q} \\
{\T^{n+1}M} & {\T^{n+1}M} & {\T^nM} & {\T^nM}
\arrow["{\gamma_n}", from=1-1, to=1-2]
\arrow["{\T^{n+1}q}"', from=1-1, to=2-1]
\arrow["{\T^n\R}", from=1-2, to=1-3]
\arrow["{\T^{n+1}q}"', from=1-2, to=2-2]
\arrow[from=1-3, to=1-4]
\arrow["{\iota_{\T^nq}p_{\T^n E}\T^nq}", from=1-3, to=2-3]
\arrow["{\T^n(\iota_qp_Eq)}", from=1-4, to=2-4]
\arrow["{\gamma_n}"', from=2-1, to=2-2]
\arrow["{\T^np_M}"', from=2-2, to=2-3]
\arrow["{\gamma_{n-1}}"', from=2-3, to=2-4]
\end{tikzcd}
\end{equation*}
is also a pullback. Therefore, the central square diagram must also be a pullback, that is, $\R$ satisfies \textbf{[FVC]}. The final step is to prove that $\H$ is a full horizontal connection, that is, $\H$ satisfies \textbf{[FHC]}. Consider a differential bundle $\q'\colon X\to E$ over $E$ and two linear morphisms of differential bundles $f\colon\q^\F\to\q'$ and $g\colon\q^\V\to\q'$, making the following diagram
\begin{equation*}
\begin{tikzcd}
{\0_E} & {q^\V} \\
{q^\F} & {\p_E} \\
&& {\q'}
\arrow["{z_q^\V}", from=1-1, to=1-2]
\arrow["{z_q^\F}"', from=1-1, to=2-1]
\arrow["{\iota_q}", from=1-2, to=2-2]
\arrow["g", curve={height=-18pt}, from=1-2, to=3-3]
\arrow["\H"', from=2-1, to=2-2]
\arrow["f"', curve={height=12pt}, from=2-1, to=3-3]
\end{tikzcd}
\end{equation*}
commutative. We may define the following morphism:
\begin{align*}
h\colon&\T E\xrightarrow{\<\R g,\pi_qf\>}X_2\xrightarrow{s_{q'}}X
\end{align*}
We leave it to the reader to prove that $h$ is a linear morphism of differential bundles. We shall prove that $h$ is the unique morphism satisfying $\H h=f$ and $\iota_qh=g$. We begin by proving uniqueness. Consider a morphism $h'\colon\T E\to X$ which satisfies the same equations, that is, $\H h'=f$ and $\iota_qh'=g$. We compute:
\begin{align*}
h&=~\<\R g,\pi_qf\>s_{q'}                               \Tag{h=\<\R g,\pi_qf\>s_{q'}}\\
&=~\<\R\iota_qh',\pi_q\H h'\>s_{q'}                     \Tag{g=\iota_qh',f=\H h'}\\
&=~\<\R\iota_q,\pi_q\H\>sh'                             \Tag{h'\text{ linear thus }(h'\times_Eh')s_{q'}=sh'}\\
&=~h'                                                   \Tag{\textbf{[EC.2]}}
\end{align*}
Let us now prove that $h$ satisfies the desired equations:
\begin{align*}
\H h&=~\H\<\R g,\pi_qf\>s_{q'}                         \Tag{h=\<\R g,\pi_qf\>s_{q'}}\\
&=~\<q^\F z_q^\V g,f\>s_{q'}                           \Tag{\H\R=q^\F z_q^\V,\H\pi_q=\id}\\
&=~\<q^\F z_q^\F,\id\>(f\times_Ef)s_{q'}              \Tag{z_q^\V g=z_q^\F f}\\
&=~\<q^\F z_q^\F,\id\>s_{\F q}f                       \Tag{f\text{ linear thus }(f\times_Ef)s_{q'}=s_{\F q}f}\\
&=~f                                                    \Tag{\text{unitality}}
\end{align*}
Moreover:
\begin{align*}
\iota_qh&=~\iota_q\<\R g,\pi_qf\>s_{q'}                 \Tag{h=\<\R g,\pi_qf\>s_{q'}}\\
&=~\<g,q^\V z_q^\F f\>s_{q'}                        \Tag{\iota_q\R=\id_{\V q},\iota_q\pi_q=q^\V z_q^\F}\\
&=~\<g,q^\V z_q^\V g\>s_{q'}                        \Tag{z_q^\F f=z\pi_qf=z_q^\V g}\\
&=~\<\id,q^\V z_q^\V\>s_q^\V g                    \Tag{g\text{ linear thus }(g\times_Eg)s_{q'}=s_q^\V g}\\
&=~g                                                    \Tag{\text{unitality}}
\end{align*}
This concludes the proof.
\end{proof}

We previosuly showed that full vertical and horizontal connections can always be completed uniquely into an Ehresmann connection. We also proved that the vertical and horizontal components of a Ehresmann connections are always full. This allows us to put all the pieces together into the following result:

\begin{theorem}
\label{theorem:full-connections}
Given a tangent display map $q\colon E\to M$, the following are equivalent:
\begin{enumerate}
\item $\R\colon\T E\to\V q$ is a full vertical connection on $q$ and $q$ is a submersion;

\item $\R$ is the vertical component of an Ehresmann connection $(\R,\H)$ on $q$.
\end{enumerate}
Furthermore, the following are also equivalent:
\begin{enumerate}
\item $\H\colon\F q\to\T W$ is a full horizontal connection on $q$ and $q$ is a submersion;

\item $\H$ is the horizontal component of an Ehresmann connection $(\R,\H)$ on $q$.
\end{enumerate}
\end{theorem}

\subsection{Abstract connections}
\label{subsection:abstract-connection}
In the previous section, we showed that full vertical connections, full horizontal connections, and Ehresmann connections are in fact equivalent concepts. In this section, we introduce another new perspective: abstract connections.

\par Our approach so far was to introduce the concepts of vertical and Finsler bundle of a tangent display map, and later used them to define vertical and horizontal connections. Notice that both the vertical and the Finsler bundle are only specified up to a unique isomorphism, since they are constructed using the universal properties of tangent pullbacks. Therefore, choosing a vertical connection or a horizontal connection on a tangent display map $q$ implicitly requires making a choice of either the vertical or the Finsler bundle of $q$.

\par We then showed that a vertical connection defines an idempotent, the \emph{vertical connection form}, and that a horizontal connection is naturally associated with another idempotent, the \emph{horizontal connection form}. The vertical connection form $\phi_\R$ of a vertical connection $\R$ splits on the vertical bundle, and the horizontal connection form $\phi_\H$ of a horizontal connection $\H$ splits on the Finsler bundle.

In this section, we introduce the notions of \emph{abstract} vertical connections and \emph{abstract} horizontal connections. An abstract vertical connection on a \emph{map} $q\colon E\to M$ consists of a linear idempotent on the tangent bundle of $E$, satisfying an extra universal property. Similarly, an abstract horizontal connection on $q$ consists of another linear idempotent of $\p_E$, satisfying another universal property. Crucially, no mention of the vertical or the Finsler bundle is made in defining these concepts, nor is the map $q$ required to be tangent display. A \emph{splitting} of an abstract vertical connection gives both a choice of the vertical bundle of $q$ and a vertical connection on it. A \emph{splitting} of an abstract horizontal connection gives both a choice of the Finsler bundle of $q$ and a horizontal connection on it.

\par The universal properties satisfied by an abstract vertical connection and an abstract horizontal connection directly link to the axioms \textbf{[FVC]} and \textbf{[FHC]} of full connections. To begin, we define abstract vertical connections.

\begin{definition}
\label{definition:abstract-vertical-connection}
An \textbf{abstract vertical connection} on a map $q\colon E\to M$ consists of a morphism $\phi\colon\T E\to\T E$ subject to the following conditions:
\begin{description}
\item[AVC.1] $\phi$ is an idempotent, that is, $\phi\phi=\phi$;

\item[AVC.2] $\phi\colon\p_E\to\p_E$ is a linear morphism of differential bundles, that is, $\phi p_E=p_E$ and $\phi l_E=l_E\T\phi$;

\item[AVC.3] The following diagram
\begin{equation*}
\begin{tikzcd}
{\T E} && {\T E} \\
{\T M} & M & {\T M}
\arrow["\phi", from=1-1, to=1-3]
\arrow["{\T q}"', from=1-1, to=2-1]
\arrow["\lrcorner"{anchor=center, pos=0.125}, draw=none, from=1-1, to=2-3]
\arrow["{\T q}", from=1-3, to=2-3]
\arrow["{p_M}"', from=2-1, to=2-2]
\arrow["{z_M}"', from=2-2, to=2-3]
\end{tikzcd}
\end{equation*}
commutes and is a tangent pullback diagram in the base tangent category $(\X,\TT)$.
\end{description}
\end{definition}

An idempotent $e\colon A\to A$ in a category $\X$ \emph{splits} when there exists an object $B$ of $\X$ together with a section-retraction pair $(s,r)\colon B\to A$, that is, $s\colon B\to A$ and $r\colon A\to B$ satisfying $sr=\id_B$, such that $e=rs$. If $(B,s,r)$ and $(B',s',r')$ are two splittings of an idempotent $e$, there exists a unique isomorphism $\varphi\colon B\to B'$ satisfying $r\varphi=r'$.

We aim to show a correspondence between \emph{linear} splittings of an abstract vertical connection and a choice of the vertical bundle together with a full vertical connection on it. First, let us define the notion of linear splittings.

\begin{definition}
\label{definition:linear-splittings}
A \textbf{linear idempotent} in a tangent category $(\X,\TT)$ on an object $M$ of $(\X,\TT)$ consists of a morphism $\phi\colon\T E\to\T E$, compatible with the tangent bundles, that is, $\phi p_E=p_E$, which is an idempotent in the category $\DB(\X,\TT;M)$ of differential bundles over $M$. A \textbf{linear splitting} of a linear idempotent $\phi$ consists of a splitting of $\phi$ in $\DB(\X,\TT;M)$. A linear idempotent \textbf{splits} provided there exists at least one linear splitting of it.
\end{definition}

\begin{definition}
\label{definition:split-abstract-vertical-connection}
An abstract vertical connection $\phi\colon\T E\to\T E$ on a map $q\colon E\to M$\textbf{splits} provided the linear idempotent $\phi$ of $E$ splits.
\end{definition}

\begin{lemma}
\label{lemma:split-pullback}
Let $\X$ be a category equipped with an endofunctor $\T\colon\X\to\X$. Consider the following commutative diagram
\begin{equation*}
\begin{tikzcd}
{A'} & {B'} & {A'} \\
A & B & A
\arrow["{r'}", from=1-1, to=1-2]
\arrow["f"', from=1-1, to=2-1]
\arrow["{s'}", from=1-2, to=1-3]
\arrow["g"', from=1-2, to=2-2]
\arrow["f", from=1-3, to=2-3]
\arrow["r"', from=2-1, to=2-2]
\arrow["s"', from=2-2, to=2-3]
\end{tikzcd}
\end{equation*}
of $\X$, where $(s,r)\colon A\to B$ and $(s',r')\colon A\to B'$ are two section-retraction pairs. Then, the outer square is a $\T$-pullback if and only if the left and the right squares are both $\T$-pullbacks.
\end{lemma}
\begin{proof}
This is a standard categorical argument, and left to the reader.  
\end{proof}

\begin{lemma}
\label{lemma:vertical-bundle-from-split-abstract-connections}
If $\phi\colon\T E\to\T E$ is a split abstract vertical connection on a map $q\colon E\to M$ and $(\q_\phi^\V,\iota_\phi,\R_\phi)$ is a splitting of $\phi$, then the differential bundle $\q_\phi^\V\colon\V_\phi q\to E$ is a choice of the vertical bundle of $q$.
\end{lemma}
\begin{proof}
Consider a split abstract vertical connection $\phi$ on $q$ and a linear splitting $(\q_\phi^\V,\iota_\phi,\R_\phi)$ of $\phi$. Moreover, consider the following diagram:
\begin{equation}
\label{equation:splitting-vertical-connections}
\begin{tikzcd}
{\T E} & {\V_\phi q} & {\T E} \\
{\T M} & M & {\T M}
\arrow["{\R_\phi}", from=1-1, to=1-2]
\arrow["\phi", shift left, curve={height=-24pt}, from=1-1, to=1-3]
\arrow["{\T q}"', from=1-1, to=2-1]
\arrow["{\iota_\phi}", from=1-2, to=1-3]
\arrow["{\iota_\phi p_Eq}"', from=1-2, to=2-2]
\arrow["{\T q}", from=1-3, to=2-3]
\arrow["{p_M}"', from=2-1, to=2-2]
\arrow["{z_M}"', from=2-2, to=2-3]
\end{tikzcd}
\end{equation}
First, we prove that this diagram commutes. By \textbf{[AVC.3]}, the outer square commutes. Moreover, we compute:
\begin{align*}
\T qp_M&=~p_Eq                                                  \Tag{\T qp_M=p_Eq}\\
&=~\phi p_Eq                                                      \Tag{p_E=\phi p_E}\\
&=~\R_\phi\iota_\phi p_Eq                                         \Tag{\phi=\R_\phi\iota_\phi}
\end{align*}
To prove that the right square commutes, we compute:
\begin{align*}
&\R_\phi\iota_\phi\T q=\T qp_Mz_M=\R_\phi\iota_\phi p_Eqz_M
\end{align*}
Since $\iota_\phi\R_\phi=\id_{\V_\phi q}$ thus, by applying $\iota_\phi$ on both sides we obtain that, $\iota_\phi\T q=\iota_\phi p_Eqz_M$, that is, the right square commutes.

Since by \textbf{[AVC.3]}, the outer square is a tangent pullback, by Lemma~\ref{lemma:split-pullback}, so are the two inner square diagrams. In particular, the left square diagram is a tangent pullback. Therefore, by using the same construction as in Section~\ref{subsection:vertical-bundle}, we can equip the map $\iota_\phi p_E\colon\V_\phi q\to E$ with the structure of a differential bundle. Our goal is to show that this differential bundle, that represents a choice of the vertical bundle for $q$, coincides with the differential bundle $\q^\V_\phi\colon\V_\phi q\to E$, that comes directly from the linear splitting of $\phi$. Notice that $\iota_\phi\colon\q^\V\to\p_E$ is a linear tangent monic (that is $\T^n\iota_\phi$ is monic, since $\iota_\phi$ is a section).

We denote by $\q^\V\colon\V q=\V_\phi q\to E$ this differential bundle to distinguish it from the other one, which is denoted by $\q_\phi^\V$, instead. For starters, since $\R_\phi$ is a morphism of bundles, the underlying projection $q_\phi^\V$ of $\q^\V_\phi$ must satisfies $\R_\phi q^\V_\phi=p_E$. However, $p_E=\phi p_E=\R_\phi\iota_\phi p_E$, therefore, $\R_\phi q^\V_\phi=\R_\phi\iota_\phi p_E$, which implies that $q_\phi^\V=\iota_\phi p_E$, since $\R_\phi$ is epi being a retraction. Next, let $z_\phi^\V$, $s_\phi^\V$, and $l_\phi^\V$ denote the structural morphisms of $\q_\phi^\V$, respectively and let $z_q^\V$, $s_q^\V$, and $l_q^\V$ denote the ones of $\q^\V$, respectively. Since $\iota_\phi$ is linear, it satisfies the following equations:
\begin{align*}
&z_\phi^\V\iota_\phi=z_E=z_q^\V\iota_\phi   &&s_\phi^\V\iota_\phi=(\iota_\phi\times_E\iota_\phi)s_E=s_q^\V\iota_\phi    &&l_\phi^\V\T\iota_q=\iota_ql_E=l_q^\V\T\iota_q
\end{align*}
However, since $\iota_\phi$ is tangent monic, these equations imply that the differential structures of $\q_\phi^\V$ and of $\q^\V$ must coincide.
\end{proof}

\begin{theorem}
\label{theorem:abstract-vertical-connection}
For a tangent display map $q\colon E\to M$ equipped with a morphism $\phi\colon\T E\to\T E$, the following are equivalent:
\begin{enumerate}
\item $\phi$ is a split abstract vertical connection of $q$;

\item There is a full vertical connection $\R$ of $q$ whose connection form is $\phi$;

\item There is a vertical connection $\R$ of $q$ which is the vertical component of an Ehresmann connection $(\R,\H)$ of $q$ and whose connection form is $\phi$.
\end{enumerate}
\end{theorem}
\begin{proof}
The equivalence between $[2]$ and $[3]$ was already proven by Theorem~\ref{theorem:full-connections}. It is only left to prove the equivalence between $[1]$ and $[2]$. We begin by considering a split abstract vertical connection $\phi\colon\T E\to\T E$ on a map $q\colon E\to M$, with  section-retraction pair $(\iota_\phi,\R_\phi)\colon\q_\phi^\V\to\p_E$. By Lemma~\ref{lemma:vertical-bundle-from-split-abstract-connections}, the differential bundle $\q_\phi^\V\colon\V_\phi q\to E$ is a choice of the vertical bundle of $q$. Consider the diagram of Equation~\eqref{equation:splitting-vertical-connections}. By \textbf{[AVC.3]}, the outer square is a tangent pullback and $(\iota_\phi,\R_\phi)$ and $(z_M,p_M)$ are section-retraction pairs. Therefore, by Lemma~\ref{lemma:split-pullback}, both the left and the right squares are tangent pullbacks.

Furthermore, since the left square of the above diagram is a tangent pullback, $\R_\phi$ satisfies \textbf{[FVC]}. Thus, to prove that $\R_\phi$ is a full vertical connection, it is only left to prove that $\R_\phi$ is, in fact, a vertical connection. To this end, we harness the correspondence of Theorem~\ref{theorem:vertical-connection-form} between vertical connections and vertical connection forms and show that $\phi$ is, in fact, a vertical connection form. \textbf{[VCF.1]} is a direct consequence of $\iota_\phi$ being the section in the splitting of $\phi$, that is, $\iota_\phi\phi=\iota_\phi\R\iota_\phi=\iota_\phi$; \textbf{[VCF.2]} corresponds to \textbf{[AVC.2]} and \textbf{[AVC.3]} implies \textbf{[VCF.3]}. Thus, $\phi$ is a vertical connection form and therefore, $\R_\phi$ is a full vertical connection, as expected.

Conversely, suppose that $\R\colon\T E\to\V q$ is a full vertical connection on $q$, then, by Theorem~\ref{theorem:vertical-connection-form}, the associated vertical connection form $\phi_\phi\=\R\iota_q$ is a linear idempotent $\phi_\R\colon\p_E\to\p_E$ of $E$. Furthermore, the following diagram
\begin{equation*}
\begin{tikzcd}
{\T E} & {\V q} & {\T E} \\
{\T M} & M & {\T M}
\arrow["\R", from=1-1, to=1-2]
\arrow["{\phi_\R}", shift left, curve={height=-24pt}, from=1-1, to=1-3]
\arrow["{\T q}"', from=1-1, to=2-1]
\arrow["{\iota_q}", from=1-2, to=1-3]
\arrow["{\T^\V q}"', from=1-2, to=2-2]
\arrow["{\T q}", from=1-3, to=2-3]
\arrow["{p_M}"', from=2-1, to=2-2]
\arrow["{z_M}"', from=2-2, to=2-3]
\end{tikzcd}
\end{equation*}
commutes and, by \textbf{[FVC]} and Equation~\eqref{equation:vertical-bundle-pullback} is the composition of two tangent pullback diagrams. Therefore, $\phi_\R$ satisfies both \textbf{[AVC.1]} and \textbf{[AVC.2]}.
\end{proof}

\begin{remark}
\label{remark:abstract-connections-vs-ehresmann}
Theorem~\ref{theorem:abstract-vertical-connection} shows a correspondence between split abstract vertical connections, full vertical connections, and Ehresmann connections. However, this correspondence only works when the base map $q\colon E\to M$ is a tangent display map. In fact, in order to construct the horizontal component $\H$ of the associated Ehresmann connection, we make use of the existence of the Finsler bundle of $q$. A splitting of an abstract vertical connection on an arbitrary map $q$ equips $q$ with a notion of vertical bundle, which is implicit in a full vertical connection. However, it does not provide a notion for the Finsler bundle. It might accidentally happen that $q$ does have a vertical bundle, but it does not admit a Finsler bundle. In that case, the correspondence between split abstract vertical connections and full vertical connections still exists, but full connections might fail to be completed into an Ehresmann connection. In that sense, both full and abstract vertical connections are a generalization of Ehresmann connections.
\end{remark}

Full vertical connections on maps form a tangent category $\FVC(\X,\TT)$, whose objects are pairs $(q,\R)$ formed by a map $q\colon E\to M$ which admits a vertical bundle, that is, the tangent pullback of Equation~\eqref{equation:vertical-bundle-pullback} exists, and a full vertical connection $\R$ on $q$. The morphisms and the tangent structure are defined as in $\VC(\X,\TT)$. Similarly, abstract vertical connections on maps form also a tangent category $\AVC(\X,\TT)$, whose objects are pairs $(q,\phi)$ formed by a map $q\colon E\to M$ together with an abstract vertical connection $\phi$ on $q$. The morphisms are morphisms of bundles that commute with the abstract vertical connections in the obvious way, and the tangent bundle functor sends a pair $(q,\phi)$ to $(\T q,c_E\T\phi c_E)$. Split abstract vertical connections form a tangent sub-category of $\AVC(\X,\TT)$ denoted by $\sAVC(\X,\TT)$.

\begin{proposition}
\label{proposition:equivalence-abstract-full-vertical-connections}
There is an equivalence of tangent categories:
\begin{align*}
&\FVC(\X,\TT)\simeq\sAVC(\X,\TT)
\end{align*}
\end{proposition}
\begin{proof}
Theorem~\ref{theorem:abstract-vertical-connection} gives a correspondence between full vertical connections and split abstract vertical connections. Using this correspondence we may define two functors, one that sends a full vertical connection $(q,\R)$ to the split abstract vertical connection $(q,\phi_\R)$, $\phi_\R$ being the associated vertical connection form of $\R$, and the second, that sends a split abstract vertical connection $(q,\phi)$ with splitting $(\V_\phi q,\iota_\phi,\R_\phi)$ to $(q,\R_\phi)$. Starting from a full vertical connection $(q,\R)$, one constructs the associated abstract vertical connection $(q,\phi_\R)$. By choosing a splitting of $\phi_\R$, we define a full vertical connection $\R_\phi$. However, since splittings are unique up to a unique isomorphism, there must be a (necessarily unique) linear isomorphism $\varphi\colon\V_\phi q\to\V q$ such that $\R_\phi\varphi=\R$. Conversely, given a split abstract vertical connection $(q,\phi)$, by choosing a splitting of $\phi$, one obtains a full vertical connection $\R_\phi$. Thus, the associated connection form defines a split abstract vertical connection which is precisely equal to $\phi$. This proves the existence of an equivalence of categories. It is not hard to see that the tangent structures of $\FVC(\X,\TT)$ and of $\sAVC(\X,\TT)$ correspond via this equivalence.
\end{proof}

We may now define an abstract horizontal connection.

\begin{definition}
\label{definition:abstract-horizontal-connection}
An \textbf{abstract horizontal connection} on a tangent display map $q\colon E\to M$ consists of a morphism $\psi\colon\T E\to\T E$ subject to the following conditions:
\begin{description}
\item[AHC.1] $\psi$ is an idempotent, that is, $\psi\psi=\psi$;

\item[AHC.2] $\psi\colon\p_E\to\p_E$ is a linear morphism of differential bundles, that is, $\psi p_E=p_E$ and $\psi l_E=l_E\T\psi$;

\item[AHC.3] The following diagram
\begin{equation*}
\begin{tikzcd}
{\q^\V} & {\0_E} & {\q^\V} \\
{\p_E} && {\p_E}
\arrow["{q^\V}", from=1-1, to=1-2]
\arrow["{\iota_q}"', from=1-1, to=2-1]
\arrow["{z_q^\V}", from=1-2, to=1-3]
\arrow["{\iota_q}", from=1-3, to=2-3]
\arrow["\psi"', from=2-1, to=2-3]
\arrow["\lrcorner"{anchor=center, pos=0.125, rotate=180}, draw=none, from=2-3, to=1-2]
\end{tikzcd}
\end{equation*}
is a pushout diagram in $\DB(\X,\TT;E)$.
\end{description}
An abstract horizontal connection $\psi\colon\T E\to\T E$ on a tangent display map $q\colon E\to M$ \textbf{splits} provided that the linear idempotent $\psi$ of $E$ splits.
\end{definition}

\begin{remark}
\label{remark:abstract-horizontal-connection}
In defining an abstract horizontal connection, we have already assumed the base map $q$ to be tangent display, in contrast with Definition~\ref{definition:abstract-vertical-connection} in which $q$ was not required to satisfy this condition. The reason for this choice lies in axiom \textbf{[AHC.3]}, in which we explicitly make use of the vertical bundle of $q$. One might only assume the existence of the vertical bundle instead of asking $q$ to be tangent display (see Remark~\ref{remark:vertical-bundle-0-carrability}). However, for simplicity, we have decided to assume that $q$ satisfies the full display condition.
\end{remark}

We would like show that splittings of abstract horizontal connections correspond to a \emph{choice} of the Finsler bundle of $q$ together with a full horizontal connection on $q$. However, in order to fully prove this correspondence, we need to require $q$ to be a submersion. The reason for this assumption lies in the fact that the pushout diagram
\begin{equation*}
\begin{tikzcd}
{\q^\V} & {\p_E} \\
{\p_E} & {\F_\psi q}
\arrow["{q^\V}", from=1-1, to=1-2]
\arrow["{\iota_q}"', from=1-1, to=2-1]
\arrow["{z_\psi^\F}", from=1-2, to=2-2]
\arrow["{\pi_\psi}"', from=2-1, to=2-2]
\arrow["\lrcorner"{anchor=center, pos=0.125, rotate=180}, draw=none, from=2-2, to=1-1]
\end{tikzcd}
\end{equation*}
defines the Finsler bundle $\F_\psi q\to E$ of $q$ only when $q$ is a submersion. In fact, there is no reason to believe that, in general, $\F_\psi q\to E$ would be the Finsler bundle of $q$ when $q$ fails to be a submersion. We start with a technical lemma, which is the dual of Lemma~\ref{lemma:split-pullback}.

\begin{lemma}
\label{lemma:split-pushout}
Let $\X$ be a category. Consider the following commutative diagram
\begin{equation*}
\begin{tikzcd}
{A'} & {B'} & {A'} \\
A & B & A
\arrow["{r'}", from=1-1, to=1-2]
\arrow["f"', from=1-1, to=2-1]
\arrow["{s'}", from=1-2, to=1-3]
\arrow["g"', from=1-2, to=2-2]
\arrow["f", from=1-3, to=2-3]
\arrow["r"', from=2-1, to=2-2]
\arrow["s"', from=2-2, to=2-3]
\end{tikzcd}
\end{equation*}
of $\X$, where $(s,r)\colon A\to B$ and $(s',r')\colon A\to B'$ are two section-retraction pairs. Then, the outer square is a pushout if and only if the left and the right squares are both pushouts.
\end{lemma}
\begin{proof}
Use the dual argument of Lemma~\ref{lemma:split-pullback}.
\end{proof}

\begin{lemma}
\label{lemma:finsler-bundle-from-split-abstract-connections}
If $\psi\colon\T E\to\T E$ is a split abstract horizontal connection on a submersion $q\colon E\to M$ and $(\q^\F_\psi,\H_\psi,\pi_\psi)$ is a splitting of $\psi$, then the differential bundle $\q_\psi^\F\colon\F_\psi q\to E$ is a choice of the Finsler bundle of $q$.
\end{lemma}
\begin{proof}
Consider a split abstract horizontal connection $\psi$ on $q$ with a linear splitting $(\q^\F_\psi,\H_\psi,\pi_\psi)$. Since the splitting is in the category of differential bundles, $\q_\psi^\F$ is already a differential bundles. We want to prove that $\q_\psi^\F$ coincides with the Finsler bundle of $q$. Consider the following diagram in $\DB(\X,\TT;E)$:
\begin{equation}
\label{equation:splitting-horizontal-connection}
\begin{tikzcd}
{\q^\V} & {\p_E} & {\q^\V} \\
{\p_E} & {\q_\psi^\F} & {\p_E}
\arrow["{q^\V}", from=1-1, to=1-2]
\arrow["{\iota_q}"', from=1-1, to=2-1]
\arrow["{z_q^\V}", from=1-2, to=1-3]
\arrow["{z_q^\V\iota_q\pi_\psi}", from=1-2, to=2-2]
\arrow["{\iota_q}", from=1-3, to=2-3]
\arrow["{\pi_\psi}"', from=2-1, to=2-2]
\arrow["\psi"', shift right, curve={height=24pt}, from=2-1, to=2-3]
\arrow["{\H_\psi}"', from=2-2, to=2-3]
\end{tikzcd}
\end{equation}
By \textbf{[AHC]}, the outer diagram is a pushout. Therefore, thanks to Lemma~\ref{lemma:split-pushout}, the inner square diagrams are also pushout diagrams. In particular, the left square diagram is a pushout in $\DB(\X,\TT;E)$. However, since $q$ is a submersion, the pushout of $q^\V$ along $\iota_q$ defines the Finsler bundle up to a unique isomorphism. Thus, $\q_\psi^\F$ represents a choice of the Finsler bundle of $q$.
\end{proof}

\begin{theorem}
\label{theorem:abstract-horizontal-connection}
Given a submersion $q\colon E\to M$ equipped with a morphism $\psi\colon\T E\to\T E$, the following are equivalent:
\begin{enumerate}
\item $\psi$ is a split abstract horizontal connection on $q$;

\item There is a full horizontal connection $\H$ whose horizontal form is $\psi$;

\item There is a horizontal connection $\H$ which is the horizontal component of an Ehresmann connection $(\R,\H)$ on $q$ and whose connection form is $\psi$.
\end{enumerate}
\end{theorem}
\begin{proof}
The equivalence between $[2]$ and $[3]$ was already proven by Theorem~\ref{theorem:full-connections}. It is only left to prove the equivalence between $[1]$ and $[2]$. To begin, consider an abstract horizontal connection $\psi\colon\T E\to\T E$ on the submersion $q$. Furthermore, let us assume that $\psi$ splits, that is, there is a section-retraction pair $(\H_\psi,\pi_\psi)\colon\q_\psi^\F\to\p_E$. By Lemma~\ref{lemma:finsler-bundle-from-split-abstract-connections}, $\q_\psi^\F$ defines a choice of the Finsler bundle of $q$. Now, consider again the diagram of Equation~\eqref{equation:splitting-horizontal-connection}. By Lemma~\ref{lemma:split-pushout}, the right square diagram is a pushout in $\DB(\X,\TT;E)$, therefore, $\H_\psi$ satisfies \textbf{[FHC]}. In order to prove that $\H_\psi$ is full horizontal connection, it is only left to show that $\H_\psi$ is a horizontal connection. To this end, we leverage the correspondence of Theorem~\ref{theorem:horizontal-connection-form} between horizontal connections and horizontal connection forms of submersions and show that $\psi$ is in fact a horizontal connection form. \textbf{[HCF.1]} corresponds to \textbf{[AHC.1]}, since $\psi\pi_\psi=\pi_\psi\H_\psi\pi_\psi=\pi_\psi$; finally, \textbf{[HCF.2]} corresponds to \textbf{[AHC.2]}. Thus, $\psi$ is a horizontal connection form and therefore, $\H_\psi$ is a full horizontal connection.

Conversely, using Theorem~\ref{theorem:horizontal-connection-form}, it is immediate to see that the associated horizontal connection form of a full horizontal connection satisfies the axioms of a split abstract horizontal connection.
\end{proof}

We end this section by collecting in the next theorem all the equivalent forms of an Ehresmann connection that we explored.

\begin{theorem}
\label{theorem:all-equivalent-forms-of-connections}
For a tangent display map $q\colon E\to M$ equipped with two maps $\R\colon\T E\to\V q$ and $\H\colon\F q\to\T E$, the following are equivalent:
\begin{enumerate}
\item $(\R,\H)$ is an Ehresmann connection on $q$;

\item $q$ is a submersion and $(\phi,\psi)$ is an Ehresmann connection form on $q$; moreover, $\phi=\R\iota_q$ and $\psi=\pi_q\H$;

\item $q$ is a submersion and $\R$ is a full vertical connection on $q$;

\item $q$ is a submersion and $\H$ is a full horizontal connection on $q$;

\item $q$ is a submersion and $\phi$ is a split abstract vertical connection on $q$; moreover, $\phi=\R\iota_q$;

\item $q$ is a submersion and $\psi$ is a split astract horizontal connection on $q$, moreover, $\psi=\pi_q\H$.
\end{enumerate}
\end{theorem}

\subsection{Linear vs Koszul connections}
\label{subsection:vertical-linear-connection}
In~\cite{cockett:connections}, Cockett and the first author introduced and studied connections on differential bundles in the context of tangent categories. In this section, we compare their work with our notion of Ehresmann connection. To distinguish the two, we will refer to their notion as a \textbf{Koszul connection}. To begin, we recall their definitions, starting from the notion of a vertical Koszul connection.

\begin{definition}[{\cite[Definition~3.2]{cockett:connections}}]
\label{definition:vertical-koszul-connection}
A \textbf{vertical Koszul connection} on a differential bundle $\q$ consists of a morphism $\K\colon\T E\to E$, where $E$ is the total space of $\q$, subject to the following conditions:
\begin{description}
\item[VKC.1] $\K$ is a retract of the vertical lift $l_q\colon E\to\T E$, that is, $l_q\K=\id_E$;

\item[VKC.2] $(p_M,\K)\colon\T\q\to\q$ is a linear morphism:
\begin{equation*}
\begin{tikzcd}
{\T E} & E \\
{\T M} & M
\arrow["\K", from=1-1, to=1-2]
\arrow["{\T q}"', from=1-1, to=2-1]
\arrow["q", from=1-2, to=2-2]
\arrow["{p_M}"', from=2-1, to=2-2]
\end{tikzcd}\quad
\begin{tikzcd}
{\T^2E} & {\T E} \\
{\T^2E} \\
{\T E} & E
\arrow["{\T\K}", from=1-1, to=1-2]
\arrow["{c_E}", from=2-1, to=1-1]
\arrow["{l_q}", from=3-1, to=2-1]
\arrow["\K"', from=3-1, to=3-2]
\arrow["{l_q}"', from=3-2, to=1-2]
\end{tikzcd}
\end{equation*}

\item[VKC.3] $(q,\K)\colon\p_E\to\q$ is a linear morphism:
\begin{equation*}
\begin{tikzcd}
{\T E} & E \\
E & M
\arrow["\K", from=1-1, to=1-2]
\arrow["{p_E}"', from=1-1, to=2-1]
\arrow["q", from=1-2, to=2-2]
\arrow["q"', from=2-1, to=2-2]
\end{tikzcd}\quad
\begin{tikzcd}
{\T^2E} & {\T E} \\
{\T E} & E
\arrow["{\T\K}", from=1-1, to=1-2]
\arrow["{l_E}", from=2-1, to=1-1]
\arrow["\K"', from=2-1, to=2-2]
\arrow["{l_q}"', from=2-2, to=1-2]
\end{tikzcd}
\end{equation*}
\end{description}
\end{definition}

Next, we recall the notion of a horizontal Koszul connection. To this end, note that, for a differential bundle $\q\colon E\to M$, not only is the Finsler bundle $\q^\F\colon\F q\to E$ a differential bundle, but also the projection $\T^\F q\colon\F q\to\T M$ onto $\T M$ carries the structure of a differential bundle, since $\T^\F q$ is a tangent pullback of $\q$ and that differential bundles are stable under tangent pullbacks (\cite[Lemma~2.7]{cockett:differential-bundles}). Furthermore, $\T q$ also carries the structure of a differential bundle whose projection, zero, and sum morphisms are the image of the projection, zero, and sum morphisms of $q$ along $\T$ and the vertical lift is defined as follows:
\begin{align*}
&l_{\T q}\colon\T E\xrightarrow{\T l_q}\T^2E\xrightarrow{c_E}\T^2E
\end{align*}
We may denote the differential bundles over $\T^\F q$ and $\T q$, respectively as $\T^\F\q$ and $\T\q$.

\begin{definition}[{\cite[Definition~4.5]{cockett:connections}}]
\label{definition:horizontal-koszul-connection}
A \textbf{horizontal Koszul connection} on a differential bundle $\q$ consists of a morphism $\H\colon\F q\to\T E$, satisfying the following conditions:
\begin{description}
\item[HKC.1] $\H$ is a section of the horizontal descent $\pi_q\colon\T E\to\F q$;

\item[HKC.2] $(\id_E,\H)\colon\q^\F\to\p_E$ is a linear morphism of differential bundles:
\begin{equation*}
\begin{tikzcd}
{\F q} & {\T E} \\
E & E
\arrow["\H", from=1-1, to=1-2]
\arrow["{q^\F}"', from=1-1, to=2-1]
\arrow["{p_E}", from=1-2, to=2-2]
\arrow[equals, from=2-1, to=2-2]
\end{tikzcd}\quad
\begin{tikzcd}
{\T\F q} & {\T^2E} \\
{\F q} & {\T E}
\arrow["{\T\H}", from=1-1, to=1-2]
\arrow["{l^\F_q}", from=2-1, to=1-1]
\arrow["\H"', from=2-1, to=2-2]
\arrow["{l_E}"', from=2-2, to=1-2]
\end{tikzcd}
\end{equation*}

\item[HKC.3] $(\id_{\T M},\H)\colon\T^\F\q\to\T\q$ is a linear morphism of differential bundles:
\begin{equation*}
\begin{tikzcd}
{\F q} & {\T E} \\
{\T M} & {\T M}
\arrow["\H", from=1-1, to=1-2]
\arrow["{q^\F}"', from=1-1, to=2-1]
\arrow["{\T q}", from=1-2, to=2-2]
\arrow[equals, from=2-1, to=2-2]
\end{tikzcd}\quad
\begin{tikzcd}
{\T\F q} & {\T^2E} \\
{\F q} & {\T E}
\arrow["{\T\H}", from=1-1, to=1-2]
\arrow["{l_{\T^\F q}}", from=2-1, to=1-1]
\arrow["\H"', from=2-1, to=2-2]
\arrow["{\T l_qc_E}"', from=2-2, to=1-2]
\end{tikzcd}
\end{equation*}
where $l_{\T^\F q}$ denotes the vertical lift of $\T^\F\q\colon\F q\to\T M$.
\end{description}
\end{definition}

Finally, we can recall the definition of a Koszul connection.

\begin{definition}[{\cite[Definition~5.2]{cockett:connections}}]
\label{definition:koszul-connection}
A \textbf{Koszul connection} on a differential bundle $\q\colon E\to M$ consists of a pair $(\K,\H)$ formed by a vertical Koszul connection $\K$ of $\q$ together with a horizontal Koszul connection $\H$ of $\q$, subject to the following conditions:
\begin{description}
\item[KC.1] The following diagram commutes:
\begin{equation*}
\begin{tikzcd}
{\F q} && {\T E} \\
E & M & E
\arrow["\H", from=1-1, to=1-3]
\arrow["{q^\F}"', from=1-1, to=2-1]
\arrow["\K", from=1-3, to=2-3]
\arrow["q"', from=2-1, to=2-2]
\arrow["{z_q}"', from=2-2, to=2-3]
\end{tikzcd}
\end{equation*}

\item[KC.2] The following diagram commutes
\begin{equation*}
\begin{tikzcd}
{\T E} & & {\T_2E} \\
& & {\T E}
\arrow["{\<\<\K,p_E\>\xi_q,\pi_q\H\>}", from=1-1, to=1-3]
\arrow[equals, from=1-1, to=2-3]
\arrow["{s_E}", from=1-3, to=2-3]
\end{tikzcd}
\end{equation*}
where $\xi_q\=(l_q,z_E)\T s_q\colon E_2\to\T E$
\end{description}
\end{definition}

The universal property of the vertical lift of differential bundles requires that the vertical bundle of a differential bundle is necessarily the trivial bundle $\pi_1\colon E_2\to E$ (see Example~\ref{example:vertical-bundle-differential-bundles}). Using this property of differential bundles, Cockett and the first author in~\cite[Section~3.35]{cockett:connections}, realized that a vertical Koszul connection is equivalent to a so-called \emph{Finsler connection}, that is, a map $\R\colon\T E\to E_2$, which is a retraction of the map $\xi_q$ and that satisfies some linearity conditions. We now recontextualize this result within the general theory of Ehresmann connections. We begin by defining the notions of linear vertical, horizontal, and Ehresmann connections. To this end, note that, using again the stability of differential bundles under the tangent bundle functor and tangent pullbacks, for a differential bundle $\q$, also the map $\T^\V q\colon\V q\to M$ carries the structure of a differential bundle that we may denote by $\T^\V\q$.

\begin{definition}
\label{definition:vertical-linear-connection}
A \textbf{linear vertical connection} on a differential bundle $\q\colon E\to M$ consists of a vertical connection $\R\colon\T E\to\V q$ on $q$ as per Definition~\ref{definition:vertical-connection} subject to the following condition:
\begin{description}
\item[VLC] The bundle morphism $(p_M,\R)\colon\T\q\to\T^\V\q$ of Lemma~\ref{lemma:vertical-connections} is linear
\begin{equation*}
\begin{tikzcd}
{\T^2E} & {\T\V q} \\
{\T E} & {\V q}
\arrow["{\T\R}", from=1-1, to=1-2]
\arrow["{\T l_qc_E}", from=2-1, to=1-1]
\arrow["\R"', from=2-1, to=2-2]
\arrow["{l_{\T^\V\q}}"', from=2-2, to=1-2]
\end{tikzcd}
\end{equation*}
where $l_{\T^\V\q}$ denotes the vertical lift of $\T^\V\q\colon\V q\to M$.
\end{description}
\end{definition}

\begin{definition}
\label{definition:horizontal-linear-connection}
A \textbf{linear horizontal connection} on a differential bundle $\q\colon E\to M$ consists of a horizontal connection $\H\colon\F q\to\T E$ on $q$ as per Definition~\ref{definition:horizontal-connection}, subject to the following condition:
\begin{description}
\item[HLC] The bundle morphism $(\id_{\T M},\H)\colon\T^\F\q\to\T\q$ of Lemma~\ref{lemma:horizontal-connection} is linear
\begin{equation*}
\begin{tikzcd}
{\T\F q} & {\T^2E} \\
{\F q} & {\T E}
\arrow["{\T\H}", from=1-1, to=1-2]
\arrow["{l_{\T^\F q}}", from=2-1, to=1-1]
\arrow["\H"', from=2-1, to=2-2]
\arrow["{\T l_qc_E}"', from=2-2, to=1-2]
\end{tikzcd}
\end{equation*}
where $l_{\T^\F q}$ denotes the vertical lift of $\T^\F\q$.
\end{description}
\end{definition}

\begin{definition}
\label{definition:ehresmann-linear-connection}
A \textbf{linear Ehresmann connection} on a differential bundle $\q\colon E\to M$ consists of an Ehresmann connection $(\R,\H)$ of the underlying tangent display map $q$, satisfying the further conditions:
\begin{description}
\item[VLC] The vertical connection $\R$ is linear;

\item[HLC] The horizontal connection $\H$ is linear.
\end{description}
\end{definition}

Thanks to the triviality of the vertical bundle of differential bundles (Example~\ref{example:vertical-bundle-differential-bundles}), linear vertical connections become equivalent to vertical Koszul connections, as already proved by Cockett and Cruttwell. We recall here this result. For starters, recall that, by the universal property of the vertical lift $l_q$ of a differential bundle $\q$, for each morphism $f\colon N\to\T E$ satisfying $f\T q=f\T qp_Mz_M$, there exists a unique morphism $\tilde f\colon N\to E_2$ such that $\tilde f\xi_q=f$ and $\tilde f\pi_1q=f\T qp_M$. This defines a morphism $\{f\}\colon N\to E$ as $\tilde f\pi_1$. Details about this operation $\{-\}$ can be found in~\cite{cockett:differential-bundles}.

\begin{proposition}[{\cite[Theorem~3.40]{cockett:connections}}]
\label{proposition:vertical-linear-connections}
Consider a differential bundle $\q\colon E\to M$ and a morphism $\R\colon\T E\to\V q$. The following are equivalent:
\begin{enumerate}
\item $\R$ is a linear vertical connection;

\item The unique morphism
\begin{align*}
&\K_\R\=\{\phi_\R\}_q\colon\T E\to E
\end{align*}
induced by the universal property of the vertical lift $l_q$ of $\q$, is a vertical Koszul connection, where $\phi_\R\=\R\iota_q$.
\end{enumerate}
\end{proposition}

Using this result, we can prove an equivalence between Koszul connections and linear Ehresmann connections.

\begin{theorem}
\label{theorem:koszul-linear-connection}
Consider a differential bundle $\q\colon E\to M$ and two morphisms $\R\colon\T E\to\V q$ and $\H\colon\F q\to\T E$. The following are equivalent:
\begin{enumerate}
\item $(\R,\H)$ is a linear Ehresmann connection on $\q$;

\item $(\K_\R,\H)$ is a Koszul connection on $\q$, where $\K_\R\colon\T E\to E$ is defined as in Proposition~\ref{proposition:vertical-linear-connections}.
\end{enumerate}
\end{theorem}
\begin{proof}
This is a direct consequence of Proposition~\ref{proposition:vertical-linear-connections} since horizontal Koszul connections are exactly horizontal (Ehresmann) connections and the orthogonality conditions \textbf{[KC.1]} and \textbf{[KC.2]}, correspond directly to \textbf{[LC.1]} and \textbf{[LC.2]},
\end{proof}

Now that we have clarified the relationship between Ehresmann connections and Koszul connections, we want to give a new perspective on Koszul connections that make use of the categorical nature of the structures involved.  As already mentioned, differential bundles and linear morphisms form a tangent category denoted by $\DB(\X,\TT)$, whose tangent structure obtained by lifting the one of $(\X,\TT)$.  We would like to show that Koszul connections are precisely Erhesmann connections in this category. To this end, we begin with a few technical results.  

We have mentioned in Example~\ref{example:db-trivial-tangent-bundle}, that every object $M$ of a tangent category admits a trivial differential bundle, denoted by $\0_M$. If $\q\colon E\to M$ is a differential bundle in $(\X,\TT)$, it is easy to see that the morphism $(\id_M,q)\colon\q\to\0_M$ becomes a linear morphism of differential bundles. We shall now prove that such a linear morphism admits sufficient pullbacks in $\DB(\X,\TT)$ provided $q$ is a tangent display map in the base tangent category. Unfortunately, in general, $(\id_M,q)$ fails to become a full tangent display map in $\DB(\X,\TT)$. However, the existing pullbacks suffice for our goals. 

\begin{lemma}
\label{lemma:lifting-differential-bundles}
For a given tangent category $(\X,\TT)$ and a differential bundle $\q\colon E\to M$ in $(\X,\TT)$, the linear morphism $(\id_M,q)\colon\q\to\0_M$ of differential bundles admits all tangent pullbacks along maps of type $(f,f)\colon\0_{M'}\to\0_M$ in $\DB(\X,\TT)$ provided that $q$ these tangent pullbacks exist in the base tangent category.
\end{lemma}
\begin{proof}
Consider a map of differential bundles $(f,f)\colon\0_{M'}\to\0_M$, where $f\colon M'\to M$ is a morphism of the base tangent category. Thanks to~\cite[Lemma~2.7]{cockett:differential-bundles}, we already know that differential bundles are stable under tangent pullbacks. Therefore, we may define $\q'\colon E'\to M'$ as the differential bundle in $(\X,\TT)$ obtained by pulling back $\q$ along $f$, as in the following diagram:
\begin{equation*}
\begin{tikzcd}
{E'} & E \\
{M'} & M
\arrow["g", from=1-1, to=1-2]
\arrow["{q'}"', from=1-1, to=2-1]
\arrow["\lrcorner"{anchor=center, pos=0.125}, draw=none, from=1-1, to=2-2]
\arrow["q", from=1-2, to=2-2]
\arrow["f"', from=2-1, to=2-2]
\end{tikzcd}
\end{equation*}
\cite[Lemma~2.7]{cockett:differential-bundles} also tells us that the morphism $(f,g)\colon\q'\to\q$ becomes a linear morphism of differential bundles and thus a map in $\DB(\X,\TT)$. We may use this fact to draw a diagram in the category of differential bundles, as follows:
\begin{equation*}
\begin{tikzcd}
{\q'} & \q \\
{\0_{M'}} & {\0_M}
\arrow["{(f,g)}", from=1-1, to=1-2]
\arrow["{(\id_{M'},q')}"', from=1-1, to=2-1]
\arrow["{(\id_M,q)}", from=1-2, to=2-2]
\arrow["{(f,f)}"', from=2-1, to=2-2]
\end{tikzcd}
\end{equation*}
Our goal is to prove that this diagram is in fact a tangent pullback. To this end, we start by considering two morphisms $(\alpha_\downarrow,\alpha^\uparrow)\colon\q''\to\0_{M'}$ and $(\beta_\downarrow,\beta^\uparrow)\colon\q''\to\q'$ of differential bundles, making the following diagram
\begin{equation*}
\begin{tikzcd}
{\q''} && \\
& {\q'} & \q \\
& {\0_{M'}} & {\0_M}
\arrow["{(\beta_\downarrow,\beta^\uparrow)}", curve={height=-18pt}, from=1-1, to=2-3]
\arrow["{(\alpha_\downarrow,\alpha^\uparrow)}"', curve={height=24pt}, from=1-1, to=3-2]
\arrow["{(f,g)}", from=2-2, to=2-3]
\arrow["{(\id_{M'},q')}"', from=2-2, to=3-2]
\arrow["{(\id_M,q)}", from=2-3, to=3-3]
\arrow["{(f,f)}"', from=3-2, to=3-3]
\end{tikzcd}
\end{equation*}
commutative. Using that $\0_{M'}$ is trivial, we have that $\alpha^\uparrow=q''\alpha_\downarrow$. Using the universal property of the tangent pullback that defines $\q'$, we may also define a morphism $\gamma\colon E''\to E'$ as follows:
\begin{equation*}
\begin{tikzcd}
{E''} && \\
& {E'} & E \\
& {M'} & M
\arrow["{\gamma^\uparrow}", dashed, from=1-1, to=2-2]
\arrow["{\beta^\uparrow}", curve={height=-18pt}, from=1-1, to=2-3]
\arrow["{\alpha^\uparrow}"', curve={height=18pt}, from=1-1, to=3-2]
\arrow["g", from=2-2, to=2-3]
\arrow["{q'}"', from=2-2, to=3-2]
\arrow["\lrcorner"{anchor=center, pos=0.125}, draw=none, from=2-2, to=3-3]
\arrow["q", from=2-3, to=3-3]
\arrow["f"', from=3-2, to=3-3]
\end{tikzcd}
\end{equation*}
However, since $\alpha^\uparrow=q''\alpha_\downarrow$, we compute that $\gamma^\uparrow q'=\alpha^\uparrow=q''\alpha_\downarrow$. Moreover, using that the pullback that defines $q'$ is also a tangent pullback, we can prove that $(\gamma_\downarrow,\gamma^\uparrow)\colon\q''\to\q'$ is linear, where $\gamma^\downarrow\=\alpha_\downarrow$. Using the linearity of $(\alpha_\downarrow,\alpha^\uparrow)$ and $(\beta_\downarrow,\beta^\uparrow)$, it is not hard to convince ourselves that the following two equations
\begin{align*}
&l_{q''}\T\gamma^\uparrow\T g=\beta^\uparrow l_q    &&l_{q''}\T\gamma^\uparrow\T q'=\alpha^\uparrow z_{M'}
\end{align*}
hold. However, also $(f,g)$ and $(\id_{M'},q')$ are linear, and therefore, for a similar argument, also the following two equations must hold:
\begin{align*}
&\gamma^\uparrow l_{q'}\T g=\beta^\uparrow l_q      &&\gamma^\uparrow l_{q'}\T q'=\alpha^\uparrow z_{M'}
\end{align*}
Therefore, we must conclude that $(\gamma_\downarrow,\gamma^\uparrow)$ is a linear morphism. Furthermore, $(\gamma_\downarrow,\gamma^\uparrow)$ makes the following diagram
\begin{equation*}
\begin{tikzcd}
{\q''} && \\
& {\q'} & \q \\
& {\0_{M'}} & {\0_M}
\arrow["{(\gamma_\downarrow,\gamma^\uparrow)}", dashed, from=1-1, to=2-2]
\arrow["{(\beta_\downarrow,\beta^\uparrow)}", curve={height=-18pt}, from=1-1, to=2-3]
\arrow["{(\alpha_\downarrow,\alpha^\uparrow)}"', curve={height=24pt}, from=1-1, to=3-2]
\arrow["{(f,g)}", from=2-2, to=2-3]
\arrow["{(\id_{M'},q')}"', from=2-2, to=3-2]
\arrow["{(\id_M,q)}", from=2-3, to=3-3]
\arrow["{(f,f)}"', from=3-2, to=3-3]
\end{tikzcd}
\end{equation*}
commutative and, by the universal property of the pullback that defines $q'$ it must also be the unique such map. This proves that $(\id_M,q)$ admits all the pullbacks along the maps of type $(f,f)\colon\0_{M'}\to\0_M$. To prove that these pullbacks are in fact tangent pullbacks, we can apply the same argument to $(\id_{\T^nM},\T^nq)$ and to $(\T^nf,\T^nf)\colon\0_{\T^nM'}\to\0_{\T^nM}$.
\end{proof}

Despite $(\id_M,q)\colon\q\to\0_M$ not fullfilling the requirements of being a tangent display map, it still admits enough pullbacks to admit both the constructions of the vertical and the Finsler bundles. In particular, it admits the vertical bundle and the Finsler bundle, which are the differential bundles
\begin{align*}
&(\id_M,q)^\V=(\id_E,\q^\V)\colon\T^\V q\to\0_E       &&(\id_M,q)^\F=(\id_{\T M},\q^\F)\colon\T^\F q\to\0_{\T M}
\end{align*}
respectively. Therefore, it makes sense to talk about Ehresmann connections on $(\id_M,q)$.  We can now give our characterization.  

\begin{theorem}
\label{theorem:linear-connections-as-connections-on-dbs}
Consider a display differential bundle $\q\colon E\to M$ of a tangent category $(\X,\TT)$. The following are equivalent.
\begin{enumerate}
\item $(\R,\H)$ is a linear Ehresmann connection on $\q$;

\item $(\K_\R,\H)$ is a Koszul connection on $\q$, where $\K_\R$ is defined as in Proposition~\ref{proposition:vertical-linear-connections};

\item $((p_M,\R),(\id_{\T M},\H))$ is an Ehresmann connection on $(\id_M,q)$ in $\DB(\X,\TT)$.
\end{enumerate}
Furthermore, if $((\R_\downarrow,\R^\uparrow),(\H_\downarrow,\H^\uparrow))$ is an Ehresmann connection on $(\id_M,q)$ in $\DB(\X,\TT)$, then $\R_\downarrow$ and $\H_\downarrow$ are necessarily $p_M$ and $\id_{\T M}$, respectively.
\end{theorem}
\begin{proof}
We have already proved the equivalence between $[1]$ and $[2]$ in Theorem~\ref{theorem:koszul-linear-connection}. It is only left to prove the equivalence between $[2]$ and $[3]$. Consider a linear Ehresmann connection $(\R,\H)$ on $\q$. For starters, notice that, by virtue of the axioms \textbf{[VLC]} and \textbf{[HLC]}, $(p_M,\R)\colon\T\q\to\T^\V\q$ and $(\id_{\T M},\H)\colon\T^\F\q\to\T\q$ become morphisms in $\DB(\X,\TT)$. Furthermore, thanks to Lemma~\ref{lemma:lifting-differential-bundles}, the vertical and the Finsler bundle of $(\id_M,q)$ are constructed directly using the pullbacks defining the vertical and the Finsler bundle of $\q$. Therefore, all the equational axioms of an Ehresmann connection on $(\id_M,q)$ follow directly from the ones satisfied by $(\R,\H)$. Therefore, $((p_M,\R),(\id_{\T M},\H))$ becomes in fact an Ehresmann connection on $(\id_M,q)$. Now, consider an Ehresmann connection $((\R_\downarrow,\R^\uparrow),(\H_\downarrow,\H^\uparrow))$ of $(\id_M,q)$ in $\DB(\X,\TT)$. However, since by \textbf{[VC.1]}, $(\R_\downarrow,\R^\uparrow)\colon p_{\q}^\DB\to(\id_M,q)^\V$ is required to be a retraction of the inclusion $\iota_{(\id_M,q)}=(z_M,\iota_q)$, $z_M\R_\downarrow=\id_M=z_Mp_M$, which implies that $\R_\downarrow=p_M$, since $z_M$ is monic. Furthermore, since by \textbf{[HC.1]} $(\H_\downarrow,\H^\uparrow)\colon(\id_M,q)^\F\to p_\q^\DB$ is required to be a section of the horizontal descent $\pi_{(\id_M,q)}=(\id_{\T M},\pi_q)$, it follows that $\H_\downarrow=\id_{\T M}$. \textbf{[VLC]} and \textbf{[HLC]} follow from $(p_M,\R^\uparrow)$ and $(\id_{\T M},\H^\uparrow)$ being linear morphisms, while all the other equational axioms follow from the other axioms satisfied by $((p_M,\R^\uparrow),(\id_{\T M},\H^\uparrow))$.
\end{proof}


\section{Vertical, Finsler, and horizontal vector fields}
\label{section:vector-fields}
In differential geometry, a \emph{vertical} vector field of a submersion $q\colon E\to M$ consists of a vector field $X\colon E\to\T E$ of $E$ such that, for each $y\in E$, the corresponding tangent vector $X_y\in\T_yE$ is vertical, that is, it belongs to the kernel of $\d_yq\colon\T_yE\to\T_{q(y)}M$. The vertical tangent vectors of $q$ form the vertical bundle of $q$, that we studied in Section~\ref{subsection:vertical-bundle}.
\par When an Ehresmann connection is provided, there is a prescribed way to \emph{split} the tangent bundle $\T E$ of the total space $E$ into vertical tangent vectors and \emph{horizontal} tangent vectors, which are those tangent vectors in the image of the horizontal connection $\H\colon\F q\to\T E$ or, equivalently, in the kernel of the vertical connection $\R\colon\T E\to\V q$.
\par While the vertical tangent vectors capture the degrees of freedom of the total space $E$ that do not change the base point of $M$, the \emph{horizontal} tangent vectors capture those degrees of freedom of $E$ that are \emph{parallel} to $M$, that is, that change the base point $M$ without changing the vertical degree of freedom of $E$. A horizontal vector field of $q$ consists of a vector field $Y\colon E\to\T E$ of $E$ such that, for each $y\in E$, the corresponding tangent vector $Y_y\in\T_yE$ is horizontal.
\par In this section, we introduce vertical and horizontal vector fields in tangent categories. These concepts will be important later when we discuss the curvature and parallel transport of an Ehresmann connection.

\subsection{Vertical vector fields}
\label{subsection:vertical-vector-fields}
A vector field on an object $M$ in a tangent category is a section of the tangent bundle~\cite{rosicky:tangent-cats,cockett:tangent-cats}, that is, a morphism $X\colon M\to\T M$ subject to the equation $Xp_M=\id_M$. As proved in~\cite{cockett:tangent-cats,cockett:jacobi}, in a tangent category with negatives, there is a well-defined notion of Lie bracket that generalize the usual notion of differential geometry. We  briefly recall this construction, since it will be important later.

As previously discussed for differential bundles, given a morphism $f\colon X\to\T^2M$ satisfying the equation $f\T p_M=fp_{\T M}p_Mz_M$, by the universal property of the vertical lift, there exists a unique morphism $\tilde f\colon N\to\T_2M$ satisfying the following equation. $\tilde f\xi_M=f$, where $\xi_M=(l_M\times_{z_M}z_{\T M})\T s_M$. We recommend the reader check~\cite{cockett:tangent-cats} for details. Since $\T_2M$ is the pullback of $p_M$ along itself, it comes with two projections, $\pi_1,\pi_2\colon\T_2M\to\T M$. The morphism $\tilde f\pi_1$ is denoted by $\{f\}\colon N\to\T M$.

Now, consider two maps $f,g\colon N\to\T M$ of $M$. If $fp_M=gp_M$, we can sum $f$ and $g$ together as follows:
\begin{align*}
&f+g\colon N\xrightarrow{\<X,Y\>}\T_2M\xrightarrow{s_M}\T M
\end{align*}
Using this notation, we may denote by $0$ the vector field $z_M$ and also by $-f$, the morphism $fn_M$, where $n_M$ denotes the negation $n_M\colon\T M\to\T M$.

Now, consider two vector fields $X,Y\colon M\to\T M$ of $M$. The Lie bracket of $X$ and $Y$ is the vector field $[X,Y]$, defined by the formula:
\begin{align*}
&[X,Y]\=\left\{X\T Y-Y\T Xc_M\right\}
\end{align*}

We shall now introduce the notion of vertical vector fields, which are those vector fields which pick out vertical tangent vectors of $q$.

\begin{definition}
\label{definition:vertical-vector-field}
A \textbf{vertical vector field} on a map $q\colon E\to M$ is of a vector field $X\colon E\to\T E$ on $E$ in $(\X,\TT)$ such that the following diagram commutes:
\begin{equation*}
\begin{tikzcd}
E & {\T E} \\
M & {\T M}
\arrow["X", from=1-1, to=1-2]
\arrow["q"', from=1-1, to=2-1]
\arrow["{\T q}", from=1-2, to=2-2]
\arrow["{z_M}"', from=2-1, to=2-2]
\end{tikzcd}
\end{equation*}
\end{definition}

Similarly to ordinary vector fields~\cite[Proposition~2.10]{cockett:differential-equations}, vertical vector fields form a tangent category. 

\begin{proposition}
\label{proposition:tangent-category-vertical-vector-fields}
There is a tangent category $\VVF(\X,\TT)$, whose objects are pairs $(q,X)$ formed by a tangent display map $q\colon E\to M$ and a vertical vector field $X$ on $q$. A morphism of \ $\VVF(\X,\TT)$ from $(q,X)$ to $(q',X')$ is a morphism $(f,g)\colon q\to q'$ of tangent display maps, that is, a pair of morphisms $f\colon M\to M'$ and $g\colon E\to E'$ that commute with $q$ and $q'$. Furthermore, a morphism $(f,g)$ is required to make the following diagram commute:
\begin{equation*}
\begin{tikzcd}
{\T E} & {\T E'} \\
E & {E'}
\arrow["{\T g}", from=1-1, to=1-2]
\arrow["X", from=2-1, to=1-1]
\arrow["g"', from=2-1, to=2-2]
\arrow["{X'}"', from=2-2, to=1-2]
\end{tikzcd}
\end{equation*}
The tangent bundle functor $\T^\VVF$ of $\VVF(\X,\TT)$ sends a pair $(q,X)$ to $(\T q,X_\T)$, where we define:
\begin{align*}
&X_\T\colon\T E\xrightarrow{\T X}\T^2E\xrightarrow{c_E}\T^2E
\end{align*}
Finally, the structural natural transformations of $\VVF(\X,\TT)$ are the same as the base category.
\end{proposition}
\begin{proof}
Similar to the proof of ~\cite[Proposition~2.10]{cockett:differential-equations}.
\end{proof}

\par As already mentioned, for a given tangent category $(\X,\TT)$ and an object $M$ of $(\X,\TT)$, one can define a new tangent category, denoted $(\X,\TT)/M$ and called the slice tangent category on $M$. The objects of $(\X,\TT)/M$ are tangent display maps $q\colon E\to M$ with $M$ per codomain, and morphisms $f\colon q\to q'$ are morphisms of $\X$ that commute with $q$ and $q'$, that is, $fq'=q$. The tangent bundle functor $\T^\V $ sends an object $q\colon E\to M$ to $\T^\V q\colon\V q\to M$, that is, the pullback of $\T q$ along the zero $z_M$ of Equation~\eqref{equation:vertical-bundle-pullback}, and a morphism $f\colon q\to q'$ to the unique morphism $\T^\V f\colon\T^\V q\to\T^\V q'$ satisfying $\T^\V f\T^\V q'=\T^\V q$ and $\T^\V f\iota_{q'}=\iota_q\T f$. The projection $p^\V_q\colon\T^\V q\to q$ is given by $\iota_qp_E$, while the other structural natural transformations are constructed using the universal property of tangent pullbacks. We suggest the reader to see~\cite[Section~3.3]{cruttwell:tangent-display-maps} for details.

\par Our next goal is to show that one may view vertical vector fields as vector fields in the slice tangent category $(\X,\TT)/M$. To prove this, we will first establish several technical results.  We say that a functor $F\colon(\X,\TT)\to(\X',\TT')$ between two tangent categories \textbf{creates tangent limits} when, given a diagram $D\colon I\to(\X,\TT)$, if the diagram $F\o D\colon I\to\X'$ admits a tangent limit cone $\varphi_c\colon M\to FD(c)$ in $\X'$, there exists a cone $\psi_c\colon L\to D(c)$ of $D$ on $\X$ such that $FL=M$, $F\psi_c=\varphi_c$, and $\psi_c$ is a tangent limit cone for $D$.

\begin{lemma}
\label{lemma:U-VF-reflects-limits}
The forgetful functor $\U^\VF\colon\VF(\X,\TT)\to(\X,\TT)$ creates tangent limits.
\end{lemma}
\begin{proof}
First, we show that $\U^\VF$ is conservative, that is, $\U^\VF$ reflects isomorphisms. To this end, consider a morphism $f\colon(M,X)\to(M',X')$ of vector fields which is an isomorphism in the base tangent category, with inverse $g\colon M'\to M$. We want to show that $g$ is a morphism of vector fields $g\colon(M',X')\to(M,X)$:
\begin{align*}
gX&=~gX\T f\T g                      \Tag{fg=\id_M}\\
&=~gfX'\T g                          \Tag{X\T f=fX'}\\
&=~X'\T g                            \Tag{gf=\id_{M'}}
\end{align*}
The next step is to show that for each diagram $D\colon J\to\VF(\X,\TT)$, if $\U^\VF\o D$ admits a tangent limit cone, there exists a tangent limit cone of $D$ preserved by $\U^\VF$. Consider such a diagram together with a tangent limit cone $\varphi_c\colon L\to\U^\VF(D(c))$ of $\U^\VF\o D$ in $(\X,\TT)$. In particular, $\T\varphi_c\colon\T L\to\T(\U^\VF(D(c)))$ is a limit cone of $\T\o\U^\VF\o D$. Let $D(c)\=(M_c,X_c)$ be the vector field corresponding to $c\in I$. Define $Y\colon L\to\T L$ as the unique morphism making the following diagram
\begin{equation*}
\begin{tikzcd}
{\T L} & {\T(\U^\VF(D(c)))} \\
L & {\U^\VF(D(c))}
\arrow["{\T\varphi_c}", from=1-1, to=1-2]
\arrow["Y", dashed, from=2-1, to=1-1]
\arrow["{\varphi_c}"', from=2-1, to=2-2]
\arrow["{X_c}"', from=2-2, to=1-2]
\end{tikzcd}
\end{equation*}
commutative, where $X_c\colon\U^\VF(D(c))\to\T(\U^\VF(D(c)))$ since $X_c$ is a vector field. However, by the universality of the limit cone, $Yp$ is the unique morphism making the following diagram
\begin{equation*}
\begin{tikzcd}
L & {\U^\VF(D(c))} \\
{\T L} & {\T(\U^\VF(D(c)))} \\
L & {\U^\VF(D(c))}
\arrow["{\varphi_c}", from=1-1, to=1-2]
\arrow["{p_L}", from=2-1, to=1-1]
\arrow["{\T\varphi_c}", from=2-1, to=2-2]
\arrow["{p_{D(c)}}"', from=2-2, to=1-2]
\arrow["Y", from=3-1, to=2-1]
\arrow["{\varphi_c}"', from=3-1, to=3-2]
\arrow["{\id_{\U^\VF(D(c))}}"', shift right=5, curve={height=24pt}, from=3-2, to=1-2]
\arrow["{X_c}"', from=3-2, to=2-2]
\end{tikzcd}
\end{equation*}
commute. Thus, $Yp$ must be the identity, since $X_cp_{D(c)}=\id_{\U^\VF(D(c))}$, namely, $Y$ is a vector field on $L$. Furthermore, each $\varphi_c$ becomes a morphism of vector fields $\varphi_c\colon(L,Y)\to D(c)$. Let us prove that the cone $\varphi_c\colon(L,Y)\to D(c)$ is a limit cone in $\VF(\X,\TT)$ for $D$. To this end, consider another cone $\psi_c\colon(M,X)\to D(c)=(M_c,X_c)$. Since $\U^\VF$ preserves the cone $\varphi_c$, we obtain a unique morphism $\xi\colon M\to L$ in $(\X,\TT)$:
\begin{equation*}
\begin{tikzcd}
L & {\U^\VF(D(c))} \\
M
\arrow["{\varphi_c}", from=1-1, to=1-2]
\arrow["\xi", dashed, from=2-1, to=1-1]
\arrow["{\psi_c}"', from=2-1, to=1-2]
\end{tikzcd}
\end{equation*}
Let us prove that $\xi$ is a morphism of vector fields. Notice that $\xi Y\colon M\to\T L$ is the unique morphism making the following diagram commute:
\begin{equation*}
\begin{tikzcd}
{\T L} & {\T(\U^\VF(D(c)))} \\
L & {\U^\VF(D(c))} \\
M
\arrow["{\T\varphi_c}", from=1-1, to=1-2]
\arrow["Y", from=2-1, to=1-1]
\arrow["{\varphi_c}", from=2-1, to=2-2]
\arrow["{X_c}"', from=2-2, to=1-2]
\arrow["\xi", from=3-1, to=2-1]
\arrow["{\psi_c}"', from=3-1, to=2-2]
\end{tikzcd}
\end{equation*}
However, so does $X\T\xi$, since:
\begin{equation*}
\begin{tikzcd}
{\T L} & {\T(\U^\VF(D(c)))} \\
{\T M} & {\U^\VF(D(c))} \\
M
\arrow["{\T\varphi_c}", from=1-1, to=1-2]
\arrow["{\T\xi}", from=2-1, to=1-1]
\arrow["{\T\psi_c}"{description}, from=2-1, to=1-2]
\arrow["{X_c}"', from=2-2, to=1-2]
\arrow["X", from=3-1, to=2-1]
\arrow["{\psi_c}"', from=3-1, to=2-2]
\end{tikzcd}
\end{equation*}
Therefore, by the universality of $L$ in $(\X,\TT)$, $\xi Y$ must coincide with $X\T\xi$, namely, $\xi$ is a morphism of vector fields. Finally, since $\U^\VF$ is conservative and preserves the limit cone $\varphi_c\colon(L,Y)\to D(c)$, $\U^\VF$ creates tangent limits~\cite{kelly:conservative-functors}.
\end{proof}

The next proposition shows that Cartesian tangent morphisms which create tangent limits \textbf{reflect tangent display maps}, that is, if $q\colon E\to M$ is a map of $(\X,\TT)$ such that $Fq$ is tangent display in $(\X',\TT')$, then $q$ is a tangent display map in $(\X,\TT)$.

\begin{proposition}
\label{proposition:reflects-tangent-display-maps}
Cartesian tangent morphisms that create tangent limits reflect tangent display maps.
\end{proposition}
\begin{proof}
Consider a strong tangent morphism $(F,\alpha)\colon(\X,\TT)\to(\X',\TT')$ and a map $q\colon E\to M$ of $(\X,\TT)$ such that $Fq$ is tangent display in $(\X',\TT')$. For starters, we realize that, for each $n\geq 0$, since $(F,\alpha)$ is Cartesian, the following diagram
\begin{equation*}
\begin{tikzcd}
{F\T^nE} & {\T'F\T^{n-1}E} & \dots & {{\T'}^nFE} \\
{F\T^nM} & {\T'F\T^{n-1}M} & \dots & {{\T'}^nFM}
\arrow["{\alpha_{\T^{n-1}E}}", from=1-1, to=1-2]
\arrow["{F\T^nq}"', from=1-1, to=2-1]
\arrow["{\T'\alpha_{\T^{n-2}E}}", from=1-2, to=1-3]
\arrow["{\T'F\T^{n-1}q}", from=1-2, to=2-2]
\arrow["{{\T'}^{n-1}\alpha_E}", from=1-3, to=1-4]
\arrow["{{\T'}^nFq}", from=1-4, to=2-4]
\arrow["{\alpha_{\T^{n-1}M}}"', from=2-1, to=2-2]
\arrow["{\T'\alpha_{\T^{n-2}M}}"', from=2-2, to=2-3]
\arrow["{{\T'}^{n-1}\alpha_M}"', from=2-3, to=2-4]
\end{tikzcd}
\end{equation*}
is a tangent pullback. Since tangent display maps are stable under tangent pullbacks, this implies that $F\T^nq$ is also tangent display. Now, consider a morphism $f\colon N\to\T^nM$ of $(\X,\TT)$. Since $F\T^nq$ is tangent display, the tangent pullback of $Ff$ along $F\T^nq$ exists in $(\X',\TT')$. However, since $F$ creates tangent limits, the tangent pullback of $f$ along $\T^nq$ exists in $(\X,\TT)$. Therefore, $q$ is a tangent display map in $(\X,\TT)$.
\end{proof}

Unfortunately, in general, a tangent display map of $\VF(\X,\TT)$ fails to be a tangent display map of the base tangent category $(\X,\TT)$. However, thanks to Proposition~\ref{proposition:reflects-tangent-display-maps}, we now show that, if $q$ is a morphism of vector fields which is a tangent display map in $(\X,\TT)$, then $q$ is also a tangent display map in $\VF(\X,\TT)$.

\begin{corollary}
\label{corollary:U-reflects-tangent-display-maps}
The forgetful functor $\U^\VF\colon\VF(\X,\TT)\to(\X,\TT)$ reflects tangent display maps.
\end{corollary}
\begin{proof}
This is a direct application of Lemma~\ref{lemma:U-VF-reflects-limits} and Proposition~\ref{proposition:reflects-tangent-display-maps}.
\end{proof}

We can now prove that vertical vector fields can  be regarded as vector fields in the slice category (and are also equivalent to two other formulations).  

\begin{theorem}
\label{theorem:vertical-vector-fields}
Let $X\colon E\to\T E$ be a morphism in a tangent category $(\X,\TT)$ and $q\colon E\to M$ a tangent display map. The following are equivalent:
\begin{enumerate}
\item $X$ is a vertical vector field on $q$;

\item $X$ factors through $\iota_q$ by defining a vector field $\hat X\colon q\to\T^\V q$ in the slice tangent category $(\X,\TT)/M$ over $q$;

\item $X$ is a vector field and $q\colon(E,X)\to(M,z_M)$ is a morphism of vector fields;

\item $q\colon(E,X)\to(M,z_M)$ is an object of the slice tangent category $\VF(\X,\TT)/(M,z_M)$ of the tangent category of vector fields $\VF(\X,\TT)$ of $(\X,\TT)$.
\end{enumerate}
\end{theorem}
\begin{proof}
For starters, let us show that [1] implies [2]. Consider a vertical vector field $X$ of a tangent display map $q^E_M$. Thus, $q$ is an object of $(\X,\TT)/M$. Moreover, since $X\T q=qz_M$, $X$ factors through $\iota_q$ as follows:
\begin{equation*}
\begin{tikzcd}
E \\
& {\V q} & {\T E} \\
& M & {\T M}
\arrow["{\hat X}", dashed, from=1-1, to=2-2]
\arrow["X", curve={height=-12pt}, from=1-1, to=2-3]
\arrow["q"', curve={height=12pt}, from=1-1, to=3-2]
\arrow["{\iota_q}", from=2-2, to=2-3]
\arrow["{\T^\V q}"', from=2-2, to=3-2]
\arrow["\lrcorner"{anchor=center, pos=0.125}, draw=none, from=2-2, to=3-3]
\arrow["{\T q}", from=2-3, to=3-3]
\arrow["{z_M}"', from=3-2, to=3-3]
\end{tikzcd}
\end{equation*}
Moreover, since $X$ is a vector field:
\begin{align*}
&\hat X p^\V_q=\hat X\iota_qp_E=Xp_E=\id_E
\end{align*}
Thus, $\hat X$ becomes a vector field on $q$ in $(\X,\TT)/M$.
\par Next, consider a vector field $\hat X\colon q\to\T^\V q$ on $q$ in $(\X,\TT)/M$; let $X$ be $\hat X\iota_q$. $X$ is clearly a vector field of $(\X,\TT)$, since $Xp_E=\hat X\iota_qp_E=\hat X p^\V_q=\id_E$.
Moreover, by assumption, $\hat X$ is a morphism in the slice tangent category, thus, $\hat X\T^\V q=q$. We shall compute:
\begin{align*}
&X\T q=\hat X\iota_q\T q=\hat X\T^\V qz_M=qz_M
\end{align*}
Therefore, $q\colon(E,X)\to(M,z_M)$ becomes a morphism of vector fields. To prove that $q$ defines an object in the slice tangent category $\VF(\X,\TT)/(M,z_M)$, we only need to prove that $q$ is in fact a tangent display map in $\VF(\X,\TT)$. However, by Corollary~\ref{corollary:U-reflects-tangent-display-maps}, the forgetful functor $\U^\VF\colon\VF(\X,\TT)\to(\X,\TT)$ reflects tangent display maps. Thus, since $q$ is tangent display in $(\X,\TT)$, it becomes a tangent display map in $\VF(\X,\TT)$.
\par Finally, let us assume [4]; namely that $q\colon(E,X)\to(M,z_M)$ is an object of $\VF(\X,\TT)/(M,z_M)$. In particular, $X$ is a vector field on $E$ in $(\X,\TT)$. Moreover, since $q$ is a morphism of vector fields, $X\T q=qz_M$, thus, $X$ is a vector field on $q$ in $(\X,\TT)$.
\end{proof}

Using the characterization provided by Theorem~\ref{theorem:vertical-vector-fields}, we shall now prove that, given a vertical connection, vertical vector fields are those vector fields $X\colon E\to\T E$ satisfying the equation $X=X\phi$.

\begin{corollary}
\label{corollary:vertical-vector-fields-connection}
Given a vertical connection on $q\colon E\to M$, a vector field $X\colon E\to\T E$ of $E$ is vertical if and only if $X=X\phi$, where $\phi$ is the vertical connection form.
\end{corollary}
\begin{proof}
Suppose that a vector field $X\colon E\to\T E$ satisfies $X=X\phi$. By definition, $\phi=\R\iota_q$, where $\R$ is the vertical connection. Thus, $X=X\R\iota_q$. However, $X\R q^\V=Xp_E=\id_E$, since $X$ is a vector field and $\R\colon\q^\V\to\p_E$ is a bundle morphism. Thus, $X\R q^\V$ is a section of the vertical bundle and therefore, $X$ is necessarily vertical. Now, suppose that $X$ is a vertical vector field. Thus, $X=\hat X\iota_q$ for some section $\hat X$ of the vertical bundle. Thus, $X\phi=\hat X\iota_q\phi$. However, by \textbf{[VCF.1]}, $\iota_q\phi=\iota_q$. Thus, $X\phi=\hat X\iota_q=X$.
\end{proof}

\subsection{Finsler and horizontal vector fields}
\label{subsection:finsler-vector-fields}
In this section, we introduce Finsler and horizontal vector fields.

\begin{definition}
\label{definition:finsler-vector-field}
A \textbf{Finsler vector field} on a map $q\colon E\to M$ consists of a morphism $A\colon E\to\T M$ subject to the following condition:
\begin{equation*}
\begin{tikzcd}
E & {\T M} \\
& M
\arrow["A", from=1-1, to=1-2]
\arrow["q"', from=1-1, to=2-2]
\arrow["{p_M}", from=1-2, to=2-2]
\end{tikzcd}
\end{equation*}
\end{definition}

\begin{remark}
\label{remark:finsler-vector-fields}
It is important to realize that a Finsler vector field does not need to be a vector field in any meaningful way. This naming convention will be clear later.
\end{remark}

Just like vertical vector fields, Finsler vector fields form a tangent category.

\begin{proposition}
\label{proposition:tangent-category-finsler-vector-fields}
There is a tangent category $\FVF(\X,\TT)$ with the following structure.  An object of $\FVF(\X,\TT)$ consists of a pair $(q,A)$ formed by a tangent display map $q\colon E\to M$ and a Finsler vector field $A$ on $q$; a morphism $(f,g)\colon(q,A)\to(q',A')$ of $\FVF(\X,\TT)$ consists of a pair of morphisms $f\colon M\to M'$ and $g\colon E\to E'$ that commute with $q$ and $q'$, that is, $qf=gq'$, and commute with the Finsler vector fields, that is, $A\T f=A'g$. The tangent bundle functor $\T^\FVF\colon\FVF(\X,\TT)\to\FVF(\X,\TT)$ sends a pair $(q,A)$ to $(\T q,A_\T)$, where we define:
\begin{align*}
&A_\T\colon\T E\xrightarrow{\T A}\T^2M\xrightarrow{c}\T^2M
\end{align*}
Finally, the structural natural transformations of $\FVF(\X,\TT)$ are the same as the base category.  
\end{proposition}
\begin{proof}
Similar to the proof of ~\cite[Proposition~2.10]{cockett:differential-equations}.    
\end{proof}

\begin{theorem}
\label{theorem:finsler-vector-fields}
Let $A\colon E\to\T M$ be a morphism in a tangent category $(\X,\TT)$ and let $q\colon E\to M$ be a tangent display map. The following are equivalent:
\begin{enumerate}
\item $A$ is a Finsler vector field on $q$;

\item $A$ factors through $\T^\F q\colon\F q\to\T M$ by defining a section $A_\F\colon E\to\F q$ of the Finsler bundle $\q^\F$.
\end{enumerate}
\end{theorem}
\begin{proof}
For starters, suppose that $A$ is a Finsler vector field on $q$. Thus, by using the universal property of the Finsler bundle, we can define a unique morphism $A_\F\colon E\to\F q$ making the following diagram commutative:
\begin{equation*}
\begin{tikzcd}
E \\
& {\F q} & E \\
& {\T M} & M
\arrow["{A_\F}", dashed, from=1-1, to=2-2]
\arrow[curve={height=-18pt}, equals, from=1-1, to=2-3]
\arrow["A"', curve={height=18pt}, from=1-1, to=3-2]
\arrow["{q^\F}", from=2-2, to=2-3]
\arrow["{\T^\F q}"', from=2-2, to=3-2]
\arrow["\lrcorner"{anchor=center, pos=0.125}, draw=none, from=2-2, to=3-3]
\arrow["q", from=2-3, to=3-3]
\arrow["{p_M}"', from=3-2, to=3-3]
\end{tikzcd}
\end{equation*}
By construction, $A_\F$ is a section of the Finsler bundle. Now, assume that $A$ factors through $\T^\F q$ via a section $A_\F\colon E\to\F q$ of the Finsler bundle. Thus, we shall compute:
\begin{align*}
&Ap_M=A_\F\T^\F q p_M=A_\F q^\F q=q
\end{align*}
Therefore, $A$ is a Finsler vector field.
\end{proof}

Thanks to the characterization of Theorem~\ref{theorem:finsler-vector-fields}, in the presence of a horizontal connection $\H$, a Finsler vector field $A$ can be used to define a vector field over $E$ obtained by composing the corresponding section $A_\F$ of the Finsler bundle $\q^\F$ by $\H$, that is, $A_\F\H$. Those vector fields of that form are known as \emph{horizontal} vector fields.

\begin{definition}
\label{definition:horizontal-vector-field}
Given a horizontal connection $\H\colon\F q\to\T E$ on a tangent display map $q\colon E\to M$, a \textbf{$\H$-horizontal vector field} on $q$ consists of a vector field $X\colon E\to\T E$ of $E$ such that, there exists a (necessarily unique) Finsler vector field $A\colon E\to\T M$ satisfying $X=A_\F\H$, where $A_\F$ is the associated section of the Finsler bundle of $q$.
\end{definition}

In Corollary~\ref{corollary:vertical-vector-fields-connection}, we showed that a vector field $X\colon E\to\T E$ is vertical if and only if $X$ satisfies the equation $X=X\phi$, provided the existence of a vertical connection. Similarly, horizontal vector fields are those vector fields that satisfy $X=X\psi$, where $\psi$ denotes the horizontal connection form.

\begin{corollary}
\label{corollary:horizontal-vector-fields-connection}
Given a horizontal connection on $q\colon E\to M$, a vector field $X\colon E\to\T E$ of $E$ is horizontal if and only if $X=X\psi$, where $\psi$ is the horizontal connection form.
\end{corollary}
\begin{proof}
Suppose that a vector field $X\colon E\to\T E$ satisfies $X=X\psi$. By definition, $\psi=\pi_q\H$, where $\H$ is the horizontal connection. Thus, $X=X\pi_q\H$. However, $X\pi_q q^\F=Xp_E=\id_E$, where we used that $X$ is a vector field and that $\pi_qq^\F=p_E$. Thus, $X\pi_q$ is a section of the Finsler bundle and therefore, $X$ is necessarily horizontal. Now, suppose that $X$ is a horizontal vector field. Thus, $X=A_\F\H$ for some section $A_\F$ of the Finsler bundle. Thus, $X\psi=A_\F\H$. However, by \textbf{[HCF.1]}, $\H\psi=\H$. Thus, $X\psi=A_\F\H\psi=A_\F\H=X$.
\end{proof}

\subsection{The decomposition induced by the connection}
\label{subsection:decomposition}
Thanks to Theorem~\ref{theorem:splitting-of-ses}, an Ehresmann connection on a submersion is equivalent to a splitting of the fundamental short exact sequence. A splitting allows one to decompose the tangent bundle $\p_E$ as a biproduct $\p_E\cong\q^\V\oplus_E\q^\F$, that is, $\T E\cong\V q\oplus_E\F q$, where $\oplus_E$ is a biproduct in the category of differential bundles over $E$ (Corollary~\ref{corollary:decomposition-connection}). This decomposition induces a decomposition of vector fields into a vertical and a horizontal component.

\begin{theorem}
\label{theorem:unique-decomposition-of-vectors}
Given an Ehresmann connection $(\R,\H)$ on a tangent display map, every vector field $X\colon E\to\T E$ can be uniquely decomposed into a vertical and a horizontal component:
\begin{align*}
&X=X_\V+X_\H
\end{align*}
where $X_\V\=X\phi$ is a vertical vector field and $X_\H\=X\psi$ is a horizontal vector field.
\end{theorem}
\begin{proof}
For starters, let us show that $X_\V+X_\H$ correspond to $X$. Using \textbf{[ECF.2]}, we shall compute:
\begin{align*}
&X_\V+X_\H=X\phi+X\psi=X(\phi+\psi)=X\id_{\T E}=X
\end{align*}
Now, suppose that $X'_\V$ and $X'_\H$ are vertical and horizontal vector fields such that $X=X'_\V+X'_\H$. By Corollaries~\ref{corollary:vertical-vector-fields-connection} and~\ref{corollary:horizontal-vector-fields-connection}, $X'_\V=X'_\V\phi$ and $X'_\H=X'_\H\psi$. Therefore, by using \textbf{[ECF.1]}, we can prove:
\begin{align*}
&X_\V=X\phi=(X_\V'+X_\H')\phi=X_\V'\phi\phi+X_\H'\psi\phi=X_\V'\phi+X_\H'p_Ez_E=X_\V'+0=X_\V'\\
&X_\H=X\psi=(X_\V'+X_\H')\psi=X_\V'\phi\psi+X_\H'\psi\psi=X_\V'p_Ez_E+X_\H'\psi=0+X_\H'=X_\H'
\end{align*}
This proves that the decomposition is necessarily unique.
\end{proof}

\subsection{The Finsler tangent category}
\label{subsection:strong-finsler-vector-fields}
Theorem~\ref{theorem:vertical-vector-fields} characterizes vertical vector fields as vector fields in the slice tangent category. In this section, we take a brief detour from the paper's main story to explore the following question. Can Finsler vector fields be characterized as vector fields in a suitable tangent category? This exploration, which we believe is worth including, will not, however, be necessary to the rest of the paper. Therefore, the reader may feel free to skip to the next section.

We shall start by observing that the tangent structure of the slice tangent category is induced, in some appropriate sense, by the vertical bundle. In fact, in the slice, the tangent bundle of an object $q$, $q$ being a tangent display map $q\colon E\to M$, is the object $\T^\V q$, which corresponds to the morphism $\T^\V q\colon\V q\to M$, obtained by pulling back $\T q$ along $z_M$, as in Diagram~\eqref{equation:vertical-bundle-pullback}. Moreover, the underlying map of the projection $p^\V\colon\T^\V q\to q$ corresponds to the projection $\q^\V=\iota_qp_E\colon\V q\to E$ of the vertical bundle of $q$. Inspired by this observation, we shall construct a new tangent category, denoted by $\F(\X,\TT)$ as follows.

The underlying category is the category $\Dsply(\X,\TT)$ tangent display maps.

The tangent bundle functor $\T^\F\colon\Dsply(\X,\TT)\to\Dsply(\X,\TT)$ sends each $q\colon E\to M$ to $\T^\F q\colon\F E\to\T M$, defined by the tangent pullback:
\begin{equation*}
\begin{tikzcd}
{\F q} & E \\
{\T M} & M
\arrow["{q^\F}", from=1-1, to=1-2]
\arrow["{\T^\F q}"', from=1-1, to=2-1]
\arrow["\lrcorner"{anchor=center, pos=0.125}, draw=none, from=1-1, to=2-2]
\arrow["q", from=1-2, to=2-2]
\arrow["{p_M}"', from=2-1, to=2-2]
\end{tikzcd}
\end{equation*}
Furthermore, $\T^\F$ sends a morphism $(f,g)\colon q\to q'$ to the morphism $(\T f,\F_fg)\colon\T^\F q\to\T^\F q'$ defined by Lemma~\ref{lemma:functoriality-F}. Next, we shall describe the structural natural transformations of the tangent structure $\TT^\F$, starting from the projection. The projection $p^\F_q\colon\T^\F q\to q$ corresponds to the pair $p^\F\=(p_M,q^\F)$. Now, recall that, since $q^\F$ is a tangent pullback of a differential bundle, $\p_M\colon\T M\to M$ being the tangent bundle, also $q^\F$ acquires the structure of a differential bundle, which is, by definition, the Finsler bundle $\q^\F$ of $q$. Using the zero $z_q^\F$ and the sum $s_q^\F$ morphisms of $\q^\F$, we may introduce the zero and sum morphism of $(\Dsply(\X,\TT),\TT^\F)$ as the following commutative diagrams
\begin{equation*}
\begin{tikzcd}
E & {\F q} \\
M & {\T M}
\arrow["{z_q^\F}", from=1-1, to=1-2]
\arrow["q"', from=1-1, to=2-1]
\arrow["{\T^\F q}", from=1-2, to=2-2]
\arrow["{z_M}"', from=2-1, to=2-2]
\end{tikzcd}\quad
\begin{tikzcd}
{\F_2q} & {\F q} \\
{\T_2M} & {\T M}
\arrow["{s_q^\F}", from=1-1, to=1-2]
\arrow["{\T^\F q\times_M\T^\F q}"', from=1-1, to=2-1]
\arrow["{\T^\F q}", from=1-2, to=2-2]
\arrow["{s_M}"', from=2-1, to=2-2]
\end{tikzcd}
\end{equation*}
respectively. Finally, we introduce the vertical lift and the canonical flip, by invoking once again the universal property of the tangent pullback that defines $\T^\F q$:
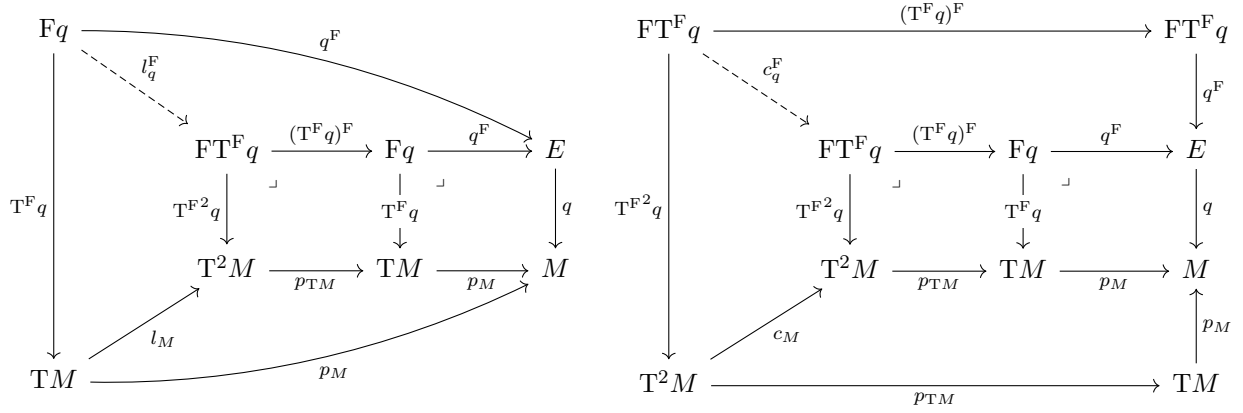
\begin{equation*}
\adjustbox{width=\linewidth}{
\begin{tikzcd}
{\F q} &&& \\
& {\F\T^\F q} & {\F q} & E \\
& {\T^2M} & {\T M} & M \\
{\T M}
\arrow["{l^\F_q}", dashed, from=1-1, to=2-2]
\arrow["{q^\F}", curve={height=-18pt}, from=1-1, to=2-4]
\arrow["{\T^\F q}"', from=1-1, to=4-1]
\arrow["{(\T^\F q)^\F}", from=2-2, to=2-3]
\arrow["{{\T^\F}^2q}"', from=2-2, to=3-2]
\arrow["\lrcorner"{anchor=center, pos=0.125}, draw=none, from=2-2, to=3-3]
\arrow["{q^\F}", from=2-3, to=2-4]
\arrow["{\T^\F q}"{description}, from=2-3, to=3-3]
\arrow["\lrcorner"{anchor=center, pos=0.125}, draw=none, from=2-3, to=3-4]
\arrow["q", from=2-4, to=3-4]
\arrow["{p_{\T M}}"', from=3-2, to=3-3]
\arrow["{p_M}"', from=3-3, to=3-4]
\arrow["{l_M}"', from=4-1, to=3-2]
\arrow["{p_M}"', curve={height=18pt}, from=4-1, to=3-4]
\end{tikzcd}\quad
\begin{tikzcd}
{\F\T^\F q} &&& {\F\T^\F q} \\
& {\F\T^\F q} & {\F q} & E \\
& {\T^2M} & {\T M} & M \\
{\T^2M} &&& {\T M}
\arrow["{(\T^\F q)^\F}", from=1-1, to=1-4]
\arrow["{c^\F_q}", dashed, from=1-1, to=2-2]
\arrow["{{\T^\F}^2q}"', from=1-1, to=4-1]
\arrow["{q^\F}", from=1-4, to=2-4]
\arrow["{(\T^\F q)^\F}", from=2-2, to=2-3]
\arrow["{{\T^\F}^2q}"', from=2-2, to=3-2]
\arrow["\lrcorner"{anchor=center, pos=0.125}, draw=none, from=2-2, to=3-3]
\arrow["{q^\F}", from=2-3, to=2-4]
\arrow["{\T^\F q}"{description}, from=2-3, to=3-3]
\arrow["\lrcorner"{anchor=center, pos=0.125}, draw=none, from=2-3, to=3-4]
\arrow["q", from=2-4, to=3-4]
\arrow["{p_{\T M}}"', from=3-2, to=3-3]
\arrow["{p_M}"', from=3-3, to=3-4]
\arrow["{c_M}"', from=4-1, to=3-2]
\arrow["{p_{\T M}}"', from=4-1, to=4-4]
\arrow["{p_M}"', from=4-4, to=3-4]
\end{tikzcd}
}
\end{equation*}

\begin{lemma}
\label{lemma:finsler-tangent-category}
The following diagrams
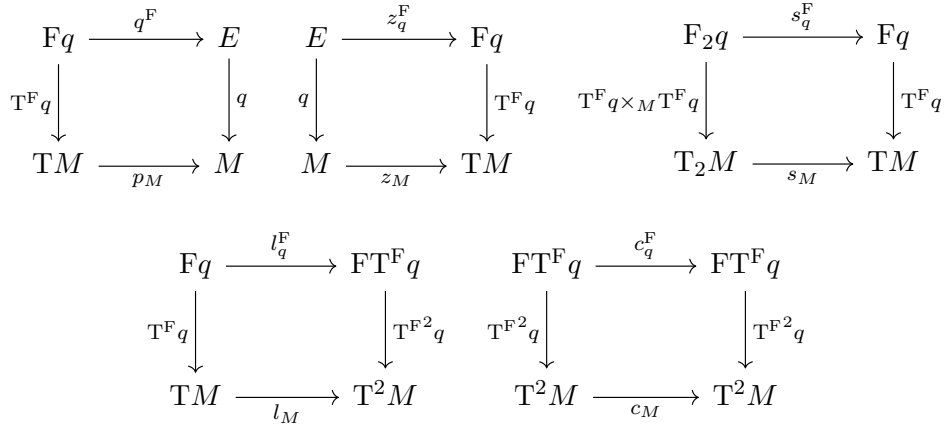
\begin{equation*}
\begin{tikzcd}
{\F q} & E \\
{\T M} & M
\arrow["{q^\F}", from=1-1, to=1-2]
\arrow["{\T^\F q}"', from=1-1, to=2-1]
\arrow["q", from=1-2, to=2-2]
\arrow["{p_M}"', from=2-1, to=2-2]
\end{tikzcd}\quad
\begin{tikzcd}
E & {\F q} \\
M & {\T M}
\arrow["{z_q^\F}", from=1-1, to=1-2]
\arrow["q"', from=1-1, to=2-1]
\arrow["{\T^\F q}", from=1-2, to=2-2]
\arrow["{z_M}"', from=2-1, to=2-2]
\end{tikzcd}\quad
\begin{tikzcd}
{\F_2q} & {\F q} \\
{\T_2M} & {\T M}
\arrow["{s_q^\F}", from=1-1, to=1-2]
\arrow["{\T^\F q\times_M\T^\F q}"', from=1-1, to=2-1]
\arrow["{\T^\F q}", from=1-2, to=2-2]
\arrow["{s_M}"', from=2-1, to=2-2]
\end{tikzcd}
\end{equation*}
\begin{equation*}
\begin{tikzcd}
{\F q} & {\F\T^\F q} \\
{\T M} & {\T^2M}
\arrow["{l^\F_q}", from=1-1, to=1-2]
\arrow["{\T^\F q}"', from=1-1, to=2-1]
\arrow["{{\T^\F}^2q}", from=1-2, to=2-2]
\arrow["{l_M}"', from=2-1, to=2-2]
\end{tikzcd}\quad
\begin{tikzcd}
{\F\T^\F q} & {\F\T^\F q} \\
{\T^2M} & {\T^2M}
\arrow["{c^\F_q}", from=1-1, to=1-2]
\arrow["{{\T^\F}^2q}"', from=1-1, to=2-1]
\arrow["{{\T^\F}^2q}", from=1-2, to=2-2]
\arrow["{c_M}"', from=2-1, to=2-2]
\end{tikzcd}
\end{equation*}
are tangent pullbacks in the base tangent category.
\end{lemma}
\begin{proof}
The first diagram is simply the tangent pullback diagram of Equation~\eqref{equation:curvature}. Now, consider the following diagram:
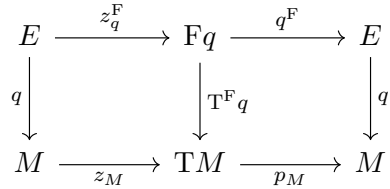
\begin{equation*}
\begin{tikzcd}
E & {\F q} & E \\
M & {\T M} & M
\arrow["{z_q^\F}", from=1-1, to=1-2]
\arrow["q"', from=1-1, to=2-1]
\arrow["{q^\F}", from=1-2, to=1-3]
\arrow["{\T^\F q}", from=1-2, to=2-2]
\arrow["q", from=1-3, to=2-3]
\arrow["{z_M}"', from=2-1, to=2-2]
\arrow["{p_M}"', from=2-2, to=2-3]
\end{tikzcd}
\end{equation*}
The right square diagram is the first diagram, which we have already shown to be a tangent pullback. Moreover, the outer diagram is straightforwardly a tangent pullback since $z_Mp_M=\id_M$ and $z_q^\F q^\F=\id_E$. Therefore, by Lemma~\ref{lemma:property-tangent-pullbacks}, the left square diagram is also a tangent pullback. To prove that the third diagram is a tangent pullback, we shall adopt a similar technique. Let us start by considering the following diagram:
\begin{equation}
\label{equation:sum-diagram-is-pullback}
\begin{tikzcd}
{\F_2q} & {\F q} & E \\
{\T_2M} & {\T M} & M
\arrow["{s_q^\F}", from=1-1, to=1-2]
\arrow["{\T^\F q\times_M\T^\F q}"', from=1-1, to=2-1]
\arrow["{q^\F}", from=1-2, to=1-3]
\arrow["{\T^\F q}", from=1-2, to=2-2]
\arrow["q", from=1-3, to=2-3]
\arrow["{s_M}"', from=2-1, to=2-2]
\arrow["{p_M}"', from=2-2, to=2-3]
\end{tikzcd}
\end{equation}
Once yet, the right diagram is a tangent pullback. Moreover, $s_Mp_M=\pi_1p_M=\pi_2p_M$ and $z_q^\F q^\F=\pi_1 q^\F=\pi_2 q^\F$. We shall prove that the outer diagram is again a tangent pullback, by considering the following general situation. Consider two tangent display maps $q\colon E\to M$ and $q'\colon E'\to M'$ and a morphism $(f,g)\colon q\to q'$ such that the commutative square $fq=gq'$ is a tangent pullback. Now, consider two maps $\alpha\colon X\to M'$ and $\beta\colon X\to E$ making the following diagram
\begin{equation*}
\adjustbox{width=.5\linewidth}{
\begin{tikzcd}
X &&&& \\
&& E \\
& {E_2} & {E'} && M \\
& {E'_2} && E & {M'} \\
&&& {E'}
\arrow["\gamma"{description}, curve={height=-30pt}, dashed, from=1-1, to=2-3]
\arrow["{\<\alpha,\beta\>}"{description}, dashed, from=1-1, to=3-2]
\arrow["\beta", curve={height=-24pt}, from=1-1, to=3-5]
\arrow["\alpha"', curve={height=18pt}, from=1-1, to=4-2]
\arrow["\gamma"{description, pos=0.2}, curve={height=-30pt}, dashed, from=1-1, to=4-4]
\arrow["g"{description}, from=2-3, to=3-3]
\arrow["q"{description}, from=2-3, to=3-5]
\arrow["{\pi_2}"{description}, from=3-2, to=2-3]
\arrow["{g\times_fg}"{description}, from=3-2, to=4-2]
\arrow["{\pi_1}"{description}, from=3-2, to=4-4]
\arrow["{q'}"{description}, from=3-3, to=4-5]
\arrow["f", from=3-5, to=4-5]
\arrow["{\pi_2}"{description}, from=4-2, to=3-3]
\arrow["{\pi_1}"{description}, from=4-2, to=5-4]
\arrow["q"{description}, from=4-4, to=3-5]
\arrow["g"{description}, from=4-4, to=5-4]
\arrow["{q'}"{description}, from=5-4, to=4-5]
\end{tikzcd}
}
\end{equation*}
commutative. As represented in the diagram, from the universal property of the tangent pullback given by $(f,g)\colon q\to q'$, there exists a unique map $\gamma\colon X\to E$, and by using the universal property of the top face diagram, we construct a unique map $\<\alpha,\beta\>\colon X\to E_2$, proving that the desired diagram, is in fact a tangent pullback. In particular, the outer diagram of Equation~\eqref{equation:sum-diagram-is-pullback} is then a tangent pullback and therefore, by invoking once again Lemma~\ref{lemma:property-tangent-pullbacks}, the third diagram is a tangent pullback. Finally, to prove that the last two diagrams are also tangent pullbacks, we may consider the following diagrams:
\begin{equation*}
\begin{tikzcd}
{\F q} & {\F\T^\F q} & {\F q} \\
{\T M} & {\T^2M} & {\T M}
\arrow["{l^\F_q}", from=1-1, to=1-2]
\arrow["{\T^\F q}"', from=1-1, to=2-1]
\arrow["{(\T^F q)^\F}", from=1-2, to=1-3]
\arrow["{{\T^\F}^2q}", from=1-2, to=2-2]
\arrow["{\T^\F q}", from=1-3, to=2-3]
\arrow["{l_M}"', from=2-1, to=2-2]
\arrow["{p_{\T M}}"', from=2-2, to=2-3]
\end{tikzcd}\quad
\begin{tikzcd}
{\F\T^\F q} & {\F\T^\F q} & {\F q} \\
{\T^2M} & {\T^2M} & {\T M}
\arrow["{c^\F_q}", from=1-1, to=1-2]
\arrow["{{\T^\F}^2q}"', from=1-1, to=2-1]
\arrow["{(\T^\F q)^\F}", from=1-2, to=1-3]
\arrow["{{\T^\F}^2q}", from=1-2, to=2-2]
\arrow["{\T^\F q}", from=1-3, to=2-3]
\arrow["{c_M}"', from=2-1, to=2-2]
\arrow["{p_{\T M}}"', from=2-2, to=2-3]
\end{tikzcd}
\end{equation*}
Once again, in both scenarios, the right hand square diagram is a tangent pullback. Furthermore, using that $l_Mp_{\T M}=p_Mz_M$ and that $l_q^\F(\T^\F q)^\F=q^\F z_q^\F$, and that $c_Mp_{\T M}=\T p_M$ and $c_q^\F(\T^\F q)^\F=\T^\F(q^\F)$, we shall rewrite the outer diagrams as follows:
\begin{equation*}
\begin{tikzcd}
{\F q} & E & {\F q} \\
{\T M} & M & {\T M}
\arrow["{q^\F}", from=1-1, to=1-2]
\arrow["{\T^\F q}"', from=1-1, to=2-1]
\arrow["{z_q^\F}", from=1-2, to=1-3]
\arrow["q", from=1-2, to=2-2]
\arrow["{\T^\F q}", from=1-3, to=2-3]
\arrow["{p_M}"', from=2-1, to=2-2]
\arrow["{z_q}"', from=2-2, to=2-3]
\end{tikzcd}\quad
\begin{tikzcd}
{\F\T^\F q} & {\F q} \\
{\T^2M} & {\T M}
\arrow["{\T^\F(q^\F)}", from=1-1, to=1-2]
\arrow["{{\T^\F}^2q}"', from=1-1, to=2-1]
\arrow["{\T^\F q}", from=1-2, to=2-2]
\arrow["{\T p_M}"', from=2-1, to=2-2]
\end{tikzcd}
\end{equation*}
The left hand side diagram is two tangent pullback diagrams stack to each other, while, for the right hand side one, we can argue it is a tangent pullback from the preservation of the pullback by the tangent bundle functor.
\end{proof}

\begin{proposition}
\label{proposition:finsler-tangent-category}
There is a tangent structure $\TT^\F$ over the category $\Dsply(\X,\TT)$ of tangent display maps of a tangent category $(\X,\TT)$, whose tangent bundle functor and structural natural transformations are $\T^\F$, $p^\F$, $z^\F$, $s^\F$, $l^\F$, and $c^\F$. Furthermore, if $(\X,\TT)$ admits negatives, so does $(\Dsply(\X,\TT),\TT^\F)$. Finally, the horizontal descent $(\id_{\T M},\pi_q)\colon\T q\to\T^\F q$ induces a morphism of tangent structures $(\id_{\T M},\pi_q)\colon\TT\Rightarrow\TT^\F$ on $\Dsply(\X,\TT)$.
\end{proposition}
\begin{proof}
We shall not complete all the details of this proof, since it is rather technical and not particularly enlightning. Instead, we simply argue that, by Lemma~\ref{lemma:finsler-tangent-category}, all the diagrams that define the structural natural transformations of $\TT^\F$ are in fact tangent pullbacks in $(\X,\TT)$, the equational axioms of a tangent category follow directly from $\TT$ being a tangent structure and by the universal property of these tangent pullbacks. To establish the universal property of the vertical lift, consider a morphism $(f,g)\colon q'\to{\T^\F}^2q$ which equalizes $\T^\F p^\F_q$ and $p^\F_{\T^\F q}p^\F qz^\F_q$. Using the universal property of the vertical lift $l_M$, we obtain a unique morphism $\hat f\colon M'\to\T_2M$. Moreover, from the universal properties of the tangent pullbacks of Lemma~\ref{lemma:finsler-tangent-category}, we can argue that the map $(\xi_M,\xi_q^\F)\colon\T^\F_2q\to{\T^\F}^2q$ that establishes the universal property of the lift, defines a tangent pullback in the base category, as well. Therefore, we can pull back $g$ along $\xi_M$ as follows:
\begin{equation*}
\begin{tikzcd}
{E'} &&& \\
&& {F_2q} & {\F\T^\F q} \\
&& {\T_2M} & {\T^2M} \\
{M'}
\arrow["{\hat g}", dashed, from=1-1, to=2-3]
\arrow["g", curve={height=-18pt}, from=1-1, to=2-4]
\arrow["{q'}"', from=1-1, to=4-1]
\arrow["{\xi_q^\F}", from=2-3, to=2-4]
\arrow["{\T_2^\F q}"', from=2-3, to=3-3]
\arrow["\lrcorner"{anchor=center, pos=0.125}, draw=none, from=2-3, to=3-4]
\arrow["{{\T^\F}^2q}", from=2-4, to=3-4]
\arrow["{\xi_M}"', from=3-3, to=3-4]
\arrow["{\hat f}"', dashed, from=4-1, to=3-3]
\arrow["f"', curve={height=18pt}, from=4-1, to=3-4]
\end{tikzcd}
\end{equation*}
This proves the universal property of the vertica lift for $(\Dsply(\X,\TT),\TT^\F)$. We leave it to the reader to show that the horizontal descent defines a morphism of tangent structures.
\end{proof}

One may argue that, since the commutative square diagrams that correspond to the definition of the structural natural transformations of the tangent category $(\Dsply(\X,\TT),\TT^\F)$ are tangent pullbacks by Lemma~\ref{lemma:finsler-tangent-category}, one may restrict the whole tangent category to a subcategory of $\Dsply(\X,\TT)$ defined as follows. We denote by $\Dsply_\pb(\X,\TT)$ the subcategory of $\Dsply(\X,\TT)$ with the same objects of $\Dsply(\X,\TT)$ but where morpshisms are pairs $(f,g)\colon q\to q'$ for which the diagram
\begin{equation*}
\begin{tikzcd}
E & {E'} \\
M & {M'}
\arrow["g", from=1-1, to=1-2]
\arrow["q"', from=1-1, to=2-1]
\arrow["{q'}", from=1-2, to=2-2]
\arrow["f"', from=2-1, to=2-2]
\end{tikzcd}
\end{equation*}
is a tangent pullback in $(\X,\TT)$.

\begin{proposition}
\label{proposition:finsler-strong-tangent-category}
The tangent structure $\TT^\F$ of Proposition~\ref{proposition:finsler-tangent-category} restricts to the subcategory $\Dsply_\pb(\X,\TT)$, defining a new tangent category $(\Dsply_\pb(\X,\TT),\TT^\F)$.
\end{proposition}
\begin{proof}
We leave the proof of this proposition to the reader.  
\end{proof}

\begin{definition}
\label{definition:finsler-tangent-category}
We call the tangent category of Proposition~\ref{proposition:finsler-tangent-category} the \textbf{Finsler tangent category} of $(\X,\TT)$, while the tangent category of Proposition~\ref{proposition:finsler-strong-tangent-category} will be called the \textbf{strong Finsler tangent category} of $(\X,\TT)$. We shall denote the former by $\F(\X,\TT)$, and the latter by $\F_\pb(\X,\TT)$.
\end{definition}

Our goal was to construct a tangent category whose vector fields correspond to Finsler vector fields. Our constructions of the strong and the non-strong Finsler tangent categories get very close to this.  

\begin{lemma}
\label{lemma:finsler-vector-fields-as-vector-fields}
A vector field in either the Finsler tangent category $\F(\X,\TT)$ or the strong Finsler tangent category $\F_\pb(\X,\TT)$ consists precisely of a tangent display map $q\colon E\to M$ together with a pair $(X,A_\F)$ formed by an ordinary vector field $X\colon M\to\T M$ on $M$ in the base tangent category together with a section $A_\F\colon E\to\F q$ of the Finsler bundle $\q^\F$ of $q$, satisfying the following condition:
\begin{equation*}
\begin{tikzcd}
E & {\F q} \\
M & {\T M}
\arrow["A", from=1-1, to=1-2]
\arrow["q"', from=1-1, to=2-1]
\arrow["{\T^\F q}", from=1-2, to=2-2]
\arrow["X"', from=2-1, to=2-2]
\end{tikzcd}
\end{equation*}
Furthermore, this diagram is necessarily a tangent pullback.
\end{lemma}

By Theorem~\ref{theorem:finsler-vector-fields}, sections of the Finsler bundle are in bijective correspondence with Finsler vector fields. Unfortunately, not every Finsler vector field $A\colon E\to\T M$ has an underlying ordinary vector field $X\colon M\to\T M$.

\begin{definition}
\label{definition:full-finsler-vector-field}
A \textbf{full Finsler vector field} over a map $q\colon E\to M$ consists of a pair $(X,A)$ formed by a vector field $X\colon M\to\T M$ on $M$ together with a Finsler vector field $A\colon E\to\T M$ over $q$, satisfying the following condition:
\begin{equation*}
\begin{tikzcd}
E & {\T M} \\
M
\arrow["A", from=1-1, to=1-2]
\arrow["q"', from=1-1, to=2-1]
\arrow["X"', from=2-1, to=1-2]
\end{tikzcd}
\end{equation*}
\end{definition}

\begin{theorem}
\label{theorem:finsler-vector-fields-as-vector-fields}
Consider a tangent display map $q\colon E\to M$. Given a tangent display map $q\colon E\to M$, the following are equivalent:
\begin{enumerate}
\item A full vector field $(X,A)$ over $q$;

\item A vector field $(X,A_\F)$ over $q$ in the Finsler tangent category;

\item A vector field $(X,A_\F)$ over $q$ in the strong Finsler tangent category.
\end{enumerate}
In particular, full Finsler vector fields over tangent display maps define a tangent category $\VF(\F(\X,\TT))$.
\end{theorem}

\begin{remark}
\label{remark:lie-bracket-full-finsler-vector-fields}
An important consequence of Theorem~\ref{theorem:finsler-vector-fields-as-vector-fields} is that, when the base tangent category $(\X,\TT)$ admits negatives, there is a well-defined notion of Lie bracket of full Finsler vector fields. This is surprising since, as we will explain later, horizontal vector fields (which are related to Finsler vector fields) are not closed under the Lie bracket. In particular, the failure of the Lie bracket of two horizontal vector fields to be horizontal is measured by the curvature.

However, it is crucial to remember which tangent category the Lie bracket is defined in. The Lie bracket of full Finsler vector fields is induced by the Finsler tangent structure, while the Lie bracket under which horizontal vector fields fail to be closed is induced by the base tangent category.

One may expect that the two Lie brackets should agree, since there is a comparison morphism $(\id_{\T M},\pi_q)\colon\TT\Rightarrow\TT^\F$ between the two corresponding tangent structures. However, this is not the case, since this morphism of tangent structures induces a functor $\VF(\Dsply(\X,\TT),\TT)\to\F(\X,\TT)$ from the tangent category with the ordinary tangent structure to the Finsler one, not the other way around. Finally, one should also observe that not every Finsler vector field is full.
\end{remark}


\section{Distributions and involutivity}
\label{section:distributions}
In the previous section, we introduced vertical and Finsler vector fields, which correspond to sections of the vertical and of the Finsler bundle, respectively. It is often convenient to regard vertical vector fields not merely as sections of the vertical bundle but instead as maps with values directly in the tangent bundle $\T E$ of the total space $E$ of $q\colon E\to M$. A similar situation happens in the presence of a (horizontal) connection $\H\colon\F q\to\T E$ which embeds the Finsler bundle into the tangent bundle and allows us to regard Finsler vector fields as maps with values in $\T E$. This perspective is facilitated by the concept of \emph{distributions}.

In differential geometry, a \emph{distribution} over an object $E$ consists of a vector sub-bundle $D\subseteq\T E$ of the tangent bundle of $E$. In this section, we generalize this concept to tangent categories, and prove an important property of the vertical bundle: as a distribution, the vertical bundle is \emph{involutive}. Concretely, this means that vertical vector fields are closed under the Lie bracket associated to the tangent bundle of $E$. We begin by formally introducing the concept of a distribution.

\begin{definition}
\label{definition:distribution}
In a tangent category $(\X,\TT)$, a \textbf{tangent distribution} on $M$ consists of a differential bundle $\q\colon D\to M$ together with a morphism $\iota\colon D\to\T M$ subject to the following conditions:
\begin{description}
\item[TD.1] $\iota$ is tangent monic, that is, for every $n\geq0$, $\T^n\iota$ is monic;

\item[TD.2] $\iota\colon\q\to\p_M$ is a linear morphism of differential bundles over $M$.
\end{description}
Given a tangent distribution ($\q\colon D\to M,\iota)$ of $M$, a \textbf{$(\q,\iota)$-vector field} is a vector field $X\colon M\to\T M$ of $M$ such that there exists a (necessarily unique) section $\hat X\colon M\to D$ of the projection $q\colon D\to M$ of the distribution, satisfying $X=\hat X\iota$.
\end{definition}

Thanks to Lemma~\ref{lemma:iota-tangent-monic}, the vertical bundle $\q^\V\colon\V q\to E$ of a tangent display map $q\colon E\to M$ is a tangent distribution on $E$, where the inclusion map is $\iota_q$.

\begin{definition}
\label{definition:vertical-distribution}
The \textbf{vertical distribution} of a tangent display map $q\colon E\to M$ is the tangent distribution $(\q^\V,\iota_q)$, where $\q^\V$ denotes the vertical bundle of $\q$ and $\iota_q$ the inclusion of $\V q$ into $\T E$. We shall denote this distribution simply by $\q^\V$.
\end{definition}

Using the language of distributions, vertical vector fields correspond precisely to $\q^\V$-vector fields. We now introduce the involutivity condition.

\begin{definition}
\label{definition:involutive-distribution}
In a tangent category with negatives, a tangent distribution $(\q,\iota)$ is \textbf{involutive} provided that $\q$-vector fields form a Lie subalgebra $\VF(\X,\TT;\q,\iota)$ of the Lie algebra $\VF(\X,\TT;M)$ of vector fields. Concretely, for every pair of sections $\hat X,\hat Y\colon M\to D$ of the projection $q\colon D\to M$, there exists a (necessarily unique) section $\hat Z\colon M\to D$ satisfying the following equation
\begin{align*}
&[\hat X\iota,\hat Y\iota]=\hat Z\iota
\end{align*}
where $[,]$ denotes the Lie bracket of the tangent bundle $\p_M$.
\end{definition}

Not only is the vertical bundle a distribution, but crucially, it is also involutive. This is a consequence of Theorem~\ref{theorem:vertical-vector-fields}.

\begin{corollary}
\label{corollary:vertical-bundle-involutive}
In a tangent category with negatives, the vertical bundle $\q^\V$ of a tangent display map $q\colon E\to M$ is an involutive distribution on $E$.
\end{corollary}
\begin{proof}
One can prove this directly by showing that $[X,Y]\T q=qz_M$. Instead, we will give a more conceptual proof, which involves Theorem~\ref{theorem:vertical-vector-fields}. By this result, a vertical vector field $X\colon E\to\T E$ is equivalent to a vector field $\hat X\colon q\to\T^\V q$ on $q$ in the slice tangent category. If $\hat X$ and $\hat Y$ denote the corresponding vector fields in the slice of two vertical vector fields $X$ and $Y$, one can compute the Lie bracket of $\hat X$ and $\hat Y$ in the slice tangent category and obtain a vector field $[\hat X,\hat Y]^\V\colon q\to\T^\V q$.

Furthermore, each lax tangent morphism $(F,\alpha)\colon(\X,\TT)\to(\X',\TT')$ lifts to a lax tangent morphism $\VF[F,\alpha]\colon\VF(\X,\TT)\to\VF(\X',\TT')$ which sends every vector field $X$ to $FX\alpha_M\colon FM\to\T'FM$. In particular, the tangent morphism $\Pi\colon(\X,\TT)/M\to(\X,\TT)$ of Lemma~\ref{lemma:tangent-morphism-slice} lifts to a tangent morphism $\VF[\Pi]\colon\VF(\X,\TT)/M)\to\VF(\X,\TT)$. However, by Theorem~\ref{theorem:vertical-vector-fields}, the image of a vector field in $(\X,\TT)/M$ along  $\Pi$ is a vertical vector field.

Thus $\VF[\Pi]([\hat X,\hat Y]^\V)$ is a vertical vector field of $q$. In particular, $\VF[\Pi]([\hat X,\hat Y]^\V)$ corresponds to the vertical vector field $[\hat X,\hat Y]^\V\iota_q$ of $q$. However, for every lax tangent morphism $(F,\alpha)$, $\VF[F,\alpha]$ preserves the Lie bracket. So, in particular, we obtain:
\begin{align*}
&[\hat X,\hat Y]^\V\iota_q=[\hat X\iota_q,\hat Y\iota_q]=[X,Y]
\end{align*}
proving that the Lie bracket $[X,Y]$ of vertical vector fields $X$ and $Y$ is again a vertical vector field.
\end{proof}

The vertical bundle defines an involutive tangent distribution, and the corresponding $\q^\V$-vector fields are precisely the vertical vector fields. One might expect that Finsler bundles should give rise to $\q^\F$-vector fields. Unfortunately, the Finsler bundle is, in general, not a tangent distribution for $E$, let alone being involutive. In fact, a connection is exactly the mathematical device required to make $\q^\F$ into a tangent distribution, the \emph{horizontal distribution}. We shall start by proving that, given a horizontal connection, $(\q^\F,\H)$ defines a tangent distribution.

\begin{lemma}
\label{lemma:horizontal-distribution}
A horizontal connection $\H\colon\F q\to\T E$ of a tangent display map $q\colon E\to M$ turns the Finsler bundle $\q^\F$ of $q$ into a tangent distribution of $E$.
\end{lemma}
\begin{proof}
By \textbf{[HC.2]}, $\H$ is a linear morphism; moreover, by \textbf{[HC.1]}, $\H$ is a section of the horizontal descent of $q$. Thus, $\H$ is tangent monic.
\end{proof}

\begin{definition}
\label{definition:horizontal-distribution}
The \textbf{horizontal distribution} of a tangent display map $q\colon E\to M$ equipped with a horizontal connection $\H$ is the tangent distribution $(\q^\F,\H)$ of Lemma~\ref{lemma:horizontal-distribution}. We shall denote this distribution by $\q^\H$.
\end{definition}

As vertical vector fields correspond to $\q^\V$-vector fields, so horizontal vector fields correspond to $\q^\H$-vector fields.


\section{Curvature and parallel transport}
\label{section:curvature}

In this final section, we discuss how to define and work with curvature and parallel transport for an Erhesmann connection in a tangent category.

\subsection{The curvature of an Ehresmann connection}
\label{subsection:curvature}
Corollary~\ref{corollary:vertical-bundle-involutive} proves that the vertical distribution is involutive, that is, vertical vector fields are closed under the Lie bracket, and they form a Lie subalgebra of the Lie algebra of ordinary vector fields. One may hope that the horizontal distribution also satisfies the involutivity condition. However, this is not the case. In fact, the failure of horizontal vector fields to be closed under Lie bracket is measured by an important geometric quantity, known as the \emph{curvature} of the connection. In this section, we introduce this concept in the context of tangent categories and prove that the exterior derivative of the curvature vanishes for horizontal vector fields, giving a Bianchi identity for this notion of curvature.

\par However, instead of directly defining the curvature as a measure of the failure of the horizontal distribution to be involutive, we shall take a different route, which enlightens the categorical nature of tangent categories. We will prove later on that our definition agrees with the classical notion.  However, we begin with a philosophical detour to justify our intuition.

\par From a very general point of view, a \emph{categorical} theory of a certain class of objects aims to provide the minimal setting in which such an object may be defined. For instance, categorically, a group consists of an object $G$ in a Cartesian category $\X$, equipped with three maps $e\colon\*\to G$, $\mu\colon G\times G\to G$, and $\iota\colon G\to G$, axiomatizing the unit, the multiplication, and the inverse operation. This definition not only makes sense in the standard setting where $G$ is a set, but it also applies to new contexts, including the category of group objects $\CategoryFont{Grp}(\X)$ of a Cartesian category $\X$, itself.

\par It turns out that a group object in the category of group objects consists of an object $G$ in the original category together with two group structures, which, however, must necessarily coincide and make $G$ into an \emph{Abelian} group. This is not a special feature of group objects, but rather a general phenomenon: a type of structure in the category of those structures is two of these structures with a comparison law between them. For example, a monoid in monoids is a commutative monoid and a monad in the $2$-category of monads consists of two monads over the same object together with a distributive law between those monads~\cite{street:formal-theory-monads}.

\par Even in the context of tangent categories, this phenomenon appears for vector fields: a vector field in the tangent category of vector fields is a pair of commutative vector fields, that is, a pair $X,Y\colon M\to\T M$ of vector fields, satisfying $[X,Y]=0$.

\par Our proposed definition for the curvature makes use of the following idea. A non-flat connection \emph{fails} to be a morphism of connections, and the obstruction is measured by the curvature. This idea was inspired by the definition of the curvature of a \emph{Koszul connection} $\K\colon\T E\to E$ on a differential bundle $\q\colon E\to M$, proposed in~\cite[Definition~3.20]{cockett:connections}, given by the formula:
\begin{align*}
&\Curv_\K=c_E\T\K\K-\T\K\K\colon\T^2E\to E
\end{align*}
From this expression, we have noticed that a Koszul connection $\K$ is in fact flat, that is, the curvature vanishes, if and only if $\K\colon(\T\q,c_E\T\K)\to(\T\q,c_E\T\K)$ is a morphism of Koszul connections, where $c_E\T\K$ turns out to be a Koszul connection on $\T\q$.

\par To make this idea precise, we need to modify slightly the context we are working on. Given a tangent display submersion $q\colon E\to M$, equipped with an Ehresmann connection $(\R,\H)$, the fundamental short exact sequence of $q$ is the sequence:
\begin{equation*}
\begin{tikzcd}
{\0_E} & {\q^\V} & {\p_E} & {\q^\F} & {\0_E}
\arrow[from=1-1, to=1-2]
\arrow["{\iota_q}", from=1-2, to=1-3]
\arrow["{\pi_q}", from=1-3, to=1-4]
\arrow[from=1-4, to=1-5]
\end{tikzcd}
\end{equation*}
Thanks to the properties of the vertical and the Finsler bundle, one can apply the tangent bundle functor and obtain a new short exact sequence:
\begin{equation*}
\begin{tikzcd}
{\T(\0_E)} & {\T(\q^\V)} & {\T(\p_E)} & {\T(\q^\F)} & {\T(\0_E)}
\arrow[from=1-1, to=1-2]
\arrow["{\T\iota_q}", from=1-2, to=1-3]
\arrow["{\T\pi_q}", from=1-3, to=1-4]
\arrow[from=1-4, to=1-5]
\end{tikzcd}
\end{equation*}
Notice that this is not the fundamental short exact sequence of $\T q$, because the central term is $\T\p_E$ instead of $\p_{\T E}$. One might also argue that, $\T(\q^\V)$ and $\T(\q^\F)$ are not the vertical and the Finsler bundle of $\T q$, since, by notation, they should be $(\T q)^\V$ and $(\T q)^\F$, instead. However, one should remember that the vertical bundle and the Finsler bundle of a tangent display map are only specified up to a unique isomorphism (Remarks~\ref{remark:vertical-bundle-choice} and~\ref{remark:finler-bundle-choice}), since they are defined by a choice of tangent pullbacks. This means that $\T(\q^\V)$ and $\T(\q^\F)$ are, in fact, different, but legitimate representations of the vertical and the Finsler bundles of $\T q$. With that in mind, we can use the canonical flip to make the central term, back to $\p_{\T E}$, obtaining a correct representation of the fundamental short exact sequence of $\T q$:
\begin{equation*}
\begin{tikzcd}
{\0_{\T E}} & {\T(\q^\V)} & {\p_{\T E}} & {\T(\q^\F)} & {\0_{\T E}}
\arrow[from=1-1, to=1-2]
\arrow["{\T\iota_q c_E}", from=1-2, to=1-3]
\arrow["{c_E \T\pi_q}", from=1-3, to=1-4]
\arrow[from=1-4, to=1-5]
\end{tikzcd}
\end{equation*}
With this choice of the vertical and the Finsler bundles, an Ehresmann connection $(\R,\H)$ on $q$, defines a new Ehresmann connection on $\T q$ given by $(c_E\T\R,\T\H c_E)$. Notice that with the standard choice of the vertical and the Finsler bundle, that is, $(\T q)^\V$ and $(\T q)^\F$, the corresponding Ehresmann connection on $\T q$ would be $(c_E\T\R\gamma^\V,\gamma^\F\T\H c_E)$, instead. Unfortunately, this choice does not work for our purposes, which is why we need to work with $\T(\q^\V)$ instead of $(\T q)^\V$.

\par It is easy to work with the connection forms, instead of directly with the vertical and the horizontal connections. So, in particular, $(c_E\T\phi,\T\psi c_E)$ defines an Ehresmann connection form for $\T q$.

\par Now, recall that given two tangent display maps $q\colon E\to M$ and $q'\colon E'\to M$ over the same object $M$, respectively equipped with a vertical connection form $\phi$ and $\phi'$, a \textbf{morphism of vertical connection forms} consists of a morphism $f\colon E\to E'$ that commutes with $q$ and $q'$, that is, $fq'=q$, and commutes with the connection forms; that is, $\T f\phi'=\phi\T f$.

\par We are now in the position to define the concepts of flatness and of curvature of an Ehresmann connection.

\begin{definition}
\label{definition:flat-connection}
An Ehresmann connection $(\R,\H)$ on a tangent display map $q\colon E\to M$ is \textbf{flat} provided that the following diagram
\begin{equation}
\label{equation:curvature}
\begin{tikzcd}
{\T^2E} & {\T^2E} \\
{\T^2E} & {\T^2E}
\arrow["{c_E\T\phi}", from=1-1, to=1-2]
\arrow["{\T\psi}"', from=1-1, to=2-1]
\arrow["{\T\psi}", from=1-2, to=2-2]
\arrow["{c_E\T\phi}"', from=2-1, to=2-2]
\end{tikzcd}
\end{equation}
commutes, where $\phi$ and $\psi$ denote the vertical and the horizontal connection forms of the connection.
\end{definition}

We may now give a fully categorical characterization of flatness.

\begin{proposition}
\label{proposition:flatness}
Consider a tangent display map $q\colon E\to M$. An Ehresmann connection $(\R,\H)$ is flat if and only if the horizontal connection form $\psi$ defines a morphism of vertical connection forms $(\id_{\T M},\psi)\colon(\T q,c_E\T\phi)\to(\T q,c_E\T\phi)$.
\end{proposition}
\begin{proof}
Let us start by assuming that the connection is flat. From Lemma~\ref{lemma:horizontal-connection}, we compute that $\psi\T q=\pi_q\H\T q=\pi_q\T^\F q=\T q$. Thus, $(\id_{\T M},\psi)\colon\T q\to\T q$ defines a morphism of tangent display maps. Moreover, the flatness condition establishes precisely that $\T\psi c_E\T\phi=c_E\T\phi\T\psi$, that is, that $(\id_{\T M},\psi)$ is a morphism of vertical connection forms. Conversely, the latter condition is precisely flatness.
\end{proof}

When the tangent category admits negatives, one can measure the failure of a connection to be flat by quantifying the failure of the diagram of Equation~\eqref{equation:curvature} to commute.

\begin{definition}
\label{definition:curvature}
The (\textbf{categorical}) \textbf{curvature} of an Ehresmann connection $(\R,\H)$ on a tangent display map $q\colon E\to M$ in a tangent category with negatives is the morphism
\begin{align*}
&\Curv_{(\R,\H)}\=c_E\T\phi\T\psi-\T\psi c_E\T\phi\colon\T^2E\to\T^2E
\end{align*}
\end{definition}

\begin{remark}
\label{remark:curvature-vertical}
Notice that the formula that defines the curvature of an Ehresmann connection involves both the vertical and the horizontal connection. However, as extensively discussed in the previous sections, with negatives, one is fully determined by the other. Therefore, it is natural to wonder if one may express the curvature entirely using the vertical or the horizontal component of the connection. This is possible, thanks to \textbf{[ECF.2]}, which establishes that $\phi+\psi=\id_{\T E}$. In particular, this allows us to write $\phi$ as $\id_{\T E}-\psi$ and viceversa, $\psi=\id_{\T E}-\phi$, when negatives are available. After some obvious computations, one obtains the following two formulas:
\begin{align*}
&\Curv_\R\=p_{\T E}\T z_E-c_E\T\phi+\T\phi c\T\phi    &&\Curv_\H\=p_{\T E}\T z_E-\T\psi c_E+\T\psi c_E\T\psi
\end{align*}
We will refer to $\Curv_\R$ and $\Curv_\H$ as the \textbf{vertical curvature} of a vertical connection and the \textbf{horizontal curvature} of a horizontal connection. When $\R$ and $\H$ are the vertical and the horizontal components of an Ehresmann connection, then $\Curv_\R=\Curv_{(\R,\H)}=\Curv_\H$.
\end{remark}

To our knowledge, the definition of curvature does not exist in the literature in this form. In fact, this definition relies on the categorical nature of tangent categories: the categorical curvature measures the failure of the diagram~\eqref{equation:curvature} to commute. This formulation of the curvature is \emph{intrinsic} since it does not rely on the language of tensors or differential forms. Pictorially, we can represent the curvature as a $2$-cell
\begin{equation*}
\begin{tikzcd}
{\T^2E} & {\T^2E} \\
{\T^2E} & {\T^2E}
\arrow["{c_E\T\phi}", from=1-1, to=1-2]
\arrow["{\T\psi}"', from=1-1, to=2-1]
\arrow["{\Curv_{(\R,\H)}}"{description}, Rightarrow, from=1-2, to=2-1]
\arrow["{\T\psi}", from=1-2, to=2-2]
\arrow["{c_E\T\phi}"', from=2-1, to=2-2]
\end{tikzcd}
\end{equation*}
that measures the failure of the horizontal connection form to be a morphism of vertical connection forms. In the following, we shall assume that the ambient tangent category admits negatives.

\begin{lemma}
\label{lemma:curvature}
The curvature $\Curv_{(\R,\H)}$ of an Ehresmann connection $(\R,\H)$ can be rewritten as follows:
\begin{align*}
&\Curv_{(\R,\H)}=p_{\T E}\T z_E-\T\psi c_E\T\phi
\end{align*}
\end{lemma}
\begin{proof}
The proof follows immediately from the following computation. Recall that, from \textbf{[ECF.2]}, we can write $\psi$ as $\id_{\T E}-\phi$. Thus, using the fact that $\phi$ is an idempotent, we shall compute that
\begin{align*}
&\T\phi\T\psi=\T\phi(\id_{\T^2E}\ominus\T\phi)=\T\phi\ominus\T\phi^2=\T\phi\ominus\T\phi=\T p_E\T z_E
\end{align*}
where $\ominus$ denotes the difference operator induced by $\T s_E$ and $\T n_E$ (in contrast with $-$ which is the difference operator induced by $s_{\T E}$ and $n_{\T E}$). From this computation, it is immediate to rewrite $c_E\T\psi\T\phi$ as $p_{\T E}\T z_E$, since $c_E\T p_E=p_{\T E}$.
\end{proof}

Thanks to the previous lemma, we shall now prove that the curvature precisely measures the failure of the horizontal distribution to be involutive.

\begin{theorem}
\label{theorem:curvature}
Consider a tangent display map $q\colon E\to M$ equipped with an Ehresmann connection $(\R,\H)$. Given two horizontal vector fields $X_\H,Y_\H\colon E\to\T E$, the following formula holds:
\begin{align*}
&\left[X_\H,Y_\H\right]_\V=\left\{Y_\H\T X_\H\Curv_{(\R,\H)}\right\}
\end{align*}
where $Z_\V\=Z\phi$ denotes the vertical component of a vector field $Z\colon E\to\T E$. Moreover, if the connection $(\R,\H)$ is flat, the horizontal distribution $\q^\H$ is involutive.
\end{theorem}
\begin{proof}
Consider two horizontal vector fields $X_\H$ and $Y_\H$. By definition, the Lie bracket of $X_\H$ and $Y_\H$ can be written as follows:
\begin{align*}
&\left[X_\H,Y_\H\right]=\left\{X_\H\T Y_\H-Y_\H\T X_\H c_E\right\}
\end{align*}
Now, using the properties of the bracket $\{-\}$ proved in~\cite[Lemma~2.14]{cockett:tangent-cats}, $\{f\}\phi=\{f\T\phi\}$, provided that the two sides of the equation are well-defined. However, when $\phi$ is linear, both the two sides are defined provided that $\{f\}$ is. Therefore
\begin{align*}
&\left[X_\H,Y_\H\right]_\V=\left[X_\H,Y_\H\right]\phi=\left\{X_\H\T Y_\H-Y_\H\T X_\H c_E\right\}\phi=\left\{\left(X_\H\T Y_\H-Y_\H\T X_\H c_E\right)\T\phi\right\}=\\
&\qquad=\left\{X_\H\T Y_\H\T\phi-Y_\H\T X_\H c_E\T\phi\right\}
\end{align*}
where, in the last step, we used the naturality of $s$ and $n$ to distribute $\T\phi$. However, since by Corollary~\ref{corollary:horizontal-vector-fields-connection} $X_\H=X_\H\psi$ and $Y_\H=Y_\H\psi$, we can also write:
\begin{align*}
&X_\H\T Y_\H\T\phi-Y_\H\T X_\H c_E\T\phi=X_\H\T Y_\H\T\psi\T\phi-Y_\H\T X_\H\T\psi c_E\T\phi
\end{align*}
By \textbf{[ECF.1]}, $\psi\phi=p_Ez_E$, therefore, we can rewrite the first term as follows
\begin{align*}
&X_\H\T Y_\H\T\psi\T\phi=X_\H\T Y_\H\T p_E\T z_E=X_\H\T(Y_\H p_E)\T z_E=X_\H\T z_E
\end{align*}
where we used that $Y_\H$ is a vector field and therefore, $Y_\H p_E=\id_E$. Now, the second term can be rearranged using Lemma~\ref{lemma:curvature}, as follows:
\begin{align*}
&-Y_\H\T X_\H\T\psi c_E\T\phi=Y_\H\T X_\H(p_{\T E}\T z_E-\T\psi c_E\T\phi-p_{\T E}\T z_E)=Y_\H\T X_\H(\Curv_{(\R,\H)}-p_{\T E}\T z_E)=\\
&\qquad=Y_\H\T X_\H\Curv_{(\R,\H)}-Y_\H\T X_\H p_{\T E}\T z_E
\end{align*}
However, using the naturality of $p$ and that $Y_\H$ is a vector field, we can also write the last term as follows:
\begin{align*}
&Y_\H\T X_\H p_{\T E}\T z_E=Y_\H p_EX_\H\T z_E=X_\H\T z_E
\end{align*}
Putting all the pieces together, we finally compute:
\begin{align*}
&\left[X_\H,Y_\H\right]_\V=\left\{X_\H\T Y_\H\T\psi\T\phi-Y_\H\T X_\H\T\psi c_E\T\phi\right\}=\left\{X_\H\T z_E+Y_\H\T X_\H\Curv_{(\R,\H)}-X_\H\T z_E\right\}
\end{align*}
The terms $X_\H\T z_E$ cancel each other out, leaving us with the desired formula:
\begin{align*}
&\left[X_\H,Y_\H\right]_\V=\left\{Y_\H\T X_\H\Curv_{(\R,\H)}\right\}
\end{align*}
Finally, if the connection is flat, $\Curv_{(\R,\H)}$ vanishes, that is, $[X_\H,Y_\H]_\V=0$. This means that the vertical component of the Lie bracket $[X_\H,Y_\H]$ of two horizontal vector fields vanishes. However, thanks to Proposition~\ref{theorem:unique-decomposition-of-vectors}, a vector field $Z\colon E\to\T E$ can be uniquely decomposed into the sum $Z_\V+Z_\H$ of a vertical vector field $Z_\V$ and a horizontal vector field $Z_\H$. Therefore, if the vertical component $[X_\H,Y_\H]_\V$ vanishes, means that $[X_\H,Y_\H]$ is necessarily horizontal, that is, the horizontal distribution is involutive.
\end{proof}

In general, requiring the horizontal distribution to be involutive does not suffice to show that the connection is flat. However, this implication holds when every morphism $f\colon\T^2E\to N$ is uniquely determined by the assignment
\begin{align*}
&\hat f\colon\HVF(\X,\TT;q)\times\HVF(\X,\TT;q)\to\X(E,N) &&\hat f(X_\H,Y_\H)\=Y_\H\T X_\H f
\end{align*}
which sends a pair of horizontal vector fields to the map $Y_\H\T X_\H F$. In this case, we shall say that horizontal vector fields \emph{separate} the morphisms of type $\T^2E\to N$.

\begin{definition}
\label{definition:separability}
In a tangent category $(\X,\TT)$, horizontal vector fields \textbf{separate linear maps} provided that, given a tangent display map $q\colon E\to M$ equipped with a horizontal connection $\H$, for every $n\geq0$, for every differential bundle $\q'$ and every pair of linear morphisms $f,g\colon\T^n\p_E\to\q'$ of differential bundles, the equations
\begin{align*}
&X_\H^1\T X_\H^2\dots\T^nX_\H^nf=X_\H^1\T X_\H^2\dots\T^nX_\H^ng
\end{align*}
for every $n$-tuple $X_\H^1\,X_\H^n$ of horizontal vector fields, imply that $f=g$.
\end{definition}

\begin{corollary}
\label{corollary:curvature}
In a tangent category with negatives where horizontal vector fields separate linear maps, an Ehresmann connection $(\R,\H)$ on a tangent display map $q\colon E\to M$ is flat if and only if the horizontal distribution $\q^\H$ is involutive.
\end{corollary}
\begin{proof}
By Theorem~\ref{theorem:curvature}, we already know that if the connection is flat, the horizontal distribution is necessarily involutive. Now, assume that the horizontal distribution is involutive, that is, for every pair of horizontal vector fields $X_\H$ and $Y_\H$, the Lie bracket $[X_\H,Y_\H]$ is again a horizontal vector field. From the decomposition of Proposition~\ref{theorem:unique-decomposition-of-vectors}, we conclude that the vertical component $[X_\H,Y_\H]_\V$ must vanish. Therefore, by the formula of Theorem~\ref{theorem:curvature}, $\{Y_\H\T X_\H\Curv_{(\R,\H)}\}=0$. However, using the universal property of the vertical lift, this implies that $Y_\H\T X_\H\Curv_{(\R,\H)}$ must vanish, that is, $Y_\H\T X_\H\Curv_{(\R,\H)}=z_Ez_{\T E}$. However, since $X_\H$ and $Y_\H$ are vector fields, $z_Ez_{\T E}=X_\H p_EY_\H p_Ez_Ez_{\T E}z_Ez_{\T E}$, which, by using the naturality of $p$, can be rewritten as $X_\H p_EY_\H p_Ez_Ez_{\T E}=X_\H\T Y_\H p_{\T E}p_Ez_Ez_{\T E}$. Finally, by the separability condition, this implies that $\Curv_{(\R,\H)}=p_{\T E}p_Ez_Ez_{\T E}$. Therefore, since the curvature vanishes, the connection is flat.
\end{proof}

The separability condition holds in important cases, such as in the tangent category of finite-dimensional smooth manifolds. This is one of the reasons why differential geometers prefer to work directly with the curvature tensor instead of with an intrinsic notion of curvature as in Definition~\ref{definition:curvature}.

\subsection{The structural equation and the Bianchi identity}
\label{subsection:bianchi}
An important identity proved in differential geometry, known as Cartan's structural equation, used in the context of principal connections on principal bundles~\cite{hamilton:gauge-theory}, relates the curvature of a connection with the exterior differential of the connection. In this section, we extend this identity to the tangent-categorical context by introducing a proper notion of the exterior differential. We will use this identity to prove another important identity: the Bianchi identity. In gauge field theory, this identity establishes a local law of conservation of energy. We advise the interested reader to see~\cite{hamilton:gauge-theory}.

We may start by introducing the exterior differential. As before, let us assume the ambient tangent category has negatives.

\begin{definition}
\label{definition:differential-forms}
A (\textbf{vector-valued}) \textbf{differential form of rank $n$} on an object $E$ consists of a function
\begin{align*}
&\omega\colon\VF_n(\X,\TT;E)\to\VF(\X,\TT;E)
\end{align*}
where $\VF_n(\X,\TT;E)\=\VF(\X,\TT;E)\times{\dots}\times\VF(\X,\TT;E)$ and $n\geq1$. Furthermore, $\omega$ is required to be \textbf{additive in each variable}, that is,
\begin{align*}
&\omega(X_1\,0\,X_n)=0\\
&\omega(X_1\,X_i+Y_i\,X_n)=\omega(X_1\,X_i\,X_n)+\omega(X_1\,Y_i\,X_n)
\end{align*}
and \textbf{alternating}, that is
\begin{align*}
&\omega(X_{\sigma(1)}\,X_{\sigma(n)})=(-)^\sigma\omega(X_1\,X_n)
\end{align*}
where $(-)^\sigma$ denotes the sign of the permutation $\sigma$ on $n$ indices.
\end{definition}

We may define the \textbf{exterior derivative} of a differential form $\omega$ by the following formula
\begin{align}
\begin{split}
\label{equation:exterior-derivative}
&\d\omega\colon\VF_{n+1}(\X,\TT;E)\to\VF(\X,\TT;E)\\
&\d\omega(X_1\,X_{n+1})\=\sum_{k=1}^{n+1}(-)^{k+1}\left[X_k,\omega\left(X_1\,\hat X_k\,X_{n+1}\right)\right]+\\
&\qquad+\sum_{i<j}(-)^{i+j}\omega\left(\left[X_i,X_j\right],X_1\,\hat X_i\,\hat X_j\,X_{n+1}\right)
\end{split}
\end{align}
where $\hat X_k$ indicates the omission of that vector field.

\begin{remark}
\label{remark:exterior-derivative}
Formula~\eqref{equation:exterior-derivative} extends to vector-valued differential forms the \emph{invariant formula} of the exterior derivative of ordinary differential forms of a smooth manifold, as in~\cite[Proposition~14.32]{lie:smooth-manifolds}. In its original formulation for ordinary differential forms, the Lie bracket $[X_k,\omega(\dots)]$ in the first sum, are replaced by the Lie derivative $\mathscr{L}_X\omega$.
\end{remark}

\begin{theorem}
\label{theorem:exterior-differential}
The exterior differential of vector-valued differential forms squares to zero, that is, the following formula
\begin{align*}
&\d\d\omega=0
\end{align*}
holds for every vector-valued differential form $\omega$.
\end{theorem}
\begin{proof}
Despite being a standard result in differential geometry known as the Poincar\'e lemma, in our setting, we lack a local expression for differential forms, since in a general tangent category, there is no well-defined notion of local charts. Unfortunately, this identity is usually proved using the local form of differential forms, which makes it an easy exercise. It is not as trivial when using the invariant formula we have used here.  Therefore, we will give a full proof of this result. We start by introducing some handy notation that will simplify the computations.

\par For starters, we shall denote by $\Altsum_ia_i$, an alternating sum of terms $a_1\,a_n$, that is, the sum $\sum_{i=1}^n(-)^{i+1}a_i$. This notation will be useful when we need to change the meaning of the index $i$, while keeping track of the alternating sign. Secondly, we shall denote by $\vec X$ a tuple of vector fields $X_1\,X_n$ and by $\vec X_{\hat i_1\,\hat i_k}$ the tuple $X_1\,\hat X_{i_1}\,\hat X_{i_k}\,X_n$, where the $i_1$-th-$1_k$-th terms have been removed from $\vec X$. To become comfortable with this notation, we shall start by rewriting the defining formula of the exterior derivative for an arbitrary differential form $\theta$:
\begin{align*}
&\d\theta\left(\vec X\right)=\Altsum_k\left[X_k,\theta\left(\vec X_{\hat k}\right)\right]+\sum_{i<j}\left(-\right)^{i+j}\theta\left(\left[X_i,X_j\right],\vec X_{\hat i,\hat j}\right)
\end{align*}
Therefore, with this notation, $\d\d\omega$ can be written as follows:
\begin{align}
\label{equation:poincare-1}
&\d\d\omega\left(\vec X\right)=\Altsum_k\left[X_k,\d\omega\left(\vec X_{\hat k}\right)\right]+\sum_{i<j}\left(-\right)^{i+j}\d\omega\left(\left[X_i,X_j\right],\vec X_{\hat i,\hat j}\right)
\end{align}
To avoid getting trapped in a long, complicated chain of identities, we split the right-hand side into two pieces. We shall start by unwrapping the first bit:
\begin{align}
\label{equation:poincare-2}
\begin{split}
\Altsum_k\left[X_k,\d\omega\left(\vec X_{\hat k}\right)\right]&=\underbracket{\Altsum_k\left[X_k,\Altsum_{k'\neq k}\left[X_{k'},\omega\left(\vec X_{\hat k,\hat k'}\right)\right]\right]}_{\text{(1)}}+\\
&+\underbracket{\Altsum_k\left[X_k,\sum_{i'<j';i',j'\neq k}\left(-\right)^{i'+j'}\omega\left(\left[X_{i'},X_{j'}\right],\vec X_{\hat i',\hat j',\hat k}\right)\right]}_{\text{(2)}}
\end{split}
\end{align}
Let us have a closer look at the term (1). If we want to express the alternating sum as an actual sum, we need to split the internal sum into two pieces, as follows:
\begin{align*}
(1)=\sum_{k=2}^{n+2}(-)^{k+1}\left[X_k,\sum_{k'=1}^{k-1}(-)^{k'+1}[X_{k'},\omega\left(\vec X_{\hat k',\hat k}\right)\right]-\sum_{k=1}^{n+1}(-)^{k+1}\left[X_k,\sum_{k'=k+1}^{n+2}(-)^{k'+1}[X_{k'},\omega\left(\vec X_{\hat k',\hat k}\right)\right]
\end{align*}
where the negative sign between the two sums comes from shifting the alternating sum in $k'$ to one position, since $k'\neq k$. Now, using the additivity of the Lie bracket, and after renaming the indices $k\mapsto i$ and $k'\mapsto j$, we rearrange the two terms as follows:
\begin{align*}
(1)&=\sum_{j<i}(-)^{i+j}\left[X_i,[X_j,\omega\left(\vec X_{\hat i,\hat j}\right)\right]-\sum_{i<j}(-)^{i+j}\left[X_i,\left[X_j,\omega\left(\vec X_{\hat i,\hat j}\right)\right]\right]
\end{align*}
However, in the first term, we can swap the two indices by virtue of the identity $\sum_{j<i}=\sum_{i=2}^{n+2}\sum_{j=1}^{i-1}=\sum_{j=1}^{n+1}\sum_{i=j+1}^{n+2}$. Therefore, after renaming the indices $i\mapsto j$, $j\mapsto i$, we rewrite the terms as follows:
\begin{align}
\label{equation:poincare-3}
(1)&=\sum_{i<j}(-)^{i+j}\left[X_j,[X_i,\omega\left(\vec X_{\hat i,\hat j}\right)\right]-\sum_{i<j}(-)^{i+j}\left[X_i,[X_j,\omega\left(\vec X_{\hat i,\hat j}\right)\right]
\end{align}
Let us now come back to Equation~\eqref{equation:poincare-1}, and let us take a closer look at the second sum. In particular, since the first argument of $\d\omega$ is $[X_i,X_j]$, it is useful to separate the first term from the rest of the alternating sums that define $\d\omega$, as follows:
\begin{align*}
\sum_{i<j}\left(-\right)^{i+j}\d\omega\left(\left[X_i,X_j\right],\vec X_{\hat i,\hat j}\right)&=
\underbracket{\sum_{i<j}(-)^{i+j}\left[\left[X_i,X_j\right],\omega\left(\vec X_{\hat i,\hat j}\right)\right]}_{\text{(A)}}+\\
&-\underbracket{\sum_{i<j}(-)^{i+j}\Altsum_{k'\neq i,j}\left[X_{k'},\omega\left(\left[X_i,X_j\right],\vec X_{\hat i,\hat j,\hat k'}\right)\right]}_{\text{(B)}}+\\
&+\underbracket{\sum_{i<j}(-)^{i+j}\sum_{j'\neq i,j}(-)^{j'+1}\omega\left(\left[\left[X_i,X_j\right],X_{j'}\right],\vec X_{\hat i,\hat j,\hat j'}\right)}_{\text{(C)}}+\\
&-\underbracket{\sum_{i<j}(-)^{i+j}\sum_{i'<j';i',j'\neq i,j}(-)^{i'+j'}\omega\left(\left[X_{i'},X_{j'}\right],\left[X_i,X_j\right],\vec X_{\hat i,\hat i',\hat j,\hat j'}\right)}_{\text{(D)}}
\end{align*}
Crucially, the terms (B) and (D) carry a negative sign due to the shift in the alternating sum. After simple arrangements, it is not hard to see that the terms (2) of Equation~\eqref{equation:poincare-2} and (B) cancel each other out. Furthermore, the term (C) contains the expressions $[[X_i,X_j],X_{j'}]$, where $i<j$ and $j'\neq i,j$. Let us restrict first to the case when $j'<i$. Then, $j'<i<j$. Therefore, within the sum, for each of such triple of indices $j'<i<j$, there will be the terms $[[X_{j'},X_i],X_j]$, $[[X_j,X_{j'}],X_i]$, and $[[X_i,X_j],X_{j'}]$. However, by the Jacobi identity, the sum of these three terms vanishes. The same exact situation happens when $i<j'<j$ and when $i<j<j'$, and moreover, every term in the sum of (C) is of one of these forms. Therefore, by the Jacobi identity used multiple times, the term (C) vanishes entirely. The term (D) also vanishes by antisymmetry. In fact, within the sum, the pairs of indices $(i,j)$ and $(i',j')$ swap, and therefore, by antisymmetry of the Lie bracket, the whole sum must vanish. The only remaining terms are (1) of Equation~\eqref{equation:poincare-2} and (A). Using Equation~\eqref{equation:poincare-3}, we can write the sum (1)+(A) as follows:
\begin{align*}
&\text{(1)}+\text{(A)}=\sum_{j<i}(-)^{i+j}\left(\left[X_j,\left[X_i,\omega\left(\vec X_{\hat i,\hat j}\right)\right]\right]-\left[X_i,[X_j,\omega\left(\vec X_{\hat i,\hat j}\right)\right]+\left[\left[X_i,X_j\right],\omega\left(\vec X_{\hat i,\hat j}\right)\right]\right)
\end{align*}
Let us call $X\=X_i$, $Y\=X_j$, and $Z\=\omega(\vec X_{\hat i,\hat j})$. Thus, the term in the brackets reads as follows:
\begin{align*}
&[Y,[X,Z]]-[X,[Y,Z]]+[[X,Y],Z]=[Y,[X,Z]]+[Z,[X,Y]]+[Y,[Z,X]]+[[X,Y],Z]=0
\end{align*}
where in the first step we used the Jacobi identity and the latter, the antisymmetry. This proves that $\d\d\omega=0$.
\end{proof}

The vertical connection form $\phi$ of an Ehresmann connection defines a differential form
\begin{align*}
&\phi\colon\VF(\X,\TT;E)\to\VF(\X,\TT;E)            &&\phi(X)\=X_\V=X\phi
\end{align*}
that sends every vector field on the total space $E$ to its vertical component.

\begin{definition}
\label{definition:curvature-tensor}
The \textbf{curvature form} of an Ehresmann connection $(\R,\H)$ on a tangent display map $q\colon E\to M$ is the vector-valued differential form defined by the following formula
\begin{align*}
&\Curv_{(\R,\H)}\colon\VF(\X,\TT;E)\times\VF(\X,\TT;E)\to\VF(\X,\TT;E)\\
&\Curv_{(\R,\H)}(X,Y)\=\d\phi(X_\H,Y_\H)
\end{align*}
where, as usual, $X_\H=X\psi$ and $Y_\H=Y\psi$ denote the horizontal components of the vector fields $X$ and $Y$, respectively.
\end{definition}

The relation between the curvature differential form and the categorical curvature is expressed by the structural equation. To compare this identity with the existing literature, we need to introduce a suitable notion of the wedge product between vector-valued differential forms. To this end, consider two differential forms $\omega$ and $\tau$ of rank $n$ and $m$, respectively. Then, $[\omega\wedge\tau]$ is the differential form of rank $n+m$ defined by the following formula:
\begin{align*}
&[\omega\wedge\tau](X_1\,X_{n+m})\=\frac1{n!m!}\sum_{\sigma}(-)^\sigma\left[\omega\left(X_{\sigma(1)}\,X_{\sigma(n)}\right),\tau\left(X_{\sigma(n+1)}\,X_{\sigma(n+m)}\right)\right]
\end{align*}
The reader may be bothered by the factor $1/{n!m!}$, which is not well-defined in our context. However, this is simply a factor that keeps track of the repetitions that appear when expanding the whole formula, due to the antisymmetry of the Lie bracket and the alternating property of the differential forms. In practice, once the formula is fully expanded, no fractional factor is multiplied to the term. To match the notation used in the diffential geometry literature, we shall use the following identities:
\begin{align*}
&\frac12[\phi\wedge\phi](X,Y)=[X_\V,Y_\V]    &&[\phi\wedge\psi]_\V(X,Y)=[X_\V,Y_\H]_\V+[X_\H,Y_\V]_\V
\end{align*}
Notice that, when the commutative monoid $\VF(\X,\TT;E)$ of vector fields over $E$ has characteristic distinct from $2$, that is, $2X=0$ implies $X=0$, then $1/2[\phi\wedge\phi]$ is the differential form uniquely determined by the equation $2(1/2[\phi\wedge\phi])(X,Y)=2[X_\V,Y_\V]$.

\begin{theorem}[Structural equation]
\label{theorem:structural-equation}
The following identities hold:
\begin{align*}
&\Curv_{(\R,\H)}(X,Y)=-[X_\H,Y_\H]_\V       &&\Curv_{(\R,\H)}=\d\phi+\frac12[\phi\wedge\phi]-\left([\phi\wedge\id_{\T E}]-[\phi\wedge\psi]_\V\right)
\end{align*}
\end{theorem}
\begin{proof}
We shall start by unwrapping the definition of the exterior derivative of the connection form $\phi$:
\begin{align*}
\d\phi(X,Y)&=~[X,Y_\V]-[Y,X_\V]-[X,Y]_\V                                            \Tag{\text{Formula}~\eqref{equation:exterior-derivative}}\\
&=~[X_\V,Y_\V]+[X_\H,Y_\V]+[X_\V,Y_\V]+[X_\V,Y_\H]+\\
&~\qquad-[X_\V,Y_\V]_\V-[X_\V,Y_\H]_\V-[X_\H,Y_\V]_\V-[X_\H,Y_\H]_\V            \Tag{\text{Corollary}~\ref{corollary:decomposition-connection}}\\
&=~[X_\V,Y_\V]+[X_\V,Y_\H]_\H+[X_\H,Y_\V]_\H-[X_\H,Y_\H]_\V     \Tag{[X_\V,Y_\V]\text{ is vertical}}
\end{align*}
When we take $X$ and $Y$ as horizontal vector fields, the only term that survives is $-[X_\H,Y_\H]_\V$. Therefore, the curvature $\Curv_{(\R,\H)}(X,Y)$ is equal to $-[X_\H,Y_\H]_\V$. Let us now use the wedge product of vector-valued differential forms to express the remaining terms. By using the defining formula we obtain:
\begin{align*}
&[\phi\wedge\phi](X,Y)=[X_\V,Y_\V]-[Y_\V,X_\V]=2[X_\V,Y_\V]\\
&[\phi\wedge\psi](X,Y)=[X_\V,Y_\H]-[Y_\V,X_\H]=[X_\V,Y_\H]+[X_\H,Y_\H]
\end{align*}
Using the bilinearity of the wedge product and \textbf{[EC.2]}, we can also write:
\begin{align*}
&[\phi\wedge\psi]=[\phi\wedge\id_{\T E}]-[\phi\wedge\phi]=[\phi\wedge\id_{\T E}]_\V+[\phi\wedge\id_{\T E}]_\H-[\phi\wedge\phi]\\
&\quad=[\phi\wedge\id_{\T E}]_\V+[\phi\wedge\phi]_\H+[\phi\wedge\psi]_\H-[\phi\wedge\phi]
\end{align*}
However, $[\phi\wedge\phi]_\H(X,Y)=2[X_\V,Y_\V]_\H$ is necessarily zero, since the vertical distribution is involutive. Thus, we obtain that:
\begin{align*}
&[X_\V,Y_\H]+[X_\H,Y_\H]=([\phi\wedge\id_{\T E}]_\V+[\phi\wedge\phi]_\H+-[\phi\wedge\phi])(X,Y)
\end{align*}
Putting everything together, we obtain the desired formula.
\end{proof}

\begin{remark}
\label{remark:structural-equation}
In the context of principal connections over principal bundles in the differential geometry literature, the second identity of Theorem~\ref{theorem:structural-equation} is known as \textit{Cartan's structural equation}, and takes the following simplified form:
\begin{align*}
&\Curv_{(\R,\H)}=\d\phi+\frac12[\phi\wedge\phi]
\end{align*}
In particular, using~\cite[Lemma~5.5.5]{hamilton:gauge-theory} and the fact that Lie derivatives of fundamental vector fields vanish, the term $[\phi\wedge\id_{\T E}]-[\phi\wedge\psi]_\V$ disappears.
\end{remark}

\begin{remark}
\label{remark:structural-equation-2}
In the differential geometry literature, the structural equation for the curvature form of an Ehresmann connection can also be expressed in terms of the Fr\"olicher-Nijenhuis bracket $[-,-]^{\mathsf{FN}}$ and takes the form~\cite[Section~17]{kolar:natural-operations-diff-geom}:
\begin{align*}
&\Curv_{(\R,\H)}=\frac12[\phi,\phi]^{\mathsf{FN}}
\end{align*}
We decided not to investigate this operation in this paper and instead express the structural equations in relation to the exterior derivative.
\end{remark}

We can finally now prove a version of the Bianchi identity in the context of Ehresmann connections in a tangent category.

\begin{theorem}[Bianchi identity]
\label{theorem:bianchi-identity}
The following identity holds:
\begin{align*}
&\left(\d\Curv_{(\R,\H)}+[\id_{\T E}\wedge\Curv_{(\R,\H)}]_\H\right)(X_\H,Y_\H,Z_\H)=0
\end{align*}
\end{theorem}
\begin{proof}
From the structural equation and using that the exterior derivative squares to zero, we can reduce the derivative of the curvature form to the following expression
\begin{align*}
&\d\Curv_{(\R,\H)}=\d\left(\d\phi+\theta\right)=\d\theta
\end{align*}
where $\theta(X,Y)=[X_\V,Y_\V]+[X_\V,Y_\H]_\H+[X_\H,Y_\V]_\H$. Notice that, $\theta(X_\H,Y_\H)=0$. Let us compute the exterior derivative of $\theta$ explicitly. From the formula~\eqref{equation:exterior-derivative}:
\begin{align*}
&\d\theta(X,Y,Z)=[X,\theta(Y,Z)]-[Y,\theta(X,Z)]+[Z,\theta(X,Y)]-\theta([X,Y],Z)+\theta([X,Z],Y)-\theta([Y,Z],X)
\end{align*}
By taking $X,Y,Z$ as horizontal vector fields and using that $\theta$ vanishes for horizontal vector fields, we can reduce it to the following:
\begin{align*}
&\d\theta(X_\H,Y_\H,Z_\H)=-\theta([X_\H,Y_\H],Z_\H)+\theta([X_\H,Z_\H],Y_\H)-\theta([Y_\H,Z_\H],X_\H)
\end{align*}
However, for a vector field $X$ and a horizontal vector field $Y_\H$, $\theta(X,Y_\H)=[X_\V,Y_\H]_\H$, since the other terms vanish. Therefore:
\begin{align*}
&\d\theta(X_\H,Y_\H,Z_\H)=-[[X_\H,Y_\H]_\V,Z_\H]+[[X_\H,Z_\H]_\V,Y_\H]-[[Y_\H,Z_\H]_\V,X_\H]
\end{align*}
Let us now compute the wedge product $[\Curv_{(\R,\H)}\wedge\psi]$. By expanding the defining formula (notice that the factor $1/n!m!$ handles the duplication due to the antisymmetry):
\begin{align*}
&[\Curv_{(\R,\H)}\wedge\psi](X,Y,Z)=-[[X_\H,Y_\H]_\V,Z_\H]-[[Z_\H,X_\H]_\V,Y_\H]-[[Y_\H,Z_\H]_\V,X_\H]
\end{align*}
Therefore, we obtain the formula:
\begin{align*}
&\d\theta(X_\H,Y_\H,Z_\H)=[\Curv_{(\R,\H)}\wedge\psi]_\H(X_\H,Y_\H,Z_\H)
\end{align*}
We can also rewrite the right term as follows
\begin{align*}
&[\Curv_{(\R,\H)}\wedge\psi]_\H=[\Curv_{(\R,\H)}\wedge\id_{\T E}]_\H-[\Curv_{(\R,\H)}\wedge\phi]_\H=[\Curv_{(\R,\H)}\wedge\id_{\T E}]_\H
\end{align*}
where the term $[\Curv_{(\R,\H)}\wedge\phi]_\H$ vanishes since $\Curv_{(\R,\H)}(X,Y)$ is vertical and the vertical distribution is involutive. This proves the final formula:
\begin{align*}
&\d\Curv_{(\R,\H)}(X_\H,Y_\H,Z_\H)=[\Curv_{(\R,\H)}\wedge\id_{\T E}]_\H(X_\H,Y_\H,Z_\H)
\end{align*}
The right-hand side can be brought to the left as $[\id_{\T E}\wedge\Curv_{(\R,\H)}]_\H(X_\H,Y_\H,Z_\H)$.
\end{proof}

\begin{remark}
\label{remark:bianchi-identity}
For principal connections in differential geometry, the term $[\id_{\T E}\wedge\Curv_{(\R,\H)}]_\H$ vanishes, leaving $\d\Curv_{(\R,\H)}(X_\H,Y_\H,Z_\H)=0$.
\end{remark}

\subsection{Parallel transport}
\label{subsection:parallel-transport}

One of the most important consequences of the existence of a connection in differential geometry is the resulting parallel transport.  Given a curve in the base space $M$ of a bundle $q: E \to M$ with a connection and a point ``upstairs'' (that is, in $E$), parallel transport allows one to move the curve upstairs, so that it starts at the given point, and stays ``parallel'' relative to the connection.

In differential geometry, the existence of parallel transport involves solving a certain ordinary differential equation.  In tangent category theory, one needs to add additional structure to talk about ordinary differential equations and their solutions; such a structure is called a \emph{curve object}.  

In \cite{cockett:connections}, curve objects are first introduced.  They allow one to talk about (ordinary) differential equations, and they assume that any solution to a differential equation is unique.  They \emph{do not} assume that every differential equation has a (total) solution, as this is not true in smooth manifolds (solutions may only exist locally, but not globally).  The paper also considers \emph{linearly complete} curve objects, which assume that  \emph{linear} differential equations always have (total) solutions (again, this is true in smooth manifolds). 

With these assumptions, \cite[Theorem 5.20]{cockett:connections} shows that every Koszul connection on a differential bundle in a tangent category with a linearly complete curve object gives rise to parallel transport in that bundle. This is achieved by associating to any Koszul connection a linear differential equation, whose (unique) solution gives rise to the required parallel transport.  

However, the situation is slightly different for connections on a (not necessarily differential/vector) bundle.  In this case, the resulting vector field one can build from the connection is not necessarily linear, so a (total) solution to it may not exist.  

In this section, we describe what this means for a connection in a Cartesian tangent category with a curve object.  We show that given a connection, one can build an associated vector field.  We then show that \emph{if} the vector field has a total solution (which may not always be the case), then the resulting solution gives a notion of parallel transport for the connection.  The proofs themselves follow almost identically to their versions in the differential bundle case.  

We leave it to future work to describe tangent categories whose curve objects have partial solutions to vector fields; in this case we anticipate that every connection in such a setting would give rise to a partially-defined parallel transport (as is true in smooth manifolds).    

We begin by recalling dynamical systems and curve objects in a tangent category.  

\begin{definition}
\label{definition:dynamical-system}
A \textbf{dynamical system} on an object $M$ consists of a map $s\colon N\to M$ (which represents the \emph{initial condition} parametrized by $N$) and a vector field $X\colon M \to \T M$.
\end{definition}

We may denote a dynamical system by $(M,s,X)$.

\begin{definition}
\label{definition:curve-object}
A \textbf{curve object} in a Cartesian tangent category is a dynamical system $(C,c_0,c_1)$ where $c_0\colon\*\to C$ which has the following three properties:
\begin{description}
\item[CO.1] \textbf{Uniqueness of solutions}. For any dynamical system $(M,s,X)$, there is at most one map $\gamma\colon C \times N \to M$ such that the following diagrams commute:
\begin{equation*}
\begin{tikzcd}
N & {C \times N} & {T(C \times N)} \\
& M & TM
\arrow["{\< !c_0, \id_N\>}", from=1-1, to=1-2]
\arrow["s"', from=1-1, to=2-2]
\arrow["{c_1 \times z_N}", from=1-2, to=1-3]
\arrow["\gamma", from=1-2, to=2-2]
\arrow["{T(\gamma)}", from=1-3, to=2-3]
\arrow["X"', from=2-2, to=2-3]
\end{tikzcd}
\end{equation*}
When it exists, the map $\gamma$ is called a \textbf{solution}\footnote{For details on how this can be seen as a solution to a differential equation, see \cite[Section 3.2]{cockett:differential-equations}} of the system $(M,s,X)$.

\item[CO.2] \textbf{Commutativity of $c_1$}. $c_1T(c_1)c = c_1T(c_1)$.

\item[CO.3] \textbf{Completeness of $c_1$}. The system $(C,c_0,c_1)$ has a solution.  
\end{description}
\end{definition}

\begin{remark}
\label{remark:parallel-transport-curve-object}
We will not need Conditions \textbf{[CO.1]} and \textbf{[CO.2]} for parallel transport.
\end{remark}

For the rest of this section, we work in a Cartesian tangent category, which is a tangent category with finite products (including the terminal object) preserved by the tangent bundle functor, with a fixed curve object $(C,c_0, c_1)$.  

\begin{lemma}
Suppose that $\H\colon\F q\to\T E$ is a horizontal connection on a tangent display map $q\colon E\to M$, $\gamma\colon C \to M$ is a map, and
\begin{equation*}
\begin{tikzcd}
\gamma^\*(E) & {E} \\
C & {M}
\arrow["\pi_1", from=1-1, to=1-2]
\arrow["\pi_0"', from=1-1, to=2-1]
\arrow["{q}", from=1-2, to=2-2]
\arrow["{\gamma}"', from=2-1, to=2-2]
\end{tikzcd}
\end{equation*}
is a pullback.  Then there is a vector field $\gamma_\H\colon\gamma^\*(E) \to \T (\gamma^\*(E))$ on $\gamma^\*(E)$, defined by
\begin{align*}
&\gamma_\H\= \< \pi_0 c_1, \<\pi_0 c_1 \T(\gamma), \pi_1\>\H\>
\end{align*}
\end{lemma}
\begin{proof}
That this is a well-defined vector field is exactly the same as the first parts of the proof of Theorem 5.20 in \cite{cockett:connections}.
\end{proof}

\begin{theorem}
(Parallel transport) Suppose that $(\R,\H)$ is an Ehresmann connection on a tangent display map $q\colon E\to M$, $\gamma\colon C \to M$ is a map, $e_0\colon\*\to M$ is a map such that $e_0q = c_0\gamma$, and a (necessarily unique) solution to the system $(\gamma^\*(E), \<c_0, e_0\>, \gamma_\H)$ exists.  Then there exists a unique map $\hat{\gamma}\colon C \to E$ such that:
\begin{itemize}
\item \textbf{$\hat{\gamma}$ starts at $e_0$}, that is, $c_0 \hat{\gamma} = e_0$;

\item \textbf{$\hat{\gamma}$ is above $\gamma$}, that is, $\hat{\gamma}q = \gamma$;

\item \textbf{$\hat{\gamma}$ is parallel}, that is, the following diagram commutes:
\begin{equation*}
\begin{tikzcd}
C & {\T C} & {\T E} \\
& E & {\V q}
\arrow["{c_1}", from=1-1, to=1-2]
\arrow["{\hat{\gamma}}"', from=1-1, to=2-2]
\arrow["{\T\hat{\gamma}}", from=1-2, to=1-3]
\arrow["\R", from=1-3, to=2-3]
\arrow["{z_q^\V}"', from=2-2, to=2-3]
\end{tikzcd}
\end{equation*}
\end{itemize}
\end{theorem}
\begin{proof}
Again, most of this follows similarly to the proof of Theorem 5.20 in \cite{cockett:connections}.  The only difference is the proof and statement of the last condition.  The proof is very similar to the proof in the middle of page 883 in that reference, except that
\[ \H\K = \pi_1 q 0_q \]
in that proof is replaced by
\[ \H \R = q^\F z_q^\V. \]
The result then follows after a few simple calculations.  
\end{proof}



\end{document}